\documentclass[12pt,reqno]{NumPDEsArticle}

\usepackage[utf8]{inputenc}
\usepackage[T1]{fontenc}
\usepackage[english]{babel}

\usepackage{NumPDEsMacros}

\usepackage{csquotes}
\usepackage{enumitem}
\usepackage{amssymb}
\usepackage{microtype}
\usepackage{bbm}
\usepackage{bm}
\usepackage{stackengine}
\usepackage{float}
\usepackage{subfig}

\title{New \textsl{a~posteriori} Error Estimates \\ for full-space transmission problems}

\author{Alexander Freiszlinger}
\author{Dirk Pauly}
\author{Dirk Praetorius}
\author{Michael Schomburg}

\address{TU Wien, Institute of Analysis and Scientific Computing, Wiedner Hauptstraße 8-10, 1040 Wien, Austria}
\email{alexander.freiszlinger@tuwien.ac.at \quad \rm (corresponding author)}
\email{dirk.praetorius@tuwien.ac.at}

\address{TU Dresden, Institute of Analysis, Zellescher Weg 12-14, 01069 Dresden, Germany}
\email{dirk.pauly@tu-dresden.de}
\email{michael.schomburg@tu-dresden.de}

\newcommand{\Omegaext}{\Omega^{\mathrm{ext}}}
\newcommand{\gammaint}{\gamma^{\mathrm{int}}} 
\newcommand{\gammaext}{\gamma^{\mathrm{ext}}} 

\newcommand{\uext}{u^{\mathrm{ext}}}

\newcommand{\AAA}{\mathfrak{A}}

\thanks{This research was funded by the Austrian Science Fund (FWF) projects
\href{https://www.fwf.ac.at/en/research-radar/10.55776/F65}{10.55776/F65} (SFB
F65 ``Taming complexity in PDE systems''),
\href{https://www.fwf.ac.at/en/research-radar/10.55776/I6802}{10.55776/I6802}
(international project I6802 ``Functional error estimates for PDEs on unbounded
domains''), and
\href{https://www.fwf.ac.at/en/research-radar/10.55776/P33216}{10.55776/P33216}
(standalone project P33216 ``Computational nonlinear PDEs'').}

\begin{document}
\maketitle

\begin{abstract}
In the present work, we derive functional upper bounds for the potential error arising from finite-element boundary-element coupling formulations for a nonlinear Poisson-type transmission problem. 
The proposed \textsl{a~posteriori} error estimates are independent of the precise discretization scheme and provide guaranteed upper bounds for the potential error. 
The computation of these upper bounds is based on the solutions of local auxiliary finite element problems on patches in the interior domain and in a strip domain along the coupling boundary. 
Numerical experiments illustrate the performance of the proposed error estimation strategy for a related adaptive mesh-refinement strategy.
\end{abstract}

\section{Introduction}

Let $\Omega \subset \R^d$, $d \in \set{2,3}$, be a bounded Lipschitz domain with connected polygonal boundary $\Gamma \coloneqq \partial \Omega$. {Let $\bm{n}_\Gamma$ be the exterior unit normal vector on $\Gamma$.} {Denote} by $\Omegaext \coloneqq \R^d \setminus \overline{\Omega}$ the exterior domain with respect to $\Omega$. 
Given a volume force $f$ and boundary data $(\const{g}{D},\const{g}{N})$, we consider the transmission problem of finding $u$ and $\uext$ such that 
\begin{subequations} \label{eq:nltransmission}
 \begin{align}
  \label{eq:tmint}
  -\div(\AAA \nabla u)  
  = f 
  &\quad \text{in } \Omega, \\
  \label{eq:tmext} 
  -\Delta \uext 
  = 0 
  &\quad \text{in } \Omegaext, \\ 
  \label{eq:tmdir}
  u - \uext 
  = g_D 
  &\quad \text{on } \Gamma, \\ 
  \label{eq:tmneu}
  (\AAA \nabla u - \nabla \uext) \cdot \bm{n}_{\Gamma} 
  = \const{g}{N} 
  &\quad \text{on } \Gamma, \\
  \label{eq:tmrad}
  \uext(x) 
  = \OO(|x|^{-1}) 
  &\quad \text{as } |x| \to \infty,
 \end{align}
\end{subequations}
where $\mathfrak{A} \colon \R^d \rightarrow \R^d$ is a (possibly nonlinear) strongly monotone and Lipschitz continuous operator in the sense of \cite{Zeidler1990}, i.e., 
\begin{subequations} \label{eq:lipell}
\begin{align}
 \label{eq:lip}
 \abs{\mathfrak{A} x - \mathfrak{A} y} 
 &\leq \const{C}{Lip} \abs{x - y} 
 \quad \text{for all } x,y \in \R^d, \\ 
 \label{eq:ell}
 \const{C}{mon}\abs{x - y}^2
 &\leq (\mathfrak{A} x - \mathfrak{A} y) \cdot (x - y)
 \quad \text{for all } x,y \in \R^d. 
\end{align}
\end{subequations}
In 2D, validity of the radiation condition~\eqref{eq:tmrad} requires additional assumptions on the data $(f,\const{g}{N})$, while for general data~\eqref{eq:tmrad} holds with the unphysical radiation condition $\OO(\log \abs{x})$; see~\cite[Section~8]{McLean2000} and Remark~\ref{rem:comp}.

Since problem \eqref{eq:nltransmission} involves the unbounded exterior domain $\Omegaext$, usual solution strategies represent $\uext$ in terms of certain integral operators applied to some density $\phi$ on the boundary. 
While this favors a boundary element approach in the exterior domain $\Omegaext$, the presence of the nonlinearity $\AAA$ recommends the use of finite elements in the interior domain $\Omega$. 
In this framework, different FEM-BEM coupling methods have been proposed, e.g., the Johnson--Nédélec coupling~\cite{Johnson1980}, the symmetric Costabel--Han coupling~\cite{Costabel1987,Han1990}, and the Bielak--MacCamy coupling~\cite{, Bielak1983,Bielak1995}. 
We refer to~\cite{Aurada2012a} for the proof that all these {coupling formulations} admit unique solutions $(u,\phi) \in H^1(\Omega) \times H^{-1/2}(\Gamma)$ {on the continuous level, whereas} discrete well-posedness of the respective conforming Galerkin discretization is only guaranteed in general for the Costabel--Han coupling. Instead, the Johnson--Nédélec coupling and the Bielak--MacCamy coupling require sufficient diffusion in the interior domain to guarantee discrete well-posedness, e.g., $\const{C}{mon} > 1/4$ in~\cite{Aurada2012a} or some relaxed but $\Omega$-dependent condition in~\cite{Ferrari2023}.

Given a conforming approximation $(u_\ell,\phi_\ell) \in H^1(\Omega) \times H^{-1/2}(\Gamma)$ of $(u,\phi)$, where $u_\ell,\phi_\ell$ are related to some {simplicial} mesh $\TT_\ell$ of $\Omega$ and its induced boundary mesh $\TT_\ell^\Gamma {\coloneqq \TT_\ell|_\Gamma}$ on $\Gamma$, {this work aims, for the first time,} to derive and analyze an \textsl{a~posteriori} error estimation strategy based on so-called \textit{functional upper bounds}. 
Unlike available estimators, the derived bounds thus do not rely on how {the approximation} $(u_\ell,\phi_\ell)$ has been computed (e.g., Galerkin {discretization} with/without exact quadrature, { matrix compression of the involved discrete boundary integral operators, and inexact iterative solver}).

\textsl{A~posteriori} error estimation for the symmetric coupling is well-established; see, e.g.,~\cite{Carstensen1995,Carstensen1996,Carstensen1996a,Maischak1997,Mund1999,Brink1999,Gatica1999,Gatica2001,Stephan2005,Leydecker2010}. 
In~\cite{Aurada2012a,Aurada2012b}, \textsl{a~posteriori} estimates have been derived also for the Johnson--Nédélec coupling and the Bielak--MacCamy coupling. 
However, error control in all these works focuses on the error in $H^1(\Omega) \times H^{-1/2}(\Gamma)$, i.e., ${\norm{u - u_\ell}_{H^1(\Omega)} + }\norm{\phi - \phi_\ell}_{H^{-1/2}(\Gamma)}$, and not on the (in practice more {important}) potential error
\begin{equation} \label{eq:err}
 \EE(u_\ell,\phi_\ell)
 \coloneqq \norm{\nabla(u - u_\ell)}_{\mathfrak{A}} + \norm{\nabla(\uext - \uext_\ell)}_{L^2(\Omegaext)},
\end{equation} 
where $\uext_\ell$ is an approximation of $\uext$ extracted from $u_\ell$ and $\phi_\ell$ {and $\norm{\nabla (u - u_\ell)}_{\mathfrak{A}}^2 \coloneqq \dual{\mathfrak{A} \nabla u - \mathfrak{A} \nabla u_\ell}{\nabla (u-u_\ell)}_{L^2(\Omega)}$ denotes the error measure induced by the nonlinear diffusion $\mathfrak{A}$.}
Moreover, up until now, only residual-based error estimators provide unconditional upper bounds (e.g., \cite{Carstensen1995,Carstensen1996,Carstensen1996a,Brink1999,Gatica1999,Gatica2001,Aurada2012a}), while others like the two-level error estimators (e.g., \cite{Mund1999,Maischak1997,Aurada2012b}) rely on a suitable saturation assumption. 
Hence, \textsl{a~posteriori} error estimation strategies based on functional upper bounds constitute a considerable improvement, since there exist neither \textsl{a~posteriori} estimates controlling the more desirable potential error directly, nor estimates which do not depend on the underlying discretization scheme. 
Moreover, for particular situations like linear diffusion, the proposed estimator is, up to data oscillations, constant-free and does not involve unknown generic constants.

Estimators of functional type have been developed in, e.g.,~\cite{Repin2008,Anjam2016} for problems in bounded domains,~\cite{Pauly2009} for problems in unbounded domains, and~\cite{Kurz2019} for the boundary element method. 
Building on these existing works on functional error estimates and combining it with ideas from \textsl{a~posteriori} error estimation by equilibrated fluxes (see, e.g., \cite{Ainsworth1997,Braess2008,Ern2015}) for the FEM part, we derive guaranteed upper bounds for the combined potential error {$\EE(u_\ell,\phi_\ell)$ defined in~\eqref{eq:err}} that, additionally, are constant-free up to data oscillations for linear diffusion operators $\mathfrak{A}$.

We emphasize that our analysis is independent of the precise discretization scheme. 
In fact, the proposed \textit{functional error estimate} provides guaranteed upper bounds for any {conforming approximation $(u_\ell,\phi_\ell)$ and hence any} type of discretization method (e.g., Galerkin methods, collocation approaches, different coupling strategies, etc.). 

In this paper, the functional estimates are realized by local auxiliary problems in $\Omega$ and, in a first essentially auxiliary step (Algorithm~\ref{algo:adap}), in a (bounded) boundary strip $\omega \subset \Omegaext$ with $\Gamma \subset \partial \omega$ (see also~\cite{Kurz2019}), whereas later $\omega \subseteq \Omega$ with $\Gamma \subset \partial \omega$ (Algorithm~\ref{algo:adapint}).
While we rely on an auxiliary flux $\bm{\sigma}_\ell \approx \mathfrak{A}\nabla u_\ell$ computed by solving local discrete problems of dual mixed type in order to control $\norm{\nabla(u - u_\ell)}_{\mathfrak{A}}$, we aim to control $\norm{\nabla(\uext - \uext_\ell)}_{L^2(\Omegaext)}$ by use of local discrete solutions $w_\ell$ for the Laplace equation in a boundary strip domain $\omega$, improving the approach from~\cite{Kurz2019}, where a global discrete problem in $\omega$ was considered to control the BEM error.
This strategy naturally comes with certain technical and computational complications, as the domains of the local auxiliary functions change within the adaptive loop, which impedes the analysis of the discrete solutions $\bm{\sigma}_\ell$ and $w_\ell$. 
However, the behaviour of FEM-solutions on boundary strips was recently addressed in~\cite{Freiszlinger2025} and the generation of volume meshes is a standard problem in the context of the theory of finite elements.

Finally, by use of Poincaré--Steklov operators all computation for solving the local auxiliary problems can be performed on $\Omega$ so that meshing of (parts of) the exterior domain can indeed be avoided.

\subsection*{Outline} The paper is organized as follows. 
In Section~\ref{sec:preliminaries}, we gather the necessary notation as well as useful preliminary results, which include basic properties of boundary integral operators and solvability of certain FEM-BEM coupling methods. 
Section~\ref{sec:aposteriori} is dedicated to the derivation of the functional upper bounds (Theorem~\ref{thm:mainresult}), which are based on the error identity from Lemma~\ref{lem:errid}.
An improved variant for linear problems is stated in Corollary~\ref{cor:corint}, while Corollary~\ref{cor:corext} shows how the meshing of the exterior domain can be avoided.
Furthermore, fully computable and localizable error estimators are constructed. In Section~\ref{sec:adaptive}, we formulate two adaptive FEM-BEM algorithms in the frame of the symmetric Costabel--Han coupling, where Algorithm~\ref{algo:adap} relies on meshing parts of the exterior domain, which is avoided by Algorithm~\ref{algo:adapint}. Lastly, Section~\ref{sec:numerics} presents numerical experiments involving different coupling methods and various geometries, which illustrate the accuracy of the proposed error estimation strategy and underline that the related adaptive algorithms are capable of regaining optimal convergence rates.

\section{Preliminaries} \label{sec:preliminaries}

\subsection{General Notation}

Throughout this paper, let $\Omega$ be a bounded domain in $\R^d$, $d \in \set{2,3}$, with Lipschitz boundary $\Gamma \coloneqq \partial \Omega$ and exterior unit normal $\bm{n}_\Gamma$. 
We suppose that $\Gamma$ is polygonal and connected.
Let $\Omegaext \coloneqq \R^d \setminus \overline{\Omega}$ be the associated exterior domain. We denote by $\abs{\cdot}$, without any ambiguity, the absolute value of a scalar, the Euclidean norm of a vector in $\R^d$, the cardinality of a (finite) set, the $(d-1)$-dimensional Hausdorff measure of a surface (piece), or the $d$-dimensional Lebesgue volume of a set. 
We write $\int_\Xi \cdot \d x$ for integration over a Lebesgue-measurable set $\Xi \subseteq \R^d$ or over a Hausdorff measurable set $\Xi \subseteq \Gamma$. 
We define $L^2(\Xi)$ as the usual Lebesgue space of square-integrable functions on $\Xi$ equipped with the usual inner product $\dual{v}{w}_\Xi \coloneqq \int_\Xi v w \d x$ and norm $\norm{\cdot}_\Xi \coloneqq \dual{\cdot}{\cdot}_\Xi^{1/2}$. 
For vector-valued spaces $[L^2(\Xi)]^n$, $n \in \N$, we will omit the superscript and write simply $L^2(\Xi)$ if it is clear from the context. 
Related discrete quantities or approximations will be indexed by the same subscript, i.e., $(u_\ell,\phi_\ell)$ is an approximation related to the mesh $\TT_\ell$ etc. 
Finally, we abbreviate notation in proofs and write $A \lesssim B$ if there exists a generic constant $C > 0$, which is clear from the context, such that $A \leq CB$. We write $A \simeq B$ if $A \lesssim B$ and $B \lesssim A$.

\subsection{Function spaces}

For Lipschitz domains ${U} \subset \R^d$, we define the usual first-order Sobolev space $H^1({U})$ as the set of all functions $v \in L^2({U})$ with distributional gradient $\nabla v \in L^2({U})$ and equip it with the norm
\begin{equation*}
 \norm{v}_{H^1({U})} 
 \coloneqq (\norm{v}_{{U}}^2 + \norm{\nabla v}_{{U}}^2)^{1/2}.
\end{equation*} 
{With the distributional divergence operator $\div$, we define}
\begin{equation*}
 H(\div,{U}) 
 \coloneqq \set{\bm{\sigma} \in L^2({U})^d \given \div \bm{\sigma} \in L^2({U})}
\end{equation*}
with the corresponding {graph} norm
\begin{equation*}
 \norm{\bm{\sigma}}_{H(\div,{U})} 
 \coloneqq (\norm{\bm{\sigma}}_{{U}}^2 + \norm{\div \bm{\sigma}}_{{U}}^2)^{1/2}
 \quad \text{for all } \bm{\sigma} \in H(\div,{U}).
\end{equation*}
For {boundary pieces $S \subseteq \Gamma$ we define $H^1(S)$ as the space of all functions $\varphi \in L^2(S)$ with distributional surface gradient $\nabla_\Gamma \varphi \in L^2(S)$ and equip it with the norm
\begin{equation*}
  \norm{\varphi}_{H^1({S})} 
  \coloneqq (\norm{\varphi}_{{S}}^2 + \norm{\nabla_\Gamma \varphi}_{{S}}^2)^{1/2}.
\end{equation*}
}We define the $H^{1/2}({S})$ seminorm
\begin{equation*}
 \abs{\varphi}_{H^{1/2}({S})} 
 \coloneqq \Bigl( \int_{{S}} \int_{{S}} \frac{\abs{\varphi(x) - \varphi(y)}^2}{\abs{x-y}^{d}} \d x \d y \Bigr)^{1/2},
\end{equation*}
the norm 
\begin{equation*}
 \norm{\varphi}_{H^{1/2}({S})} 
 \coloneqq (\norm{\varphi}_{{S}}^2 + \abs{\varphi}_{H^{1/2}({S})}^2)^{1/2},
\end{equation*}
and $H^{1/2}({S}) \coloneqq \set{\varphi \in L^2({S}) \given \norm{\varphi}_{H^{1/2}({S})} < \infty}$.
We recall that $H^{1/2}(\Gamma)$ is the trace space of $H^1(\Omega)$. 
For ${S} = \Gamma$, we define its dual space
\begin{equation*}
 H^{-1/2}(\Gamma) 
 \coloneqq (H^{1/2}(\Gamma))^*.
\end{equation*}
The duality bracket with respect to $H^{-1/2}(\Gamma) \times H^{1/2}(\Gamma)$ is formally defined by continuous extension of the usual $L^2(\Gamma)$-inner product, i.e., 
\begin{equation*}
 \dual{\psi}{\varphi}_{H^{-1/2}(\Gamma) \times H^{1/2}(\Gamma)}
 = \dual{\psi}{\varphi}_{\Gamma} 
 \quad \text{for all } \psi \in L^2(\Gamma) \text{ and } \varphi \in H^{1/2}(\Gamma).
\end{equation*}
To abbreviate notation, we will simply write $\dual{\cdot}{\cdot}_\Gamma$ if it is clear from the context.
We note that every $\bm{\sigma} \in H(\div,\Omega)$ admits a well-defined normal trace $\bm{\sigma}|_\Gamma \cdot \bm{n}_\Gamma \in H^{-1/2}(\Gamma)$. 
In fact, given $g \in H^{1/2}(\Gamma)$ and $w \in H^1(\Omega)$ with $w|_\Gamma = g$, there holds
\begin{equation} \label{eq:green}
 \dual{\bm{\sigma}|_\Gamma \cdot \bm{n}_\Gamma}{g}_\Gamma
 = \dual{\bm{\sigma}}{\nabla w}_\Omega + \dual{\div \bm{\sigma}}{w}_\Omega.
\end{equation}
It is well-known that the above definition is independent of the choice of $w$.

Considering the unbounded domain $\Omegaext$, we introduce the weight function $\rho \colon \Omegaext \rightarrow \R$ defined by 
\begin{equation*}
 \rho(x) 
 \coloneqq 
 \begin{cases} 
  (1 + |x|)^{-1}(1 + \log(1 + |x|))^{-1} & \text{for } d = 2, \\ 
  (1 + |x|)^{-1} & \text{for } d = 3. 
 \end{cases}
\end{equation*}
With this, we define
\begin{equation} \label{eq:H1ext}
 H^1_{\rho}(\Omegaext)
 \coloneqq \set{v \colon \Omegaext \rightarrow \R \given \rho v \in L^2(\Omegaext), \nabla v \in L^2(\Omegaext)}
\end{equation}
as well as 
\begin{equation} \label{eq:Hdivext}
 H^1_\rho(\Delta,\Omegaext)
 \coloneqq \set{v \in {H_\rho}^1(\Omegaext) \given \Delta v = 0}
\end{equation}
with the corresponding norms
\begin{equation*}
 \norm{v}_{H^1_\rho(\Delta,\Omegaext)}
 \coloneqq \norm{v}_{H^1_\rho(\Omegaext)}
 \coloneqq (\norm{\rho v}_{\Omegaext}^2 + \norm{\nabla v}_{\Omegaext}^2)^{1/2}.
\end{equation*}
Finally, we define the solution space for {the transmission} problem \eqref{eq:nltransmission} as
\begin{equation*}
 \HH 
 \coloneqq \set{v \colon \R^d \setminus \Gamma \to \R \given v|_\Omega \in H^1(\Omega) \text{ and } v|_{\Omegaext} \in H^1_\rho(\Delta,\Omegaext)}
\end{equation*}
with the associated norm
\begin{equation*}
 \norm{v}_{\HH}
 \coloneqq \norm{v|_\Omega}_{H^1(\Omega)} + \norm{v|_{\Omegaext}}_{H^1_\rho(\Delta,\Omegaext)}.
\end{equation*}

\subsection{Boundary integral operators}

We define the usual interior and exterior trace operators, i.e.,
\begin{equation*}
 \begin{split}
  \gammaint_0 &\colon H^1(\Omega) \to H^{1/2}(\Gamma), 
  \qquad \, {\gammaint_0 v \coloneqq v|_\Gamma}, \\
  \gammaext_0 &\colon H^1_\rho(\Omegaext) \to H^{1/2}(\Gamma),
  \quad {\gammaext_0 v \coloneqq v|_\Gamma},
 \end{split}
\end{equation*}
as well as the exterior conormal derivative
\begin{equation*}
 \begin{split}
  \gammaext_1 &\colon H^1_\rho(\Delta,\Omegaext) \to H^{-1/2}(\Gamma)
 \end{split}
\end{equation*}
defined via Green's theorem{$\colon$} For $g \in H^{1/2}(\Gamma)$ and $w \in H^1_\rho(\Omegaext)$ with $\gammaext_0 w = g$, {we set}
\begin{equation} \label{eq:conormaldefext}
 \dual{\gammaext_1 {v}}{g}_\Gamma
 \coloneqq - \dual{\nabla {v}}{\nabla w}_\Omega.
\end{equation}
In particular, there holds
\begin{equation} \label{eq:harmonicnorm}
 \norm{\nabla v}_{L^2(\Omegaext)}^2 
 = - \dual{\gammaext_1 v}{\gammaext_0 v}_\Gamma.
\end{equation}
We note that for sufficiently smooth functions $v$, the conormal derivative coincides with the usual normal derivative on the boundary, i.e., $\gammaext_1 {v} =  -\nabla {v} \cdot \bm{n}_\Gamma$. 
We use the convention that for $v \in \HH$, the traces $\gammaint_0 v$ and $\gammaext_0 v$ as well as the conormal derivative $\gammaext_1 v$ are defined by restriction to the respective domains, e.g., ${\gammaext_0} v \coloneqq {\gammaext_0} (v|_{\Omegaext})$, {and note that $\gammaint_0 v \neq \gammaext_0 v$ in general.}

Let $G$ be the fundamental solution of the Laplace equation, i.e., 
\begin{equation*}
 G(x) 
 \coloneqq 
 \begin{cases} 
  -\frac{1}{2\pi}\log(|x|) & \text{for } d = 2, \\ 
  \frac{1}{4\pi|x|} & \text{for } d = 3.
 \end{cases}
\end{equation*}
Define the single-layer potential
\begin{equation*}
 \widetilde{V} \colon H^{-1/2}(\Gamma) \to \HH, 
 \quad (\widetilde{V} \phi)(x) 
 \coloneqq \int_\Gamma G(x-y) \phi(y) \d y
\end{equation*}
as well as the double-layer potential
\begin{equation*}
 \widetilde{K} \colon H^{1/2}(\Gamma) \to \HH, 
 \quad (\widetilde{K} g)(x) 
 \coloneqq \int_\Gamma (\nabla_y G(x-y) \cdot \bm{n}_\Gamma) g(y) \d y.
\end{equation*}
Furthermore, define the usual boundary integral operators used in the context of the boundary element method {via}
\begin{equation} \label{eq:mapprop}
 \begin{split}
  V \colon H^{- 1/2}(\Gamma) \rightarrow H^{1/2}(\Gamma), 
  \quad (V \phi)(x) 
  &\coloneqq \int_{\Gamma} G(x - y)\phi(y) \, \d y, \\
  W \colon H^{1/2}(\Gamma) \rightarrow H^{1/2}(\Gamma), 
  \quad (Wg)(x) 
  &\coloneqq - \nabla_x\Bigl( \int_\Gamma (\nabla_y G(x - y) \cdot \bm{n}_\Gamma) g(y) \d y\Bigr) \cdot \bm{n}_\Gamma, \\
  K \colon H^{1/2}(\Gamma) \rightarrow H^{1/2}(\Gamma), 
  \quad (Kg)(x) 
  &\coloneqq \int_\Gamma (\nabla_y G(x - y) \cdot \bm{n}_\Gamma) g(y) \d y, \\
  K' \colon H^{1/2}(\Gamma) \rightarrow H^{1/2}(\Gamma), 
  \quad (K' \phi)(x) 
  &\coloneqq \nabla_x \Bigl(\int_\Gamma G(x - y) \phi(y) \d y\Bigr) \cdot \bm{n}_\Gamma,
 \end{split}
\end{equation}
{where the integral representations are understood in the formal sense; see, e.g.,~\cite[Section~7]{McLean2000} for details.}
We stress that the operators above are {well-defined, linear, and} continuous {between the respective spaces}.
Moreover, $V$ is elliptic provided $\mathrm{diam}(\Omega) < 1$ for $d=2$; see \cite[Theorem~7.6]{McLean2000}.
We have the {following} elementary relations between the integral operators$\colon$
\begin{equation} \label{eq:BIErel}
 \begin{split}
  &\gamma_0^{\mathrm{int}} \widetilde{V}
  {= V} 
  = \gamma_0^{\mathrm{ext}} \widetilde{V}, 
  \quad \gamma_1^{\mathrm{ext}} \widetilde{K} 
  = -W, \\
  \gamma_0^{\mathrm{int}} \widetilde{K} 
  = K &- 1/2, 
  \quad \gamma_0^{\mathrm{ext}} \widetilde{K} 
  = K + 1/2,
  \quad \gamma_1^{\mathrm{ext}} \widetilde{V} 
  = K' - 1/2;
 \end{split}
\end{equation}
see, e.g., \cite[Section~6]{McLean2000}. 
Lastly, we introduce the interior and exterior {Poincaré--Steklov} operators by
\begin{equation} \label{eq:PS}
 \begin{split}
  S^{\mathrm{int}} \coloneqq V^{-1}(K + 1/2), \\
  S^{\mathrm{ext}} \coloneqq V^{-1}(K - 1/2),
 \end{split}
\end{equation}
respectively. It is well-known that for $w \in H^1(\Omega)$ with $\Delta w = 0$ in $\Omega$, there holds $S^{\mathrm{int}}(w|_\Gamma) = \nabla w|_\Gamma \cdot \bm{n}_\Gamma$, {i.e., $S^{\mathrm{int}}$ is the Dirichlet-to-Neumann map for the interior problem.} 
$S^{\mathrm{ext}}$ has an analogous property in $\Omegaext$; see, e.g., \cite[Section~7]{McLean2000}. 
Moreover, $S^{\mathrm{int}}$ and $-S^{\mathrm{ext}}$ are continuous and elliptic, i.e., there exist constants {$\const{c}{PS}^{\mathrm{int}},\const{C}{PS}^{\mathrm{int}}, \const{c}{PS}^{\mathrm{ext}},\const{C}{PS}^{\mathrm{ext}} > 0$} depending only on $\Gamma$ such that 
\begin{equation} \label{eq:PSprop}
 \begin{split}
  {\const{c}{PS}^{\mathrm{int}}} \norm{g}_{H^{1/2}(\Gamma)}^2 
  &\leq \dual{S^{\mathrm{int}} g}{g}_\Gamma 
  \leq {\const{C}{PS}^{\mathrm{int}}} \norm{g}_{H^{1/2}(\Gamma)}^2, \\
  {\const{c}{PS}^{\mathrm{ext}}} \norm{g}_{H^{1/2}(\Gamma)}^2 
  &\leq -\dual{S^{\mathrm{ext}} g}{g}_\Gamma 
  \leq {\const{C}{PS}^{\mathrm{ext}}} \norm{g}_{H^{1/2}(\Gamma)}^2
  \quad \text{for all } g \in H^{1/2}(\Gamma);
 \end{split}
\end{equation}
see, e.g., \cite[Lemma~2.1]{Steinbach2008}.

\subsection{Meshes and discrete spaces}

Let $\TT_\ell$ be a conforming triangulation of $\Omega$, i.e., a partition of $\overline{\Omega}$ into compact triangles ($d = 2$) or tetrahedra ($d=3$), such that the intersection of two different elements is either empty, a common vertex, a common edge, or a common face (for $d = 3$). 
Throughout this paper, suppose that all considered meshes are obtained by finitely many steps of newest-vertex bisection of some conforming initial mesh $\TT_0$ of $\Omega$.
We refer to~\cite{Stevenson2008,Karkulik2012,Diening2025} for comprehensive treatments of the newest-vertex bisection routine.
Clearly, $\TT_\ell$ induces a conforming boundary mesh 
\begin{equation*}
 \TT_\ell^\Gamma 
 \coloneqq \TT_\ell|_\Gamma 
 \coloneqq \set{T \cap \Gamma \given T \in \TT_\ell \text{ with } \abs{T \cap \Gamma} > 0}.
\end{equation*}
It is well-known that newest-vertex bisection guarantees uniform $\kappa$-shape regularity of $\TT_\ell$ and $\TT_\ell^\Gamma$, i.e., there exists $\kappa > 0$ depending only on $\TT_0$ such that 
\begin{equation*}
 \sup_{T \in \TT_\ell} \frac{\mathrm{diam}(T)}{\abs{T}^{1/d}} 
 + \sup_{F \in \TT_\ell^\Gamma} \frac{\mathrm{diam}(F)}{\abs{F}^{1/(d-1)}} 
 + \sup_{\substack{F,F' \in \TT_\ell^\Gamma \\ F \cap F' \neq \emptyset}} \frac{\mathrm{diam}(F)}{\mathrm{diam}(F')}
 \leq \kappa 
 \quad \text{for all } \TT_\ell.
\end{equation*}
We define $\NN_\ell$ as the set of all vertices $z$ of $\TT_\ell$ and $\NN_\ell^\Gamma$ as the set of all vertices of $\TT_\ell^\Gamma$. 
For some set $U \subseteq \overline{\Omega}$ we define the patch of $U$ as $\TT_\ell[U] \coloneqq \set{T \in \TT_\ell \given T \cap U \neq \emptyset}$ and, recursively, the $k$-patch of $U$ as
\begin{equation*}
 \TT_\ell^{[k]}[U] 
 \coloneqq 
 \begin{cases}
  \TT_\ell[U] & \text{for } k = 1, \\
  \TT_\ell[\TT_\ell^{[{k-1}]}[U]] & \text{for } k > 1.
 \end{cases}
\end{equation*}
Additionally, we define $\Omega_\ell^{[k]}[U] \coloneqq \displaystyle{\bigcup \TT_\ell^{[k]}[U]} \subseteq \overline{\Omega}$. 
If $U = \set{x}$ is a singleton, we {abbreviate} $\TT_\ell^{[k]}[x] \coloneqq \TT_\ell^{[k]}[\set{x}]$ and $\Omega_\ell^{[k]}[x] \coloneqq \Omega_\ell^{[k]}[\set{x}]$. 
Patches of sets {$S \subseteq \Gamma$} with respect to the boundary mesh $\TT_\ell^\Gamma$ on $\Gamma$ are defined accordingly, {e.g., $\TT_\ell^\Gamma[S] \coloneqq \set{F \in \TT_\ell^\Gamma \given F \cap S \neq \emptyset}$, $\Gamma_\ell^{[k]}[S] \coloneqq \bigcup \TT_\ell^{\Gamma,[k]}[S]$ etc.}

Lastly, we define the piecewise constant mesh-width function $h_\ell \in L^\infty(\Omega)$ by $h_\ell|_T \coloneqq h_T \coloneqq \mathrm{diam}(T)$ for all $T \in \TT_\ell$. The quantities $h_E$ and $h_F$ for edges $E$ and facets $F$, respectively, are defined accordingly.

For $r \in \N_0$ and a set $U \subseteq \R^d${,} we define $\P^r(U)$ as the space of all polynomials of degree at most $r$ on $U$. 
Moreover, we define the usual FEM spaces of degree $p \geq 1$ with respect to $\TT_\ell$ and $\TT_\ell^\Gamma$ by
\begin{equation*}
 \begin{split}
  \SS^p(\TT_\ell) 
  &\coloneqq \set{v \in H^1(\Omega) \given v|_T \in \P^p(T) \text{ for all } T \in \TT_\ell}, \\
  \SS^p(\TT_\ell^\Gamma) 
  &\coloneqq \set{v \in H^1(\Gamma) \given v|_F \in \P^p(F) \text{ for all } F \in \TT_\ell^\Gamma},
 \end{split}
\end{equation*}
as well as
\begin{equation*}
 \SS^p_0(\TT_\ell) 
 \coloneqq \set{v \in \SS^p(\TT_\ell) \given \gamma_0^{\mathrm{int}} v = 0}.
\end{equation*}
Additionally, we introduce the hat functions $(\zeta_\ell^z)_{z \in \NN_\ell} \subset \SS^1(\TT_\ell)$, which are defined by 
\begin{equation} \label{eq:unity}
 \zeta_\ell^z(z') 
 \coloneqq 
 \begin{cases}
  1 & \text{if } z' = z, \\
  0 & \text{if } z' \neq z.
 \end{cases}
\end{equation} 
We note that the hat functions satisfy
\begin{equation} \label{eq:hat}
 \sum_{z \in \NN_\ell \cap T} \zeta_\ell^z|_T 
 = 1 
 \quad \text{for all } T \in \TT_\ell.
\end{equation}
An analogous property holds for all facets $F$ and edges $E$.
We further define the space of (in general discontinuous) piecewise polynomials of degree $q \geq 0$ by
\begin{equation*}
 \begin{split}
  \PP^q(\TT_\ell) 
  &\coloneqq \set{v \in L^2(\Omega) \given v|_T \in \P^q(T) \text{ for all } T \in \TT_\ell}, \\
  \PP^q(\TT_\ell^\Gamma) 
  &\coloneqq \set{v \in L^2(\Gamma) \given v|_F \in \P^q(F) \text{ for all } F \in \TT_\ell^\Gamma},
 \end{split}
\end{equation*}
and the $H(\div,\Omega)$-conforming Brezzi--Douglas--Marini (BDM) space by
\begin{equation*}
 \BB\DD\MM^q(\TT_\ell) 
 \coloneqq \set{\bm{\sigma} \in H(\div,\Omega) \given \bm{\sigma}|_T \in \P^q(T)^d \text{ for all } T \in \TT_\ell}
\end{equation*} 
as well as its subspace with zero normal-trace on the boundary by
\begin{equation*}
 \BB\DD\MM_0^q(\TT_\ell) 
 \coloneqq \set{\bm{\sigma} \in \BB\DD\MM^q(\TT_\ell) \given \bm{\sigma} \cdot \bm{n}_\Gamma = 0 \text{ on } \Gamma}.
\end{equation*}
Discrete spaces with respect to a subset of $\TT_\ell$ are defined accordingly.

Lastly, we define the $L^2$-orthogonal projections
\begin{equation} \label{eq:P0proj}
 \begin{split}
  Q_{\ell,q}^\Omega &\colon L^2(\Omega) \to \PP^q(\TT_\ell), \\
  Q_{\ell,q}^\Gamma &\colon L^2(\Gamma) \to \PP^q(\TT_\ell^\Gamma)
 \end{split}
\end{equation}
and note that also $Q_{\ell,q}^\Omega|_T \colon L^2(T) \to \P^q(T)$ and $Q_{\ell,q}^\Gamma|_F \colon L^2(F) \to \P^q(F)$ are orthogonal projections for all $T \in \TT_\ell$ and $F \in \TT_\ell^\Gamma$.
In the case $q=0$, we will write $Q_\ell^\Omega \coloneqq Q_{\ell,0}^\Omega$ and $Q_\ell^\Gamma \coloneqq Q_{\ell,0}^\Gamma$. 
It is well known that $Q_\ell^\Omega$ and $Q_\ell^\Gamma$ satisfy the following approximation properties$\colon$
\begin{equation} \label{eq:P0approx}
 \begin{split}
  \norm{(1 - Q_\ell^\Omega)v}_{L^2(T)} 
  &\leq \frac{h_T}{\pi} \norm{\nabla v}_{L^2(T)} 
  \quad \text{for all } T \in \TT_\ell \text{ and } v \in H^1(T), \\
  \norm{(1 - Q_\ell^\Gamma)\gammaint_0v}_{L^2(F)} 
  &\leq  {\const{C}{N}} h_F^{1/2} \norm{\nabla v}_{L^2(T_F)} 
  \quad \text{for all } F \in \TT_\ell^\Gamma \text{ and } v \in H^1(T_F),
 \end{split}
\end{equation}
where $T_F \in \TT_\ell$ is the unique element such that $F \subset T_F$ and 
\begin{equation*}
{\const{C}{N}  \coloneqq C_d \Bigl(\frac{h_{T_F}^2 \abs{F}}{h_F \abs{T_F}}\Bigr)^{1/2} \leq \frac{C_d \kappa^{d/2}}{d-1}
\quad \text{with } C_2 \approx 0.88152 \text{ and } C_3 \approx 1.96091;}
\end{equation*}
see, e.g.,~\cite{Payne1960,Nicaise2005}.

\subsection{The symmetric Costabel--Han coupling}

We consider the transmission problem \eqref{eq:nltransmission} with $f \in L^2(\Omega)$, $\const{g}{D} \in H^{1/2}(\Gamma)$, and $\const{g}{N} \in H^{-1/2}(\Gamma)$. 
We consider the symmetric Costabel--Han coupling~\cite{Costabel1987,Han1990}.
{The latter employs a so-called} direct ansatz for the exterior solution $\uext$ in terms of {the double-layer and single-layer potential}, i.e., 
\begin{equation} \label{eq:context}
 \uext 
 {= \widetilde{K}(\gammaext_0 \uext) - \widetilde{V}(\gammaext_1 \uext)}
 \eqreff{eq:tmdir}{=} \widetilde{K}(\gammaint_0 u - \const{g}{D}) - \widetilde{V}\phi
\end{equation}
for some unknown density $\phi {\coloneqq \gammaext_1 \uext} \in H^{-1/2}(\Gamma)$.
{Applying $\gammaext_0$ to $\uext$ (see~\eqref{eq:BIErel}) yields
\begin{subequations}
\begin{equation} \label{eq:dirderiv}
 (1/2 - K)\gammaint_0 u + V\phi = (1/2 - K)\const{g}{D}.
\end{equation}
Applying $\gammaext_1$ to $\uext$ (see~\eqref{eq:BIErel}) yields
\begin{equation} \label{eq:neuderiv}
 W \gammaint_0 u + (K' - 1/2)\phi  
 = W \const{g}{D}.
\end{equation}
Using~\eqref{eq:tmint},~\eqref{eq:tmneu}, and~\eqref{eq:neuderiv}, integration by parts yields
\begin{equation} 
 \begin{split}
  &\dual{\mathfrak{A} \nabla u}{\nabla v}_\Omega + \dual{W \gammaint_0 u + (K' - 1/2)\phi}{\gammaint_0 v}_\Gamma \\
  &\qquad = \dual{f}{v}_\Omega + \dual{\const{g}{N} + W \const{g}{D}}{\gammaint_0 v}_\Gamma
 \end{split}
\end{equation}
for all $v \in H^1(\Omega)$. 
\end{subequations}
Combining this with~\eqref{eq:dirderiv} and setting}
\begin{equation} \label{eq:lhs}
 \begin{split}
  b\bigl((u,\phi),(v,\psi)\bigr) 
  &\coloneqq \dual{\mathfrak{A} \nabla u}{\nabla v}_\Omega +{\dual{W \gammaint_0 u + (K' - 1/2)\phi}{\gammaint_0 v}_\Gamma} \\
  &\qquad + {\dual{\psi}{(1/2 - K)\gammaint_0 u + V\phi}}_\Gamma
 \end{split}
\end{equation}
and
\begin{equation} \label{eq:rhs}
 F(v,\psi) 
 \coloneqq \dual{f}{v}_\Omega + {\dual{\const{g}{N} + W \const{g}{D}}{\gammaint_0 v}_\Gamma + \dual{\psi}{(1/2 - K)\const{g}{D}}}_\Gamma
\end{equation}
for all $(u,\phi),(v,\psi) \in H^1(\Omega) \times H^{-1/2}(\Gamma)$, the {Costabel--Han coupling} formulation of~\eqref{eq:nltransmission} reads {as follows}$\colon$ 
Find $(u,\phi) \in H^1(\Omega) \times H^{-1/2}(\Gamma)$ such that
\begin{equation} \label{eq:symcoupbil}
 b\bigl((u,\phi),(v,\psi)\bigr) 
 = F(v,\psi) 
 \quad \text{for all } (v,\psi) \in H^1(\Omega) \times H^{-1/2}(\Gamma).
\end{equation}
Under the foregoing assumptions,~\eqref{eq:symcoupbil} admits a unique solution $(u,\phi) \in H^1(\Omega) \times H^{-1/2}(\Gamma)$; see, e.g.,~\cite{Aurada2012a}.

\begin{remark} \label{rem:comp}
The solution $(u,\uext)$ of the symmetric coupling~\eqref{eq:symcoupbil} also solves~\eqref{eq:tmint}--\eqref{eq:tmneu} in the weak sense.
While for $d=3$, the properties of the layer potentials ensure that $\uext$ satisfies the radiation condition~\eqref{eq:tmrad}, there holds $\abs{\uext(x)} = \OO(\log|x|)$ as $|x| \to \infty$ for $d=2$. 
However, provided
\begin{equation*}
 \dual{f}{1}_\Omega + \dual{\const{g}{N}}{1}_\Gamma 
 = 0,
\end{equation*}
the exterior solution $\uext$ satisfies the radiation condition~\eqref{eq:tmrad} also for $d=2$; see, e.g.,~\cite[Section~8]{McLean2000}.
\end{remark}

\subsection{Galerkin FEM-BEM coupling}

Given a triangulation $\TT_\ell$ of $\Omega$ and its induced boundary mesh $\TT_\ell^\Gamma$, the Galerkin FEM-BEM coupling method seeks approximations $(u_\ell,\phi_\ell) \in \SS^p(\TT_\ell) \times \PP^q(\TT_\ell^\Gamma)$, which solve the discrete problem
\begin{equation} \label{eq:discprob}
 b\bigl((u_\ell,\phi_\ell),(v_\ell,\psi_\ell)\bigr) 
 = F(v_\ell,\psi_\ell) 
 \quad \text{for all } (v_\ell,\psi_\ell) \in \SS^p(\TT_\ell) \times \PP^q(\TT_\ell^\Gamma)
\end{equation}
{associated with~\eqref{eq:symcoupbil}}.
By the same arguments as for the continuous case,~\cite{Aurada2012a} proved that~\eqref{eq:discprob} admits a unique solution $(u_\ell,\phi_\ell) \in \SS^p(\TT_\ell) \times \PP^q(\TT_\ell^\Gamma)$.
{While immaterial for the \textsl{a~posteriori} error estimate presented in the following section, we solve~\eqref{eq:discprob} in the numerical experiments of Section~\ref{sec:numerics}.}

\section{A posteriori error estimation} \label{sec:aposteriori}

In this section, we introduce functional upper bounds for the potential error in the interior as well as in the exterior domain. 
Throughout the remainder of this paper, we assume that $g_D \in H^1(\Gamma)$ and $g_N \in L^2(\Gamma)$.
Let $u_\ell \in H^1(\Omega)$ with $\gamma_0^{\mathrm{int}} u_\ell \in H^1(\Gamma)$ and $\phi_\ell \in L^2(\Gamma)$ be arbitrary conforming approximations $u_\ell \approx u$ and $\phi_\ell \approx \phi$. 
Additionally, we introduce the approximation
\begin{equation} \label{eq:discext}
 \uext_\ell 
 \coloneqq \widetilde{K}(\gammaint_0 u_\ell - \const{g}{D}) - \widetilde{V}\phi_\ell \approx u^{\mathrm{ext}}.
\end{equation}
We introduce the interior error quantity
\begin{equation*}
 \norm{\nabla u - \nabla u_\ell}_{\AAA}^2 
 \coloneqq \dual{\AAA \nabla u - \AAA \nabla u_\ell}{\nabla (u - u_\ell)}_{\Omega}
 \eqreff*{eq:lipell}{\simeq} \norm{\nabla u - \nabla u_\ell}_{\Omega}^2.
\end{equation*}
The full potential error is then quantified by 
\begin{equation} \label{eq:fullerror}
 \EE(u_\ell,\phi_\ell) 
 \coloneqq (\norm{\nabla u - \nabla u_\ell}_{\AAA}^2 + \norm{\nabla u^{\mathrm{ext}} - \nabla \uext_\ell}_{\Omegaext}^2)^{1/2}.
\end{equation}
We emphasize that, even for $u_\ell$ and $\phi_\ell$ being discrete functions, $\uext_\ell$ is not a discrete function and hence lacks basic properties of discrete functions like certain inverse estimates.
However, due to the properties of the layer potentials, we have that $\uext_\ell \in H^1_\rho(\Delta,\Omegaext)$, i.e., $\Delta \uext_\ell = 0$. 
Functional upper bounds are therefore highly suitable for controlling~\eqref{eq:fullerror}, as they do not require any \textsl{a~priori} knowledge of the approximations, but build only on the harmonicity of $\uext_\ell$.

\subsection{Error identity}

The main ingredient of the proposed upper bound is the following error identity involving the bilinear form $b(\cdot,\cdot)$ emerging from the symmetric coupling~\eqref{eq:symcoupbil}. 
The proof uses only Green's formula for harmonic functions in exterior domains~\eqref{eq:harmonicnorm} and properties of the layer potentials~\eqref{eq:BIErel}.
\begin{lemma} \label{lem:errid}
{Given any conforming approximation $(u_\ell,\phi_\ell) \in H^1(\Omega) \times H^{-1/2}(\Gamma)$, let}
\begin{equation} \label{eq:auxapprox}
 \begin{split}
  \widetilde{u}_\ell 
  \coloneqq u_\ell + \abs{\Omega}^{-1} \dual{u - u_\ell}{1}_\Omega 
  \quad \text{and} \quad 
  {\widetilde{u}_\ell^{\mathrm{ext}}}
  \coloneqq \widetilde{K}(\gammaint_0 \widetilde{u}_\ell - \const{g}{D}) - \widetilde{V}\phi_\ell.
 \end{split}
\end{equation}
{Then, there holds the error identity}
\begin{equation} \label{eq:errid}
 \begin{split}
  \EE(u_\ell,\phi_\ell)^2 
  = F\bigl(u - \widetilde{u}_\ell, \gammaext_1(u^{\mathrm{ext}} - \widetilde{u}_\ell^{\mathrm{ext}})\bigr) - b\bigl( (\widetilde{u}_\ell,\phi_\ell),(u - \widetilde{u}_\ell,\gammaext_1(u^{\mathrm{ext}} - \widetilde{u}_\ell^{\mathrm{ext}}))\bigr).
 \end{split}
\end{equation}
\end{lemma}
\begin{proof}
The proof is split into two steps.

\textbf{Step 1.} We first prove that 
\begin{equation} \label{eq:auxerrid}
 \EE(u_\ell,\phi_\ell)^2 
  = F\bigl(u - u_\ell, \gammaext_1(u^{\mathrm{ext}} - u_\ell^{\mathrm{ext}})\bigr) - b\bigl( (u_\ell,\phi_\ell),(u - u_\ell,\gammaext_1(u^{\mathrm{ext}} - u_\ell^{\mathrm{ext}}))\bigr),
\end{equation}
{which proves~\eqref{eq:errid} if $(u_\ell,u_\ell^{\mathrm{ext}}) = (\widetilde{u}_\ell,\widetilde{u}_\ell^{\mathrm{ext}})$.}
Employing \eqref{eq:harmonicnorm} yields
\begin{equation} \label{eq:extnormesti}
 \begin{split}
  \norm{\nabla u^{\mathrm{ext}} - \nabla \uext_\ell}_{\Omegaext}^2 
  &\eqreff*{eq:harmonicnorm}{=} -\dual{\gamma_1^{\mathrm{ext}} (u^{\mathrm{ext}} - \uext_\ell)}{\gamma_0^{\mathrm{ext}}(u^{\mathrm{ext}} - \uext_\ell)}_{\Gamma} \\
  &= \dual{\gamma_1^{\mathrm{ext}} (u^{\mathrm{ext}} - \uext_\ell)}{\gamma_0^{\mathrm{int}}(u - u_\ell) - \gamma_0^{\mathrm{ext}}(u^{\mathrm{ext}} - \uext_\ell)}_{\Gamma} \\
  &\qquad - \dual{\gamma_1^{\mathrm{ext}} (u^{\mathrm{ext}} - \uext_\ell)}{\gamma_0^{\mathrm{int}} (u - u_\ell)}_\Gamma.
 \end{split}
\end{equation}
The representation formulas~\eqref{eq:context} {and}~\eqref{eq:discext} yield
\begin{equation*}
 \uext - \uext_\ell 
 = \widetilde{K}\gammaint_0(u - u_\ell) - \widetilde{V}(\phi - \phi_\ell).
\end{equation*}
Together with the integral operator identities~\eqref{eq:BIErel}, this yields
\begin{equation*}
 \begin{split}
  \gamma_0^{\mathrm{int}}(u - u_\ell) - \gamma_0^{\mathrm{ext}}(u^{\mathrm{ext}} - u^{\mathrm{ext}}_\ell) 
  &= \gammaint_0(u - u_\ell) - \gammaext_0[\widetilde{K}\gammaint_0(u - u_\ell) - \widetilde{V}(\phi - \phi_\ell)] \\
  &\eqreff*{eq:BIErel}{=} \gammaint_0(u - u_\ell) - (K + 1/2)\gammaint_0(u - u_\ell) + V(\phi - \phi_\ell) \\
  &= (1/2 - K) \gammaint_0(u - u_\ell) + V(\phi - \phi_\ell)
 \end{split}
\end{equation*}
and
\begin{equation*}
 \begin{split}
  \gamma_1^{\mathrm{ext}}(u^{\mathrm{ext}} - u^{\mathrm{ext}}_\ell) 
  &= \gammaext_1[\widetilde{K}\gammaint_0(u - u_\ell) - \widetilde{V}(\phi - \phi_\ell)] \\
  &\eqreff*{eq:BIErel}{=} -W\gamma_0^{\mathrm{int}}(u - u_\ell) - (K' - 1/2)(\phi - \phi_\ell).
 \end{split}
\end{equation*}
Together with \eqref{eq:extnormesti}, this leads to
\begin{equation*}
 \begin{split}
  \norm{\nabla u^{\mathrm{ext}} - \nabla \uext_\ell}_{\Omegaext}^2 
  &= \dual{\gamma_1^{\mathrm{ext}} (u^{\mathrm{ext}} - \uext_\ell)}{(1/2 - K)\gammaint_0(u - u_\ell) + V(\phi - \phi_\ell)}_{\Gamma} \\
  &\qquad + \dual{(K' - 1/2)(\phi - \phi_\ell) + W\gamma_0^{\mathrm{int}}(u - u_\ell)}{\gamma_0^{\mathrm{int}}(u - u_\ell)}_\Gamma.
 \end{split}
\end{equation*}
By~\eqref{eq:symcoupbil} and $\norm{\nabla u - \nabla u_\ell}_{\AAA}^2 = \dual{\AAA \nabla u - \AAA \nabla u_\ell}{\nabla (u - u_\ell)}_{\Omega}$, we obtain
\begin{equation*}
 \begin{split}
  \EE(u_\ell,\phi_\ell)^2
  &= \dual{\AAA \nabla u - \AAA \nabla u_\ell}{\nabla (u - u_\ell)}_{\Omega} + \norm{\nabla u^{\mathrm{ext}} - \nabla \uext_\ell}_{\Omegaext}^2 \\
  &= b\bigl( (u,\phi),(u - u_\ell, \gammaext_1(u^{\mathrm{ext}} - \uext_\ell))\bigr) - b\bigl( (u_\ell,\phi_\ell),(u - u_\ell,\gammaext_1(u^{\mathrm{ext}} - \uext_\ell))\bigr) \\
  &\eqreff*{eq:symcoupbil}{=} F\bigl(u - u_\ell, \gammaext_1(u^{\mathrm{ext}} - \uext_\ell)\bigr) - b\bigl( (u_\ell,\phi_\ell),(u - u_\ell,\gammaext_1(u^{\mathrm{ext}} - \uext_\ell))\bigr).
 \end{split}
\end{equation*}
This concludes the proof of~\eqref{eq:auxerrid}.

\textbf{Step 2.} We now show~\eqref{eq:errid}. 
Since~\eqref{eq:auxerrid} holds for arbitrary approximations $(u_\ell,\phi_\ell) \in H^1(\Omega) \times H^{-1/2}(\Gamma)$, we may replace $u_\ell$ by $\widetilde{u}_\ell$ in~\eqref{eq:auxerrid} to obtain {that}
\begin{equation*}
 \begin{split}
  \EE(\widetilde{u}_\ell,\phi_\ell)^2 
  = F\bigl(u - \widetilde{u}_\ell, \gammaext_1(u^{\mathrm{ext}} - \widetilde{u}_\ell^{\mathrm{ext}})\bigr) 
  - b\bigl( (\widetilde{u}_\ell,\phi_\ell),(u - \widetilde{u}_\ell,\gammaext_1(u^{\mathrm{ext}} - \widetilde{u}_\ell^{\mathrm{ext}}))\bigr).
 \end{split}
\end{equation*}
It remains to show that $\EE(u_\ell,\phi_\ell) = \EE(\widetilde{u}_\ell,\phi_\ell)$. 
Because of $\nabla \widetilde{u}_\ell = \nabla u_\ell$, there holds $\norm{\nabla u - \nabla u_\ell}_\AAA = \norm{\nabla u - \nabla \widetilde{u}_\ell}_\AAA$. 
Due to~\eqref{eq:BIErel} and ${W 1= 0 = (K + 1/2) 1}$ (see, e.g., \cite[Section~8]{McLean2000}), we obtain
\begin{equation*}
 \gamma^{\mathrm{ext}}_0 \widetilde{u}^{\mathrm{ext}}_\ell 
 \eqreff*{eq:BIErel}{=} (K + 1/2)(\gammaint_0\widetilde{u}_\ell - \const{g}{D}) - V\phi_\ell 
 = (K + 1/2)(u_\ell - \const{g}{D}) - V\phi_\ell 
 = \gamma_0^{\mathrm{ext}} \uext_\ell
\end{equation*}
and, {similarly,}
\begin{equation*}
  \begin{split}
   \gamma_1^{\mathrm{ext}} \widetilde{u}^{\mathrm{ext}}_\ell 
   &\eqreff*{eq:BIErel}{=} -W(\gamma_0^{\mathrm{int}}\widetilde{u}_\ell - \const{g}{D}) - (K' - 1/2)\phi_\ell 
   = \gamma_1^{\mathrm{ext}} \uext_\ell.
 \end{split}
\end{equation*}
{This} leads to
\begin{equation*}
 \begin{split}
  \norm{\nabla u^{\mathrm{ext}} - \nabla \uext_\ell}_{\Omegaext}^2 
  &= - \dual{\gamma_1^{\mathrm{ext}}(u^{\mathrm{ext}} - u^{\mathrm{ext}}_\ell)}{\gamma_0^{\mathrm{ext}}(u^{\mathrm{ext}} - u^{\mathrm{ext}}_\ell)}_{\Gamma} \\
  &= - \dual{\gamma_1^{\mathrm{ext}}(u^{\mathrm{ext}} - \widetilde{u}^{\mathrm{ext}}_\ell)}{\gamma_0^{\mathrm{ext}}(u^{\mathrm{ext}} - \widetilde{u}^{\mathrm{ext}}_\ell)}_{\Gamma} 
  = \norm{\nabla u^{\mathrm{ext}} - \nabla \widetilde{u}^{\mathrm{ext}}_\ell}_{\Omegaext}^2.
 \end{split}
\end{equation*}
Putting the foregoing considerations together, {we are led to}
\begin{equation*}
 \begin{split}
  \EE(u_\ell,\phi_\ell)^2 
  &= \norm{\nabla u - \nabla u_\ell}_{\AAA}^2 + \norm{\nabla u^{\mathrm{ext}} - \nabla \uext_\ell}_{\Omegaext}^2
  = \EE(\widetilde{u}_\ell,\phi_\ell)^2 \\
  &= F\bigl(u - \widetilde{u}_\ell, \gammaext_1(u^{\mathrm{ext}} - \widetilde{u}_\ell^{\mathrm{ext}})\bigr) 
  - b\bigl((\widetilde{u}_\ell,\phi_\ell),(u - \widetilde{u}_\ell,\gammaext_1(u^{\mathrm{ext}} - \widetilde{u}_\ell^{\mathrm{ext}}))\bigr).
 \end{split}
\end{equation*}
This concludes the proof.
\end{proof}

\subsection{Functional upper bound}

The following {theorem} states the main result of this work and provides a guaranteed functional upper bound for the full potential error~\eqref{eq:fullerror}.
For {given $\const{g}{D} \in H^1(\Gamma)$ and $\const{g}{N} \in L^2(\Gamma)$, and} arbitrary approximations {$(u_\ell,\phi_\ell) \in H^1(\Omega) \times L^2(\Gamma)$ with $\gammaint_0 u_\ell \in H^1(\Gamma)$,} the upper bound involves the approximate interior normal flux
\begin{equation} \label{eq:Neuapprox}
 \Phi_\ell 
 \coloneqq \const{g}{N} - W(\gammaint_0 u_\ell - \const{g}{D}) - (K' - 1/2)\phi_\ell \in L^2(\Gamma)
\end{equation}
and the Dirichlet residual
\begin{equation} \label{eq:Dirres}
 G_\ell 
 \coloneqq (K - 1/2)(\gammaint_0 u_\ell - \const{g}{D}) - V \phi_\ell \in H^{1}(\Gamma).
\end{equation}
{The existence of computable candidates $\bm{\sigma}_\ell \in H(\div,\Omega)$ and $w_\ell \in H^1_\rho(\Omegaext)$ satisfying the constraints~\eqref{eq:prop} in the following theorem is discussed in the subsequent Sections~\ref{subsec:flux}--\ref{subsec:potential}.}
\begin{theorem}[Functional upper bound] \label{thm:mainresult}
Let $q' \geq 0$, $p' \geq 1$ be polynomial degrees and $J_{\ell,p'} \colon H^1(\Gamma) \to \SS^{p'}(\TT_\ell^\Gamma)$ be an $H^1(\Gamma)$-stable projection. 
{Then, for any} $\bm{\sigma}_\ell \in H(\div,\Omega)$ and {any} $w_\ell \in H^1_\rho(\Omegaext)$ satisfying
\begin{subequations} \label{eq:prop}
\begin{align}
 \div \bm{\sigma}_\ell 
 &= -Q_{\ell,q'}^\Omega f + c \quad \text{for some } c \in \R, \label{eq:divprop} \\
 \bm{\sigma}_\ell \cdot \bm{n}_\Gamma 
 &= Q_{\ell,q'}^\Gamma \Phi_\ell, \label{eq:Neuprop} \\
 \gammaext_0 w_\ell 
 &= J_{\ell,p'} G_\ell, \label{eq:Dirprop}
\end{align}
\end{subequations}
there holds
\begin{equation} \label{eq:funcupperbound}
 \begin{split}
  \EE(u_\ell,\phi_\ell) 
  &\leq \frac{1}{\const{C}{mon}^{1/2}}\norm{\AAA \nabla u_\ell - \bm{\sigma}_\ell}_\Omega + \norm{\nabla w_\ell}_{\Omegaext} + {\const{C}{D}} \norm{h_\ell^{1/2} \nabla_\Gamma (1 - J_{\ell,p'})G_\ell}_{\Gamma} \\
  &\qquad + \frac{1}{\pi \const{C}{mon}^{1/2}}\norm{h_\ell(1 - Q_{\ell,q'}^\Omega)f}_{\Omega}+ {\frac{\const{C}{N}}{\const{C}{mon}^{1/2}}} \norm{h_\ell^{1/2}(1 - Q_{\ell,q'}^\Gamma)\Phi_\ell}_{\Gamma},
 \end{split}
\end{equation}
where $\const{C}{mon}$ is the monotonicity constant of $\AAA$ from \eqref{eq:lipell}, $\const{C}{D}$ is a generic constant depending only on $J_{\ell,p'}$ and $\kappa$-shape regularity, {and $\const{C}{N}$ is the constant from~\eqref{eq:P0approx}.}
\end{theorem}
\begin{proof}
Recall $\widetilde{u}_\ell$ and $\widetilde{u}_\ell^{\mathrm{ext}}$ from~\eqref{eq:auxapprox}. Analogously to~\eqref{eq:Neuapprox}--\eqref{eq:Dirres}, we set
\begin{equation} \label{eq:auxres}
 \begin{split}
  \widetilde{\Phi}_\ell 
  &\coloneqq \const{g}{N} - W(\gamma_0^{\mathrm{int}}\widetilde{u}_\ell - \const{g}{D}) - (K' - 1/2)\phi_\ell, \\
  \widetilde{G}_\ell 
  &\coloneqq (K - 1/2)(\gamma_0^{\mathrm{int}}\widetilde{u}_\ell - \const{g}{D}) - V\phi_\ell.
 \end{split}
\end{equation} 
The remainder of the proof is split into six steps.

\textbf{Step~1.} 
By Lemma~\ref{lem:errid} {and the formulation~\eqref{eq:lhs}--\eqref{eq:symcoupbil} of the symmetric coupling}, we obtain
\begin{equation} \label{eq:erridappl}
 \begin{split}
  &\EE(u_\ell,\phi_\ell)^2 
  \eqreff*{eq:errid}{=} F\bigl(u-\widetilde{u}_\ell,\gamma_1^{\mathrm{ext}}(u^{\mathrm{ext}} - \widetilde{u}^{\mathrm{ext}}_\ell)\bigr) - {b}\bigl((\widetilde{u}_\ell,\phi_\ell), (u-\widetilde{u}_\ell,\gamma_1^{\mathrm{ext}}(u^{\mathrm{ext}} - \widetilde{u}^{\mathrm{ext}}_\ell))\bigr) \\
  &\eqreff*{eq:lhs}{=} \dual{f}{u-\widetilde{u}_\ell}_{\Omega} - \dual{\AAA \nabla \widetilde{u}_\ell}{\nabla(u-\widetilde{u}_\ell)}_{\Omega} \\
  &\qquad + \dual{\const{g}{N} - W(\gamma_0^{\mathrm{int}}\widetilde{u}_\ell - \const{g}{D}) - (K' - 1/2)\phi_\ell}{\gamma_0^{\mathrm{int}}(u-\widetilde{u}_\ell)}_{\Gamma} \\
  &\qquad - \dual{\gamma_1^{\mathrm{ext}} (u^{\mathrm{ext}} - \widetilde{u}^{\mathrm{ext}}_\ell)}{(1/2 - K)(\gamma_0^{\mathrm{int}}\widetilde{u}_\ell - \const{g}{D}) + V\phi_\ell}_{\Gamma} \\
  &\eqreff*{eq:auxres}{=} \dual{f}{u-\widetilde{u}_\ell}_{\Omega} - \dual{\AAA \nabla \widetilde{u}_\ell}{\nabla(u-\widetilde{u}_\ell)}_{\Omega} \\
  &\qquad + \dual{\widetilde{\Phi}_\ell}{\gammaint_0(u-\widetilde{u}_\ell)}_{\Gamma} + \dual{\gamma_1^{\mathrm{ext}}(u^{\mathrm{ext}} - \widetilde{u}^{\mathrm{ext}}_\ell)}{\widetilde{G}_\ell}_{\Gamma}.
 \end{split}
\end{equation}

\textbf{Step~2.}
We consider the first three terms on the right-hand side of~\eqref{eq:erridappl}. 
Let $\bm{\sigma}_\ell \in H(\div,\Omega)$ satisfy~\eqref{eq:divprop} and~\eqref{eq:Neuprop}. 
Green's formula~\eqref{eq:green} establishes
\begin{equation*}
 \dual{\bm{\sigma}_\ell}{\nabla v}_{\Omega} 
 = \dual{\bm{\sigma}_\ell|_{\Gamma} \cdot n_{\Gamma}}{\gamma_0^{\mathrm{int}}v}_{\Gamma} - \dual{\div \bm{\sigma}_\ell}{v}_{\Omega} 
 \quad \text{for all } v \in H^1(\Omega).
\end{equation*}
{N}oting that $\nabla u_\ell = \nabla \widetilde{u}_\ell$, we are led to
\begin{equation} \label{eq:sigmaaddsub}
 \begin{split}
  &\dual{f}{u - \widetilde{u}_\ell}_{\Omega} - \dual{\AAA \nabla \widetilde{u}_\ell}{\nabla(u - \widetilde{u}_\ell)}_{\Omega} + \dual{\widetilde{\Phi}_\ell}{\gamma_0^{\mathrm{int}}(u - \widetilde{u}_\ell)}_{\Gamma} \\
  &\qquad = \dual{f + \div \bm{\sigma}_\ell}{u - \widetilde{u}_\ell}_{\Omega} - \dual{\AAA \nabla u_\ell - \bm{\sigma}_\ell}{\nabla(u - u_\ell)}_{\Omega} \\
  &\qquad \quad  + \dual{\widetilde{\Phi}_\ell - \bm{\sigma}_\ell|_{\Gamma} \cdot n_{\Gamma}}{\gamma_0^{\mathrm{int}}(u - \widetilde{u}_\ell)}_{\Gamma}.
  \end{split}
\end{equation}

\textbf{Step~3.} We deal with each of the three terms on the right-hand side of~\eqref{eq:sigmaaddsub} separately, starting with the first term. 
Due to $\dual{u - \widetilde{u}_\ell}{1}_\Omega = 0$, the property~\eqref{eq:divprop} of $\bm{\sigma}_\ell$, and the properties~\eqref{eq:P0proj} and~\eqref{eq:P0approx} of $Q_{\ell,q'}^\Omega$, there holds
\begin{equation*}
 \begin{split}
  &\dual{f + \div \bm{\sigma}_\ell}{u - \widetilde{u}_\ell}_{\Omega} 
  \eqreff*{eq:divprop}{=} \dual{f - Q_{\ell,q'}^\Omega f + c}{u - \widetilde{u}_\ell}_{\Omega}
  \eqreff*{eq:auxapprox}{=} \dual{(1 - Q_{\ell,q'}^\Omega)f}{u - \widetilde{u}_\ell}_{\Omega} \\
  &\quad \eqreff*{eq:P0proj}{=} \sum_{T \in \TT_\ell} \dual{(1 - Q_{\ell,q'}^\Omega)f}{(1 - Q_{\ell,q'}^\Omega)(u - \widetilde{u}_\ell)}_T \\
  &\quad \leq \sum_{T \in \TT_\ell} \bigl(\norm{(1 - Q_{\ell,q'}^\Omega)f}_T \, \norm{(1 - Q_{\ell,q'}^\Omega)(u - \widetilde{u}_\ell)}_T \bigr) \\
  &\quad \eqreff*{eq:P0approx}{\leq} \frac{1}{\pi} \sum_{T \in \TT_\ell} \bigl(h_T \norm{(1 - Q_{\ell,q'}^\Omega)f}_T \, \norm{\nabla(u - u_\ell)}_T\bigr) 
  \eqreff*{eq:lipell}{\leq} \frac{1}{\pi \const{C}{mon}^{1/2}}\norm{h_\ell(1 - Q_{\ell,q'}^\Omega)f}_{\Omega} \, \norm{\nabla u - \nabla u_\ell}_\AAA.
 \end{split}
\end{equation*} 

\textbf{Step~4.}
For the second summand of {the right-hand side of}~\eqref{eq:sigmaaddsub}, we obtain
\begin{equation} \label{eq:intest}
 \begin{split}
 \dual{\AAA \nabla u_\ell - \bm{\sigma}_\ell}{\nabla(u - u_\ell)}_{\Omega} 
 &\leq \norm{\AAA \nabla u_\ell - \bm{\sigma}_\ell}_\Omega \, \norm{\nabla u - \nabla u_\ell}_\Omega \\
 &\eqreff*{eq:lipell}{\leq} \const{C}{mon}^{-1/2} \norm{\AAA \nabla u_\ell - \bm{\sigma}_\ell}_\Omega \, \norm{\nabla u - \nabla u_\ell}_\AAA.
 \end{split}
\end{equation}

\textbf{Step~5.}
In order to estimate the third summand on the right-hand side of~\eqref{eq:sigmaaddsub}, we note that ${W1} = 0$ (see,e.g.,~\cite[Section~8]{McLean2000}) {yields} $\widetilde{\Phi}_\ell = \Phi_\ell$. 
We use the definition of $\widetilde{u}_\ell$~\eqref{eq:auxapprox}, the property~\eqref{eq:Neuprop} of $\bm{\sigma}_\ell$, and the properties of $Q_{\ell,q'}^\Gamma$ from~\eqref{eq:P0proj} and~\eqref{eq:P0approx} to obtain
\begin{equation*}
 \begin{split}
  \dual{\widetilde{\Phi}_\ell - \bm{\sigma}_\ell|_{\Gamma} \cdot \bm{n}_\Gamma}{\gamma_0^{\mathrm{int}}(u - \widetilde{u}_\ell)}_\Gamma
  &\eqreff*{eq:auxapprox}{=} \dual{\Phi_\ell - \bm{\sigma}_\ell|_{\Gamma} \cdot \bm{n}_\Gamma}{\gamma_0^{\mathrm{int}}(u - u_\ell) - |\Omega|^{-1}\dual{u - u_\ell}{1}_{\Omega}}_\Gamma \\
  &\eqreff*{eq:Neuprop}{=}  \dual{(1 - Q_{\ell,q'}^\Gamma)\Phi_\ell}{\gamma_0^{\mathrm{int}}(u - u_\ell)- |\Omega|^{-1}\dual{u - u_\ell}{1}_{\Omega}}_\Gamma \\
  &= \sum_{F \in \TT_\ell^\Gamma} \dual{(1 - Q_{\ell,q'}^\Gamma)\Phi_\ell}{( 1- Q_{\ell,q'}^\Gamma)\gammaint_0(u - u_\ell)}_F \\
  &\leq \sum_{F \in \TT_\ell^\Gamma}  \norm{(1 - Q_{\ell,q'}^\Gamma)\Phi_\ell}_F \, \norm{(1 - Q_{\ell,q'}^\Gamma)\gammaint_0(u - u_\ell)}_F \\
  &\eqreff*{eq:P0approx}{\leq} {\const{C}{N}}\sum_{F \in \TT_\ell^\Gamma}  h_F^{1/2} \norm{(1 - Q_{\ell,q'}^\Gamma)\Phi_\ell}_F \, \norm{\nabla(u - u_\ell)}_{T_F} \\
  &\eqreff*{eq:lipell}{\leq}  \frac{{\const{C}{N}}}{\const{C}{mon}^{1/2}} \norm{h_\ell^{1/2}(1 - Q_{\ell,q'}^\Gamma)\Phi_\ell}_\Gamma \, \norm{\nabla u - \nabla u_\ell}_{\AAA}. \\
 \end{split}
\end{equation*}
Althogether, we obtain
\begin{equation} \label{eq:interrderiv}
 \begin{split}
  &\dual{f}{u - \widetilde{u}_\ell}_{\Omega} - \dual{\AAA \nabla \widetilde{u}_\ell}{\nabla(u - \widetilde{u}_\ell)}_{\Omega} + \dual{\widetilde{\Phi}_\ell}{\gamma_0^{\mathrm{int}}(u - \widetilde{u}_\ell)}_{\Gamma} \\
  &\qquad \leq \EE(u_\ell,\phi_\ell) \bigl[\frac{1}{\const{C}{mon}^{1/2}}\norm{\AAA \nabla u_\ell - \bm{\sigma}}_\Omega + \frac{1}{\pi\const{C}{mon}^{1/2}}\norm{h_\ell (1 - Q_{\ell,q'}^\Omega)f}_\Omega \\
  &\qquad \quad +  \frac{{\const{C}{N}}}{\const{C}{mon}^{1/2}} \norm{h_\ell^{1/2}(1 - Q_{\ell,q'}^\Gamma)\Phi_\ell}_\Gamma\bigr].
 \end{split}
\end{equation}

\textbf{Step~6.}
It thus only remains to estimate the Dirichlet residual, which is the final term on the right-hand side of~\eqref{eq:erridappl}. 
Let $\widetilde{w} \in H^1_\rho(\Omegaext)$ be an arbitrary function with $\gammaext_0 \widetilde{w} = \widetilde{G}_\ell$. 
Since $u^{\mathrm{ext}}, \widetilde{u}^{\mathrm{ext}}_\ell \in H^1_\rho(\Delta,\Omegaext)$, Green's formula~\eqref{eq:green} and the definition of $\widetilde{u}^{\mathrm{ext}}_\ell$ from~\eqref{eq:auxapprox} yield
\begin{equation*}
 \begin{split}
  \dual{\gamma_1^{\mathrm{ext}} (u^{\mathrm{ext}} - \widetilde{u}^{\mathrm{ext}}_\ell)}{\widetilde{G}_\ell}_{\Gamma} 
  &= \dual{\gamma_1^{\mathrm{ext}}(u^{\mathrm{ext}} - \widetilde{u}^{\mathrm{ext}}_\ell)}{\gamma_0^{\mathrm{ext}}\widetilde{w}}_{\Gamma} 
  \eqreff*{eq:green}{=} -\dual{\nabla(u^{\mathrm{ext}} - \widetilde{u}^{\mathrm{ext}}_\ell)}{\nabla \widetilde{w}}_{\Omegaext} \\
  &\leq  \norm{\nabla(u^{\mathrm{ext}} - \uext_\ell)}_{\Omegaext} \, \norm{\nabla \widetilde{w}}_{\Omegaext} 
  \leq \EE(u_\ell,\phi_\ell)\norm{\nabla \widetilde{w}}_{\Omegaext}.
 \end{split}
\end{equation*}
Since $(K - 1/2){1} = 1$ (see, e.g.,~\cite[Section~8]{McLean2000}), there holds $\widetilde{G}_\ell = G_\ell - \abs{\Omega}^{-1}\dual{u - u_\ell}{1}_\Omega$. 
Since $\widetilde{w} \in H^1_\rho(\Omegaext)$ with $\gammaext_0 \widetilde{w} = \widetilde{G}_\ell$ was arbitrary and $G_\ell$ and $\widetilde{G}_\ell$ differ only by a constant, we obtain
\begin{equation*} \label{eq:dirichlet}
 \dual{\gamma_1^{\mathrm{ext}} (u^{\mathrm{ext}} - \widetilde{u}^{\mathrm{ext}}_\ell)}{\widetilde{G}_\ell}_{\Gamma} 
 \leq \EE(u_\ell,\phi_\ell) \min_{\substack{\widetilde{w} \in H^1_{\mathrm{ext}}(\Omegaext) \\ \widetilde{w}|_{\Gamma} = \widetilde{G}_\ell}} \norm{\nabla \widetilde{w}}_{\Omegaext}
 = \EE(u_\ell,\phi_\ell) \min_{\substack{w \in H^1_{\mathrm{ext}}(\Omegaext) \\ w|_{\Gamma} = G_\ell}} \norm{\nabla w}_{\Omegaext}.
\end{equation*}
Since $J_{\ell,p'}$ is $H^1(\Gamma)$-stable, a result from \cite{Aurada2015} yields
\begin{equation*}
 \begin{split}
  \min_{\substack{w \in H^1_{\mathrm{ext}}(\Omegaext) \\ w|_{\Gamma} = G_\ell}} \norm{\nabla w}_{L^2(\Omegaext)} 
  &\leq \min_{\substack{w \in H^1_{\mathrm{ext}}(\Omegaext) \\ w|_{\Gamma} = J_{\ell,p'} G_\ell}} \norm{\nabla w}_{L^2(\Omegaext)} + {\const{C}{D}} \norm{h_\ell^{1/2} \nabla_\Gamma (1 - J_{\ell,p'})G_\ell}_{L^2(\Gamma)} \\
  &\eqreff*{eq:Dirprop}{\leq} \norm{\nabla w_\ell}_{L^2(\Omegaext)} + {\const{C}{D}} \norm{h_\ell^{1/2} \nabla_\Gamma (1 - J_{\ell,p'})G_\ell}_{L^2(\Gamma)};
 \end{split}
\end{equation*}
{see also~\cite[Theorem~5]{Kurz2019}.}
Thus,
\begin{equation} \label{eq:direrrderiv}
 \dual{\gamma_1^{\mathrm{ext}} (u^{\mathrm{ext}} - \widetilde{u}^{\mathrm{ext}}_\ell)}{\widetilde{G}_\ell}_{\Gamma} 
 \leq \EE(u_\ell,\phi_\ell) {\bigl(\norm{\nabla w_\ell}_{\Omegaext} + \const{C}{D} \norm{h_\ell^{1/2} \nabla_\Gamma (1 - J_{\ell,p'})G_\ell}_{L^2(\Gamma)} \bigr)}.
\end{equation}
Combining~\eqref{eq:erridappl},~\eqref{eq:interrderiv}, and~\eqref{eq:direrrderiv}, we conclude the proof.
\end{proof}
\begin{remark}
The fact that the choice of $J_{\ell,p'}$ is limited only to its $H^1(\Gamma)$-stability allows for a considerable freedom in choosing $J_{\ell,p'}$. 
For instance, {one} can employ the $L^2(\Gamma)$-orthogonal projection defined via
\begin{equation*}
 \dual{J_{\ell,p'} v}{v_\ell}_\Gamma 
 = \dual{v}{v_\ell}_\Gamma 
 \quad \text{for all } v \in L^2(\Gamma) \text{ and } v_\ell \in \SS^{p'}(\TT_\ell^\Gamma).
\end{equation*}
Under reasonable assumptions (see, e.g., \cite{Crouzeix1987} for $d=2$ and \cite{Gaspoz2016} for $d=3$), $J_{\ell,p'}$ is $H^1(\Gamma)$-stable.
Alternatively, one can choose the Scott--Zhang operator $J_{\ell,p'}$; see \cite{Scott1990}; which is known to be unconditionally $H^1(\Gamma)$-stable. 
\end{remark}
The subsequent corollaries present two modifications to Theorem~\ref{thm:mainresult}.
\begin{corollary} \label{cor:corint}
If $\AAA$ is {linear}, the term $\const{C}{mon}^{-1/2}\norm{\AAA \nabla u_\ell - \bm{\sigma}_\ell}$ in~\eqref{eq:funcupperbound} can be replaced by 
\begin{equation} \label{eq:linconst}
 \norm{\AAA \nabla u_\ell - \bm{\sigma}_\ell}_{\AAA^{-1}} 
 \coloneqq \dual{\AAA \nabla u_\ell - \bm{\sigma}_{\ell}}{\nabla u_\ell - \AAA^{-1} \bm{\sigma}_\ell}_{\Omega}^{1/2}.
\end{equation}
{Under the assumptions of Theorem~\ref{thm:mainresult}, this yields the error estimate}
\begin{equation} \label{eq:funcupperboundalt1}
 \begin{split}
  \EE(u_\ell,\phi_\ell) 
  &\leq \norm{\AAA \nabla u_\ell - \bm{\sigma}_\ell}_{\AAA^{-1}} + \norm{\nabla w_\ell}_{\Omegaext} + {\const{C}{D}} \norm{h_\ell^{1/2} \nabla_\Gamma (1 - J_{\ell,p'})G_\ell}_{\Gamma} \\
  &\qquad + \frac{1}{\pi \const{C}{mon}^{1/2}}\norm{h_\ell(1 - Q_{\ell,q'}^\Omega)f}_{\Omega}+ \frac{{\const{C}{N}}}{\const{C}{mon}^{1/2}} \norm{h_\ell^{1/2}(1 - Q_{\ell,q'}^\Gamma)\Phi_\ell}_{\Gamma} .
 \end{split}
\end{equation}
\end{corollary}
\begin{proof}
Since $\AAA$ is linear, we are in the position to apply the Cauchy--Schwarz inequality with respect to the inner product $\dual{\AAA \nabla \cdot}{\nabla \cdot}_{\Omega}$. 
Hence,~\eqref{eq:intest} can be replaced by
\begin{equation*}
 \dual{\AAA \nabla u_\ell - \bm{\sigma}_\ell}{\nabla(u - u_\ell)}_{\Omega} 
 \leq \norm{\AAA \nabla u_\ell - \bm{\sigma}_\ell}_{\AAA^{-1}} \, \norm{\nabla u - \nabla u_\ell}_{\AAA}
 \leq \EE(u_\ell,\phi_\ell) \norm{\AAA \nabla u_\ell - \bm{\sigma}_\ell}_{\AAA^{-1}}.
\end{equation*}
{Arguing as before, this} concludes the proof.
\end{proof}
By use of the Poincaré--Steklov operators~\eqref{eq:PS}, one can {replace} $w_\ell {\in H^1_\rho(\Omegaext)}$ in Theorem~\ref{thm:mainresult} {by $w_\ell \in H^1(\Omega)$}.
\begin{corollary} \label{cor:corext}
{Under the assumptions of Theorem~\ref{thm:mainresult} but with $w_\ell \in H^1(\Omega)$ satisfying $\gamma_0^{\mathrm{int}} w_\ell = J_{\ell,p'} G_\ell$, there holds the error estimate}
\begin{equation} \label{eq:funcupperboundalt2}
 \begin{split}
  \EE(u_\ell,\phi_\ell) 
  &\leq \frac{1}{\const{C}{mon}^{1/2}}\norm{\AAA \nabla u_\ell - \bm{\sigma}_\ell}_\Omega + \frac{\const{C}{PS}^{\mathrm{ext}}}{\const{c}{PS}^{\mathrm{int}}} \norm{\nabla w_\ell}_{\Omega} + {\const{C}{D}} \norm{h_\ell^{1/2} \nabla_\Gamma (1 - J_{\ell,p'})G_\ell}_{\Gamma} \\
  &\qquad + \frac{1}{\pi \const{C}{mon}^{1/2}}\norm{h_\ell(1 - Q_{\ell,q'}^\Omega)f}_{\Omega}+ \frac{{\const{C}{N}}}{\const{C}{mon}^{1/2}} \norm{h_\ell^{1/2}(1 - Q_{\ell,q'}^\Gamma)\Phi_\ell}_{\Gamma},
 \end{split}
\end{equation}
where $\const{C}{PS}^{\mathrm{ext}}$ and $\const{c}{PS}^{\mathrm{int}}$ are the continuity and ellipticity constants of the exterior and interior Poincaré--Steklov operators~\eqref{eq:PS}, respectively. 
\end{corollary}
\begin{proof}
Let $g \in H^{1/2}(\Gamma)$. 
The solutions $u_g^{\mathrm{int}} \in H^1(\Omega)$ and $u_g^{\mathrm{ext}} \in H^1_\rho(\Omegaext)$ to the interior and exterior Laplace problems with Dirichlet data $g$ satisfy
\begin{equation*}
 \norm{\nabla u_g^{\mathrm{int}}}_\Omega 
 = \min_{\substack{w \in H^1(\Omega) \\ \gamma_0^{\mathrm{int}} w = g}} \norm{\nabla w}_\Omega 
 \quad \text{and} \quad 
 \norm{\nabla u_g^{\mathrm{ext}}}_{\Omegaext} 
 = \min_{\substack{w \in H^1_\rho(\Omegaext) \\ \gamma_0^{\mathrm{ext}} w = g}} \norm{\nabla w}_{\Omegaext}.
\end{equation*}
Since $\Delta u_g^{\mathrm{ext}} = 0$ in $\Omegaext$, {the identity}~\eqref{eq:harmonicnorm} and {the properties of the exterior Poincaré--Steklov operator}~\eqref{eq:PSprop} yield
\begin{equation*}
 \min_{\substack{w \in H^1_\rho(\Omegaext) \\ \gamma_0^{\mathrm{ext}} w = g}} \norm{\nabla w}_{\Omegaext}^2
 = \norm{\nabla u_g^{\mathrm{ext}}}_{\Omegaext}^2
 \eqreff*{eq:harmonicnorm}{=} - \dual{g}{\gammaext_1 u_g^{\mathrm{ext}}}_\Gamma
 = - \dual{g}{S^{\mathrm{ext}} g}_\Gamma 
 \eqreff*{eq:PSprop}{\leq} \const{C}{PS}^{\mathrm{ext}} \norm{g}_{H^{1/2}(\Gamma)}^2.
\end{equation*}
{Analogously,} $\Delta u_g^{\mathrm{int}} = 0$ in $\Omega$, {the identity}~\eqref{eq:green}, and {the properties of the interior Poincaré--Steklov operator}~\eqref{eq:PSprop} yield
\begin{equation*}
 \min_{\substack{w \in H^1(\Omega) \\ \gamma_0^{\mathrm{int}} w = g}} \norm{\nabla w}^2_\Omega 
 = \norm{\nabla u_g^{\mathrm{int}}}_\Omega^2
 \eqreff*{eq:green}{=} \dual{g}{\gammaint_1 u_g^{\mathrm{int}}}_\Gamma
 = \dual{g}{S^{\mathrm{int}} g}_\Gamma 
 \eqreff*{eq:PSprop}{\geq} \const{c}{PS}^{\mathrm{int}} \norm{g}_{H^{1/2}(\Gamma)}^2.
\end{equation*}
Combining the last two estimates, we obtain
\begin{equation*}
 \min_{\substack{w \in H^1_{{\rho}}(\Omegaext) \\ \gamma_0^{\mathrm{ext}} w = g}} \norm{\nabla w}_{\Omegaext}^2 
 \leq \frac{\const{C}{PS}^{\mathrm{ext}}}{\const{c}{PS}^{\mathrm{int}}} \min_{\substack{w \in H^1(\Omega) \\ \gamma_0^{\mathrm{int}} w = g}} \norm{\nabla w}_\Omega^2.
\end{equation*}
Hence, we may replace {the argument in}~\eqref{eq:direrrderiv} by{
\begin{equation*}
 \begin{split}
  \dual{\gamma_1^{\mathrm{ext}} (u^{\mathrm{ext}} - \widetilde{u}^{\mathrm{ext}}_\ell)}{\widetilde{G}_\ell}_{\Gamma} 
  &\leq \EE(u_\ell,\phi_\ell) \bigl(\min_{\substack{w \in H^1_{{\rho}}(\Omegaext) \\ w|_{\Gamma} = J_{\ell,p'} G_\ell}} \norm{\nabla w}_{\Omegaext} + \const{C}{D} \norm{h_\ell^{1/2} \nabla_\Gamma (1 - J_{\ell,p'})G_\ell}_{\Gamma} \bigr) \\
  &\leq \EE(u_\ell,\phi_\ell) \Bigl(\frac{\const{C}{PS}^{\mathrm{ext}}}{\const{c}{PS}^{\mathrm{int}}} \norm{\nabla w_\ell}_\Omega + \const{C}{D} \norm{h_\ell^{1/2} \nabla_\Gamma (1 - J_{\ell,p'})G_\ell}_{\Gamma} \Bigr).
 \end{split}
\end{equation*}}
This concludes the proof.
\end{proof}

\subsection{Computation of auxiliary flux $\texorpdfstring{\boldsymbol{\sigma_\ell \in H(\div,\Omega)}}{}$ satisfying~\bfseries(\ref{eq:prop})} \label{subsec:flux}

In order to determine a suitable flux $\bm{\sigma}_\ell \approx \AAA \nabla u_\ell$, we adopt the idea of \textit{equilibrated fluxes}; see, e.g., \cite{Ainsworth1997,Braess2008,Ern2015}. 
For every $z \in \NN_\ell$, we set $\omega_\ell^z \coloneqq \mathrm{int}(\Omega_\ell[z])$ and consider the hat functions $(\zeta_\ell^z)_{z \in \NN_\ell} \subseteq \SS^1(\TT_\ell)$. 
For a polynomial degree $q' \geq 0$, we employ the discrete dual mixed formulation of the pure Neumann problem in $\omega_\ell^z \colon$ 
Find $(\bm{\sigma}_\ell^z,p_\ell^z) \in \BB\DD\MM^{q'+1}(\TT_\ell[z]) \times \PP^{q'}(\TT_\ell[z])$ such that, for all $(\bm{\tau}_\ell,q_\ell) \in \BB\DD\MM_0^{q'+1}(\TT_\ell[z]) \times \PP^{q'}(\TT_\ell[z])$,
\begin{subequations} \label{eq:dualmixed}
\begin{equation} \label{eq:dualmixedloc}
 \begin{split}
  \dual{\bm{\sigma}_\ell^z}{\bm{\tau}_\ell}_{\omega_\ell^z} + \dual{p_\ell^z}{\div \bm{\tau}_\ell}_{\omega_\ell^z} 
  &= \dual{\zeta_\ell^z \AAA \nabla u_\ell}{\bm{\tau}_\ell}_{\omega_\ell^z}  \\
  \dual{\div \bm{\sigma}_\ell^z}{q_\ell}_{\omega_\ell^z} + \dual{p_\ell^z}{1}_{\omega_\ell^z} \dual{q_\ell}{1}_{\omega_\ell^z} 
  &= \dual{\nabla \zeta_\ell^z \cdot \AAA \nabla u_\ell - \zeta_\ell^z f}{q_\ell}_{\omega_\ell^z}
 \end{split}
\end{equation}
subject to the boundary conditions, for all $F \in \FF_\ell$ with $F \subseteq \partial \omega_\ell^z$,
\begin{equation} \label{eq:dualmixedbndry}
 \bm{\sigma}_\ell^z|_F \cdot \bm{n}_F 
 = 
\begin{cases} Q_{\ell,q'}^\Gamma \bigl(\zeta_\ell^z \Phi_\ell\bigr)|_F & \text{if } z \in \Gamma \text{ and } F \in \TT_\ell^\Gamma, \\
  0 & \text{if } z \in \Omega \text{ or } F \notin \TT_\ell^\Gamma.
  \end{cases} 
\end{equation}
We refer to \cite{Arnold1985} for existence and uniqueness of the solution $(\bm{\sigma}_\ell^z,p_\ell^z)$ to~\eqref{eq:dualmixedloc} subject to~\eqref{eq:dualmixedbndry}. 
For every $z \in \NN_\ell$, we extend $\bm{\sigma}_\ell^z$ and $p_\ell^z$ by zero to functions defined in $\Omega$ and set
\begin{equation} \label{eq:fluxglobdef}
 \bm{\sigma}_\ell 
 \coloneqq \sum_{z \in \NN_\ell} \bm{\sigma}_\ell^z, 
 \quad p_\ell 
 \coloneqq \sum_{z \in \NN_\ell} p_\ell^z.
\end{equation}
\end{subequations}
The following proposition {shows that} $\bm{\sigma}_\ell$ {is a valid candidate in Theorem~\ref{thm:mainresult}, Corollary~\ref{cor:corint}, and Corollary~\ref{cor:corext}}, i.e., it satisfies the required constraints~\eqref{eq:prop}.
\begin{proposition} \label{prop:flux}
The flux $\bm{\sigma}_\ell$ from~\eqref{eq:dualmixed} satisfies $\bm{\sigma}_\ell \in \BB\DD\MM^{q'+1}(\TT_\ell) \subseteq H(\div,\Omega)$ with
 \begin{equation} \label{eq:flux}
  \begin{split}
   \div \bm{\sigma}_\ell 
   = -Q_{\ell,q'}^\Omega f - \dual{p_\ell}{1}_{\Omega}
   \quad \text{and} \quad
   \bm{\sigma}_\ell \cdot \bm{n}_\Gamma 
   = Q_{\ell,q'}^\Gamma\Phi_\ell.
  \end{split}
 \end{equation}
\end{proposition}
\begin{proof}
Since $\bm{\sigma}_\ell^z|_{\partial \omega_\ell^z \cap \Omega} \cdot \bm{n}_{\partial \omega_\ell^z} = 0$, extension by zero preserves normal continuity of $\bm{\sigma}_\ell^z$ across the patch boundary $\partial \omega_\ell^z \cap \Omega$. 
Thus, we obtain $\bm{\sigma}_\ell^z \in \BB\DD\MM^{q'+1}(\TT_\ell)$ and, consequently, $\bm{\sigma}_\ell \in \BB\DD\MM^{q'+1}(\TT_\ell) \subseteq H(\div,\Omega)$.
Due to $(\zeta_{\ell,z}|_F)_{z \in F}$ being a partition of unity for every $F \in \TT_\ell^\Gamma$ (see~\eqref{eq:hat}), we obtain
\begin{equation*}
 \begin{split}
  \bm{\sigma}_\ell|_F \cdot \bm{n}_F 
  &= \sum_{z \in \NN_\ell \cap F} \bigl(Q_{\ell,q'}^\Gamma(\zeta_{\ell,z}\Phi_\ell)\bigr)|_F 
  \eqreff*{eq:hat}{=} (Q_{\ell,q'}^\Gamma \Phi_\ell)|_F
  \quad \text{for all } F \in \TT_\ell^\Gamma,
 \end{split}
\end{equation*}
which implies $\bm{\sigma}_\ell|_\Gamma \cdot \bm{n}_\Gamma = Q_{\ell,q'}^\Gamma \Phi_\ell$. 
This concludes the second equality in~\eqref{eq:flux}. 
Due to~\eqref{eq:dualmixed} and the properties~\eqref{eq:hat} of the hat functions, we obtain
\begin{equation*} 
 \begin{split}
  \dual{\div \bm{\sigma}_\ell}{q_\ell}_{\Omega} 
  &= \sum_{T \in \TT_\ell} \sum_{z \in \NN_\ell} \dual{\div \bm{\sigma}_\ell^z}{q_\ell}_T 
  = \sum_{T \in \TT_\ell} \sum_{z \in \NN_\ell} \dual{\nabla \zeta_\ell^z \cdot \AAA \nabla u_\ell - \zeta_\ell^z f - \dual{p_\ell^z}{1}_{\omega_\ell^z}}{q_\ell}_T \\
  &\eqreff*{eq:hat}{=} \sum_{T \in \TT_\ell}\Bigl[ - \dual{f}{q_\ell}_T - \big\langle \sum_{z \in \NN_\ell} \dual{p_\ell^z}{1}_\Omega,q_\ell \big\rangle_T \Bigr] 
  = - \dual{f + \dual{p_\ell}{1}_\Omega}{q_\ell}_\Omega
 \end{split}
\end{equation*}
for all $q_\ell \in \PP^{q'}(\TT_\ell)$. 
Lastly, since $\bm{\sigma}_\ell \in \BB\DD\MM^{q'+1}(\TT_\ell)$, it holds that $\div \bm{\sigma}_\ell \in \PP^{q'}(\TT_\ell)$. 
Hence, the last identity implies $\div \bm{\sigma}_\ell = -Q_{\ell,q'}^\Omega f - \dual{p_\ell}{1}_\Omega$.
\end{proof}

\subsection{Computation of auxiliary potential $\texorpdfstring{\boldsymbol{w_\ell \in H^1_\rho(\Omegaext)}}{}$ satisfying~\bfseries(\ref{eq:prop})}  \label{subsec:potential}

In order to obtain a suitable function $w_\ell$, we suppose that there is a bounded domain $\widetilde{\Omega}$ such that $\Omega \cap \widetilde{\Omega} = \emptyset$ and $\Gamma \subseteq \partial \widetilde{\Omega}$. 
Moreover, we suppose that there is a triangulation $\widetilde{\TT}_\ell$ of $\widetilde{\Omega}$ such that $\widetilde{\TT}_\ell|_\Gamma = \TT_\ell^\Gamma$. 
Let $k,p' \geq 1$ be fixed and for $z \in \NN_\ell^\Gamma$ set $\widetilde{\omega}_\ell^{z,k} \coloneqq \mathrm{int}(\widetilde{\Omega}_\ell^{[k]}[z]) \subseteq \widetilde{\Omega}$. 
We define
\begin{equation*}
 \xi_\ell^z 
 \coloneqq \abs{\widetilde{\omega}_\ell^{z,k} \cap \NN_\ell^\Gamma}^{-1} \sum_{z' \in \widetilde{\omega}_\ell^{z,k} \cap \NN_\ell^\Gamma} \zeta_{z'}|_{\Gamma} \in \SS^1(\TT_\ell^\Gamma) 
 \quad \text{for all } z \in \NN_\ell^\Gamma.
\end{equation*} 
Let $J_{\ell,p'} \colon H^1(\Gamma) \to \SS^{p'}(\TT_\ell^\Gamma)$ be a $H^1(\Gamma)$-stable projection and consider the problem of finding $w_\ell^z \in \SS^{p'+1}(\widetilde{\TT}_\ell^{[k]}[z])$ such that 
\begin{subequations} \label{eq:wdef}
\begin{equation} \label{eq:wlocdef}
 \begin{split}
  \dual{\nabla w_\ell^z}{\nabla v_\ell}_{\widetilde{\omega}_\ell^{z,k}}
  &= 0 \quad \text{for all } v_\ell \in \SS^{p'+1}_0(\widetilde{\TT}_\ell^{[k]}[z]) \\
  w_\ell^z|_{\partial \widetilde{\omega}_\ell^{z,k}} &= 
  \begin{cases}
   \xi_\ell^z J_{\ell,p'}(G_\ell) & \text{on } \partial \widetilde{\omega}_\ell^{z,k} \cap \Gamma, \\
   0 & \text{on } \partial \widetilde{\omega}_\ell^{z,k} \setminus \Gamma.
  \end{cases}
 \end{split}
\end{equation}
For every $z \in \NN_\ell^\Gamma$, the problem~\eqref{eq:wlocdef} admits a unique solution $w_\ell^z \in \SS^{p'{+1}}(\widetilde{\TT}_\ell^{[k]}[z])$ for every $z \in \NN_\ell^\Gamma$; see, e.g., \cite{Bartels2004,Aurada2013}. 
{We extend $w_\ell^z$ by zero to a function defined in $\Omegaext$ and} set
\begin{equation} \label{eq:wglobdef}
 w_\ell 
 \coloneqq \sum_{z \in \NN_\ell^\Gamma} w_\ell^z.
\end{equation}
\end{subequations}
{The following proposition shows that $w_\ell$ is a valid candidate in Theorem~\ref{thm:mainresult} and Corollary~\ref{cor:corint}, i.e., it satisfies the required constraints~\eqref{eq:prop}. 
\begin{proposition} \label{prop:potential}
The function $w_\ell$ from~\eqref{eq:wdef} satisfies $w_\ell \in H^1_\rho(\Omegaext)$ with
\begin{equation} \label{eq:potential}
  \gammaext_0 w_\ell 
  = J_{\ell,p'} G_\ell.
\end{equation}
\end{proposition}
\begin{proof}
Since $w_\ell^z|_{\partial \widetilde{\omega}_\ell^{z,k} \setminus \Gamma} = 0$ for all $z \in \NN_\ell^\Gamma$, extension by zero preserves continuity of $w_\ell^z$ across the patch boundary $\partial \widetilde{\omega}_\ell^{z,k} \cap \Omegaext$. Thus, we obtain $w_\ell^z \in H^1_\rho(\Omegaext)$ and, consequently, $w_\ell \in H^1_\rho(\Omegaext)$. Since the hat functions constitute a partition of unity on $\Gamma$ (see~\eqref{eq:hat}), we obtain 
\begin{equation} \label{eq:primaltrace}
 \gammaext_0 w_\ell 
 = \sum_{z \in \NN_\ell^\Gamma} \xi_\ell^z J_{\ell,p'}(G_\ell) 
 = J_{\ell,p'}(G_\ell) \sum_{z \in \NN_\ell^\Gamma} \abs{\widetilde{\omega}_\ell^z \cap \NN_\ell^\Gamma}^{-1} \sum_{z' \in \widetilde{\omega}_\ell^z \cap \NN_\ell^\Gamma} \zeta_{z'} 
 \eqreff*{eq:hat}{=} J_{\ell,p'} G_\ell.
\end{equation}
This concludes the proof.
\end{proof}}

\begin{remark}
 One could also adopt the initial approach of \cite{Kurz2019} and consider a global formulation of the primal problem on the patch of $\Gamma$, i.e., find $w_\ell \in \SS^{p'}(\widetilde{\TT}^{[k]}[\Gamma])$ such that
\begin{equation} \label{eq:wglobsolve}
 \begin{split}
  \dual{\nabla w_\ell}{\nabla v_\ell}_{\widetilde{\Omega}_\ell^{[k]}[\Gamma]} &= 0 \quad \text{for all } v_\ell \in \SS^{p'}_0(\widetilde{\TT}^{[k]}[\Gamma]) \\
  w_\ell^z|_{\partial \widetilde{\Omega}_\ell^{[k]}[\Gamma]} 
  &= 
  \begin{cases} 
   J_{\ell,p'}(G_\ell) & \text{on } \partial \widetilde{\Omega}_\ell^{[k]}[\Gamma] \cap \Gamma, \\
   0 & \text{on } \partial \widetilde{\Omega}_\ell^{[k]}[\Gamma] \setminus \Gamma.
  \end{cases}
 \end{split}
\end{equation}
However, solving local problems~\eqref{eq:wdef} instead of a global problem lowers the computational costs.
\end{remark}
\begin{remark}
Considering Corollary~\ref{cor:corext}, we procede analogously to~\eqref{eq:wdef} but replace $\widetilde{\omega}_{\ell}^{z,k} \subset \Omegaext$ by $\omega_{\ell}^{z,k} \coloneqq \Omega_\ell^{[k]}[z] \subseteq \Omega$ and solve the local problems
 \begin{subequations} \label{eq:wintdef}
  \begin{equation} \label{eq:wintlocdef}
   \begin{split}
    \dual{\nabla w_\ell^z}{\nabla v_\ell}_{\omega_\ell^{z,k}} 
    &= 0 \quad \text{for all } v_\ell \in \SS^{p'+1}_0(\TT_\ell^{[k]}[z]) \\
    w_\ell^z|_{\partial \omega_\ell^{z,k}} &= 
    \begin{cases}
     \xi_\ell^z J_{\ell,p'}(G_\ell) & \text{on } \partial \omega_\ell^{z,k} \cap \Gamma, \\
     0 & \text{on } \partial \omega_\ell^{z,k} \setminus \Gamma.
    \end{cases}
   \end{split}
  \end{equation}
  Analogous arguments as in the proof of Proposition~\ref{prop:potential} show that the extension by zero of $w_\ell^z$ to $\Omega$ lies in $H^1(\Omega)$ and that the function
  \begin{equation} \label{eq:wintglobdef}
   w_\ell 
   \coloneqq \sum_{z \in \NN_\ell^\Gamma} w_\ell^z \in H^1(\Omega)
  \end{equation}
  \end{subequations}
  satisfies $\gammaint_0 w_\ell = J_{\ell,p'} G_\ell$.
  Thus, $w_\ell$ is a valid candidate for the upper bound~\eqref{eq:funcupperboundalt2} in Corollary~\ref{cor:corext}.
  Note that this also eliminates the need for the exterior mesh $\widetilde{\TT}_\ell$ providing an additional advantage.
\end{remark}
\section{Adaptive FEM-BEM algorithm} \label{sec:adaptive}

\subsection{Adaptive algorithm based on Theorem~\ref{thm:mainresult}} \label{subsec:adaptiveext}

In addition to $\Omega$, recall that $\widetilde{\Omega} \subset \R^d \setminus \overline{\Omega} = \Omegaext$ is a bounded polygonal Lipschitz domain with $\Gamma = \partial \Omega \cap \partial \widetilde{\Omega}$.
All our numerical experiments start with a conforming simplicial triangulation $\widehat{\TT}_0$ of the overall domain $\widehat{\Omega} \coloneqq \overline{\Omega} \cup \widetilde{\Omega}$, which satisfies
\begin{equation*}
 \bigcup \set{T \in \widehat{\TT}_0 \given T \subseteq \overline{\Omega}} = \overline{\Omega},
\end{equation*} 
i.e., $\Omega$ and therefore also $\widetilde{\Omega}$ and $\Gamma$ are fully resolved by the initial mesh $\widehat{\TT}_0$. 
For local mesh-refinement, we employ newest-vertex bisection \cite{Stevenson2008}. 
Recall the $\widehat{\TT}_\ell$-induced meshes $\TT_\ell \coloneqq \widehat{\TT}_\ell|_{\Omega}$, $\widetilde{\TT}_\ell \coloneqq \widehat{\TT}_\ell|_{\widetilde{\Omega}}$, and $\TT_\ell^\Gamma = \TT_\ell|_\Gamma = \widetilde{\TT}_\ell|_\Gamma$ on $\Gamma$.
We compute $\bm{\sigma}_\ell \in \BB\DD\MM^{q'+1}(\TT_\ell)$ and $w_\ell \in H^1_\rho(\Omegaext)$ as outlined in Section~\ref{subsec:flux}--\ref{subsec:potential}, respectively.
The \textsl{a~posteriori} error estimate of Theorem~\ref{thm:mainresult} gives rise to the (fully computable) local error indicators
\begin{equation} \label{eq:indicators}
 \begin{split}
  \eta_{\ell}^{\rm{int}}(T) 
  &\coloneqq \norm{\AAA u_\ell - \bm{\sigma}_\ell}_T 
  \quad \text{for } T \in \TT_\ell, \\
  \widetilde{\eta}_\ell^{\rm{ext}}(T) 
  &\coloneqq \norm{\nabla w_\ell}_T 
  \quad \text{for } T \in \widetilde{\TT}_\ell, \\
  \mathrm{osc}_\ell^{\Omega}(T) &\coloneqq h_T\norm{(1 - Q_{\ell,q'}^\Omega)f}_T \quad \text{for } T \in \TT_\ell, \\
  \widetilde{\mathrm{osc}}_\ell^{\rm{D}}(T) &\coloneqq {\frac{1}{2}} h_T^{1/2}\norm{\nabla_\Gamma (1 - J_{\ell,p'})G_\ell}_{\partial T \cap \Gamma} \quad \text{for } T \in \TT_\ell \cup \widetilde{\TT}_\ell, \\
  \widetilde{\mathrm{osc}}_{\ell}^{\rm{N}}(T) &\coloneqq {\frac{1}{2}}h_T^{1/2}\norm{(1 - Q_{\ell,q'}^\Gamma)\Phi_\ell}_{\partial T \cap \Gamma} \quad \text{for } T \in \TT_\ell \cup \widetilde{\TT}_\ell.
 \end{split}
\end{equation}
In addition, we define $\eta_\ell^{\rm{int}}(T) \coloneqq \mathrm{osc}_{\ell}^{\Omega}(T) \coloneqq 0$ for $T \in \widetilde{\TT}_\ell$ and $\widetilde{\eta}_\ell^{\rm{ext}}(T) \coloneqq 0$ for $T \in \TT_\ell$, as well as
\begin{equation} \label{eq:fullind}
 \begin{split}
  \widetilde{\eta}_\ell(T)^2 
  &\coloneqq \eta_{\ell}^{\rm{int}}(T)^2 + \widetilde{\eta}_{\ell}^{\rm{ext}}(T)^2 + \mathrm{osc}_\ell^{\Omega}(T)^2 + \widetilde{\mathrm{osc}}_\ell^{\rm{D}}(T)^2 + \widetilde{\mathrm{osc}}_\ell^{\rm{N}}(T)^2
 \end{split}
\end{equation}
for all $T \in \widetilde{\TT}_\ell \cup \TT_\ell$.  Furthermore, we denote by
\begin{equation} \label{eq:fullest}
 \widetilde{\eta}_\ell^2 
 \coloneqq \sum_{T \in \TT_\ell \cup \widetilde{\TT}_\ell} \widetilde{\eta}_\ell(T)^2
\end{equation}
the full error estimator.
Due to Theorem~\ref{thm:mainresult} and~\eqref{eq:funcupperbound}, the error estimator $\widetilde{\eta}_\ell$ is reliable in the sense that
\begin{equation*}
 \EE(u_\ell,\phi_\ell) \leq \const{\widetilde{C}}{rel} \widetilde{\eta}_\ell
 \quad \text{with } \const{\widetilde{C}}{rel} \coloneqq \sqrt{5} \max \set{\const{C}{mon}^{-1/2},1,{\const{C}{D}},\pi^{-1} \const{C}{mon}^{-1/2}, {\const{C}{N}\const{C}{mon}^{-1/2}}}. 
\end{equation*}

Based on the error indicators \eqref{eq:indicators} and the D\"orfler criterion~\cite{Dorfler1996}, we propose the following adaptive algorithm, whose performance will be investigated in the subsequent section. 
\begin{algorithm}[Adaptive loop] \label{algo:adap}
\textbf{Input$\colon$} Initial triangulation $\widehat{\TT}_0$ {of $\widehat{\Omega} = \overline{\Omega} \cup \widetilde{\Omega}$}, tolerance $\eps > 0$, polynomial degrees $q,q' \geq 0$ and $p,p' \geq 1$, patch-size $k \geq 1$, and marking parameter $0 < \theta \leq 1$.

For $\ell = 0,1,2,\ldots$ repeat the following steps until $\widetilde{\eta}_\ell^2 < \eps^2 \colon$ 
\begin{enumerate}[label = \rm(\roman*)]
\item Extract the interior mesh $\TT_\ell \coloneqq \widehat{\TT}_\ell|_\Omega$, the boundary mesh $\TT_\ell^\Gamma \coloneqq \widehat{\TT}_\ell|_\Gamma$, and the exterior mesh $\widetilde{\TT}_\ell \coloneqq \TT_\ell|_{\widetilde{\Omega}}$, as well as the patches $\omega_\ell^z$, $\widetilde{\omega}_\ell^{z,k}$.
\item Compute FEM-BEM approximation $(u_\ell,\phi_\ell) \in \SS^p(\TT_\ell) \times \PP^q(\TT_\ell^\Gamma)$ by solving~\eqref{eq:discprob}.
\item Compute the discretized residual $J_{\ell,p'} G_\ell$ as well as $Q_{\ell,q'}^\Gamma \Phi_\ell$ and $Q_{\ell,q'}^\Omega f$.
\item Solve the local auxiliary problems \eqref{eq:dualmixed} and \eqref{eq:wdef} to obtain $\bm{\sigma}_\ell$ and $w_\ell$.
\item Compute $\widetilde{\eta}_\ell(T)$ for $T \in \widehat{\TT}_\ell$ according to \eqref{eq:fullind} and $\widetilde{\eta}_\ell$ according to \eqref{eq:fullest}.
\item Determine a set $\MM_\ell \subseteq \widehat{\TT}_\ell$ of minimal cardinality such that
\begin{equation*}
\theta \widetilde{\eta}_\ell^2 
 \leq \sum_{T \in \MM_\ell} \widetilde{\eta}_\ell(T)^2.
\end{equation*}
\item Employ newest-vertex bisection {to obtain $\widehat{\TT}_{\ell + 1}$ by refining at least all marked elements $T \in \MM_\ell$}.
\end{enumerate}
\end{algorithm}

\begin{remark}
We emphasize that the polynomial degrees $q' \geq 0$ and $p' \geq 1$ are arbitrary and need not coincide with the polynomial degrees $q$ and $p$ of the approximation spaces used to approximate $(u,\phi)$.  Numerical experiments will show that higher-order auxiliary solutions improve the quality and accuracy of the upper bound \eqref{eq:funcupperbound}.  
\end{remark}

\begin{remark}
Algorithm~\ref{algo:adap} reveals that the approximation $\uext_\ell$ is {never} computed explicitly. 
In fact, the algorithm only requires the computation of the interior approximation $u_\ell$ and the boundary density $\phi_\ell$, whereas $\nabla \uext_\ell$ is only computed {in practice once when} $\eta_\ell$ and {hence} $\EE(u_\ell,\phi_\ell)$ are sufficiently small.
\end{remark}

\bmdefine{\bT}{\TT}
\bmdefine{\bl}{\ell}
\subsection{Adaptive algorithm based on Corollary~\ref{cor:corext} (based on $\texorpdfstring{\boldsymbol{\TT}\!\!_{\boldsymbol{\ell}}}{}$ only)} \label{subsec:adaptiveint}

We proceed analogously to Section~\ref{subsec:adaptiveext}, but consider $\widehat{\Omega} \coloneqq \Omega$, $\widehat{\TT}_\ell \coloneqq \TT_\ell$, and compute $w_\ell \in \SS^{p' + 1}(\TT_\ell) \subset H^1(\Omega)$ as outlined in~\eqref{eq:wintdef} instead of~\eqref{eq:wdef}. 
We define the local error indicators
\begin{equation} \label{eq:intcontributions}
 \begin{split}
  \eta_\ell^{\mathrm{ext}}(T) 
  &\coloneqq \norm{\nabla w_\ell}_T
  \quad \text{for } T \in \TT_\ell, \\
  \mathrm{osc}_\ell^{\mathrm{D}}(T) 
  &\coloneqq h_T^{1/2} \norm{\nabla_\Gamma (1 - J_{\ell,p'})G_\ell}_{\partial T \cap \Gamma} \quad \text{for } T \in \TT_\ell, \\
  \mathrm{osc}_\ell^{\mathrm{N}}(T) 
  &\coloneqq h_T^{1/2} \norm{(1 - Q_{\ell,q'}^\Gamma)\Phi_\ell}_{\partial T \cap \Gamma} \quad \text{for } T \in \TT_\ell.
 \end{split}
\end{equation}
With $\eta_\ell^{\mathrm{int}}$ and $\mathrm{osc}_\ell^\Omega$ from~\eqref{eq:indicators} we define
\begin{equation} \label{eq:intindicators}
 \eta_\ell(T)^2
 \coloneqq \eta_{\ell}^{\rm{int}}(T)^2 + \eta_{\ell}^{\rm{ext}}(T)^2 + \mathrm{osc}_\ell^{\Omega}(T)^2 + \mathrm{osc}_\ell^{\rm{D}}(T)^2 + \mathrm{osc}_\ell^{\rm{N}}(T)^2
 \quad \text{for all } T \in \TT_\ell.
\end{equation}
Due to Corollary~\ref{cor:corext} and~\eqref{eq:funcupperboundalt2}, the error estimator
\begin{equation} \label{eq:fullestalt}
  \eta_\ell^2 
  \coloneqq \sum_{T \in \TT_\ell} \eta_{\ell}(T)^2
\end{equation} 
is reliable in the sense that
\begin{equation*}
 \EE(u_\ell,\phi_\ell) \leq \const{C}{rel} \eta_\ell
 \quad \text{with } \const{C}{rel} \coloneqq \sqrt{5} \max \set{\const{C}{mon}^{-1/2},\const{C}{PS}^{\mathrm{ext}} (\const{c}{PS}^{\mathrm{int}})^{-1},{\const{C}{D}},\pi^{-1} \const{C}{mon}^{-1/2}, {\const{C}{N}\const{C}{mon}^{-1/2}}}. 
\end{equation*}
Based on the indicators~\eqref{eq:intindicators} and the D\"orfler marking criterion~\cite{Dorfler1996}, we may consider the following adaptive algorithm as an alternative to Algorithm~\ref{algo:adap}, which only uses the interior mesh $\TT_\ell$.
\begin{algorithm}[Adaptive loop based on $\boldsymbol{\TT}\!\!_{\boldsymbol{\ell}}$ only] \label{algo:adapint}
\textbf{Input$\colon$} Initial triangulation $\TT_0$ of $\Omega$, tolerance $\eps > 0$, polynomial degrees $q,q' \geq 0$ and $p,p' \geq 1$, patch-size $k \geq 1$, and marking parameter $0 < \theta \leq 1$.
For $\ell = 0,1,2,\ldots$ repeat the following steps until $\eta_\ell^2 < \eps^2 \colon$
\begin{enumerate}[label = \rm(\roman*)]
\item Extract the boundary mesh $\TT_\ell^\Gamma \coloneqq \TT_\ell|_\Gamma$ as well as the patches $\omega_\ell^z$ and $\omega_\ell^{z,k}$.
\item Compute FEM-BEM approximation $(u_\ell,\phi_\ell) \in \SS^p(\TT_\ell) \times \PP^q(\TT_\ell^\Gamma)$ by solving~\eqref{eq:discprob}.
\item Compute the discretized residual $J_{\ell,p'} G_\ell$ as well as $Q_{\ell,q'}^\Gamma \Phi_\ell$ and $Q_{\ell,q'}^\Omega f$.
\item Solve the local auxiliary problems \eqref{eq:dualmixed} and \eqref{eq:wintdef} to obtain $\bm{\sigma}_\ell$ and $w_\ell$.
\item Compute $\eta_\ell(T)$ for $T \in \TT_\ell$ according to \eqref{eq:intindicators} and $\eta_\ell$ according to \eqref{eq:fullestalt}.
\item Determine a set $\MM_\ell \subseteq \TT_\ell$ of minimal cardinality such that
\begin{equation*}
 \theta \eta_\ell^2 
 \leq \sum_{T \in \MM_\ell} \eta_\ell(T)^2.
\end{equation*}
\item Employ newest-vertex bisection {to obtain $\TT_{\ell + 1}$ by refining at least all marked elements $T \in \MM_\ell$}.
\end{enumerate}
\end{algorithm}

\section{Numerical experiments} \label{sec:numerics}

This section presents some numerical experiments in 2D that illustrate the performance and accuracy of Algorithm~\ref{algo:adap} and Algorithm~\ref{algo:adapint}. 
All experiments are carried out by the MATLAB-toolbox HILBERT from~\cite{Aurada2014} for the lowest-order FEM-BEM coupling with $p = 1$ and $q = 0$.
We consider $p' = 1$, $q' = 0$, and $k = 2$ for the auxiliary problems in Section~\ref{subsec:flux}--\ref{subsec:potential}.
Throughout, we consider and experimentally compare Algorithm~\ref{algo:adap} and Algorithm~\ref{algo:adapint} for uniform ($\theta = 1$) and adaptive ($0 < \theta < 1$) mesh refinement. All involved domains $\Omega \subset \R^2$ satisfy the scaling condition $\diam(\Omega) < 1$ that ensures ellipticity of the 2D single-layer integral operator $V$.

\subsection{Example 1 (Square domain, linear diffusion)} \label{subsec:squareex}

We consider the transmission problem~\eqref{eq:nltransmission} with linear diffusion $\AAA (x,y) \coloneqq (x/100,100y)$ and constant data $(f,\const{g}{D},\const{g}{N}) \coloneqq (1,0,0)$ on the square domain $\Omega \coloneqq (-1/4,1/4)^2$. Note that the monotonicity constant of $\AAA$ is $\const{C}{mon} = 1/100$, while the Lipschitz constant is $\const{C}{Lip} = 100$.

We start Algorithm~\ref{algo:adap} with a uniform initial triangulation $\widehat{\TT}_0$ comprised of $64$ rectangular triangles, where $\TT_0 = \widehat{\TT}_0|_\Omega$ consists of $16$ triangles and $\TT_0^\Gamma = \widehat{\TT}_\ell|_\Gamma$ consists of $8$ edges. Figure~\ref{fig:squaremeshesext} shows the initial mesh, some adaptively refined meshes, and the corresponding union of the boundary patches $\widetilde{\omega}_\ell^{z,k}$ (see Section~\ref{subsec:potential}) generated by Algorithm~\ref{algo:adap} for $k = 2$ and $\theta = 0.4$. For comparison, Figure~\ref{fig:squaremeshesint} depicts some meshes and patches generated by Algorithm~\ref{algo:adapint}, where we employ the same initial meshes $\TT_0$ and $\TT_0^\Gamma$ as for Algorithm~\ref{algo:adap}.

\begin{figure}[!ht]
 \resizebox{\textwidth}{!}{
   \subfloat{
    \stackunder{\stackunder[10pt]{\includegraphics{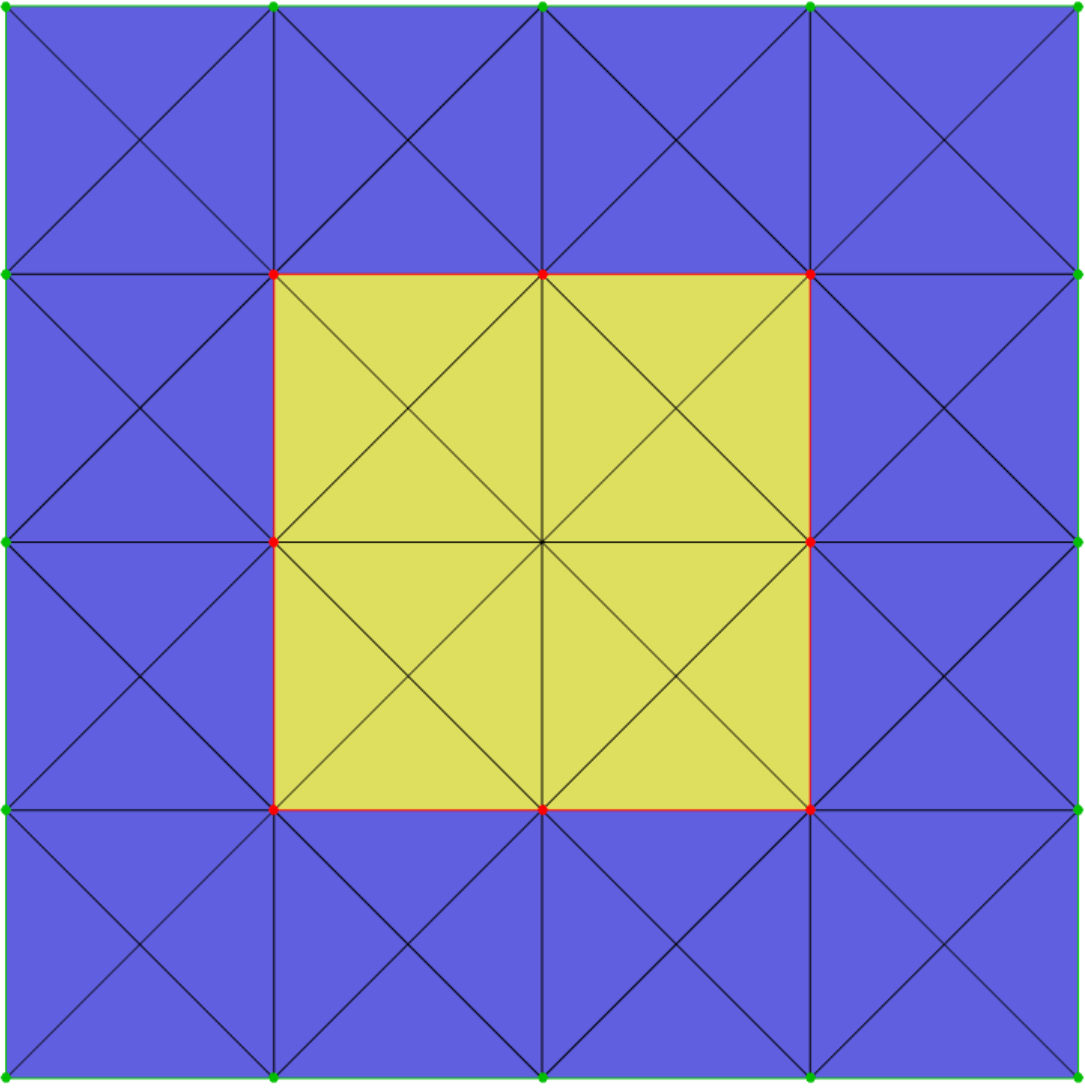}}{{\Huge$\# \TT_\ell = 16$, $\#\widehat{\TT_\ell} = 64$}}}{{\Huge $\ell=0$}}
   }
   \subfloat{
    \stackunder{\stackunder[10pt]{\includegraphics{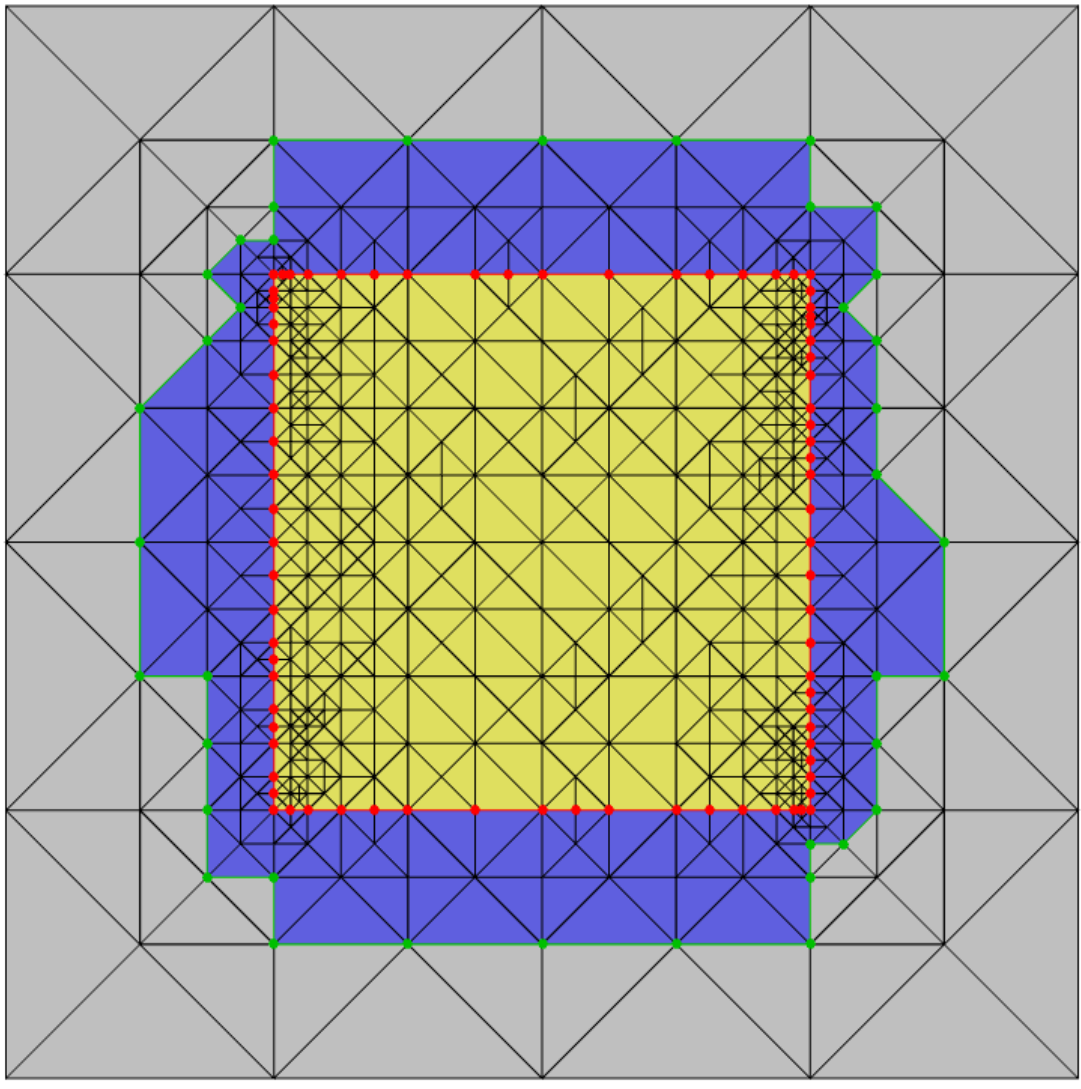}}{{\Huge $\# \TT_\ell = 642$, $\#\widehat{\TT_\ell} = 1014$}}}{{\Huge $\ell=20$}}
   }
   
   \subfloat{
    \stackunder{\stackunder[10pt]{\includegraphics{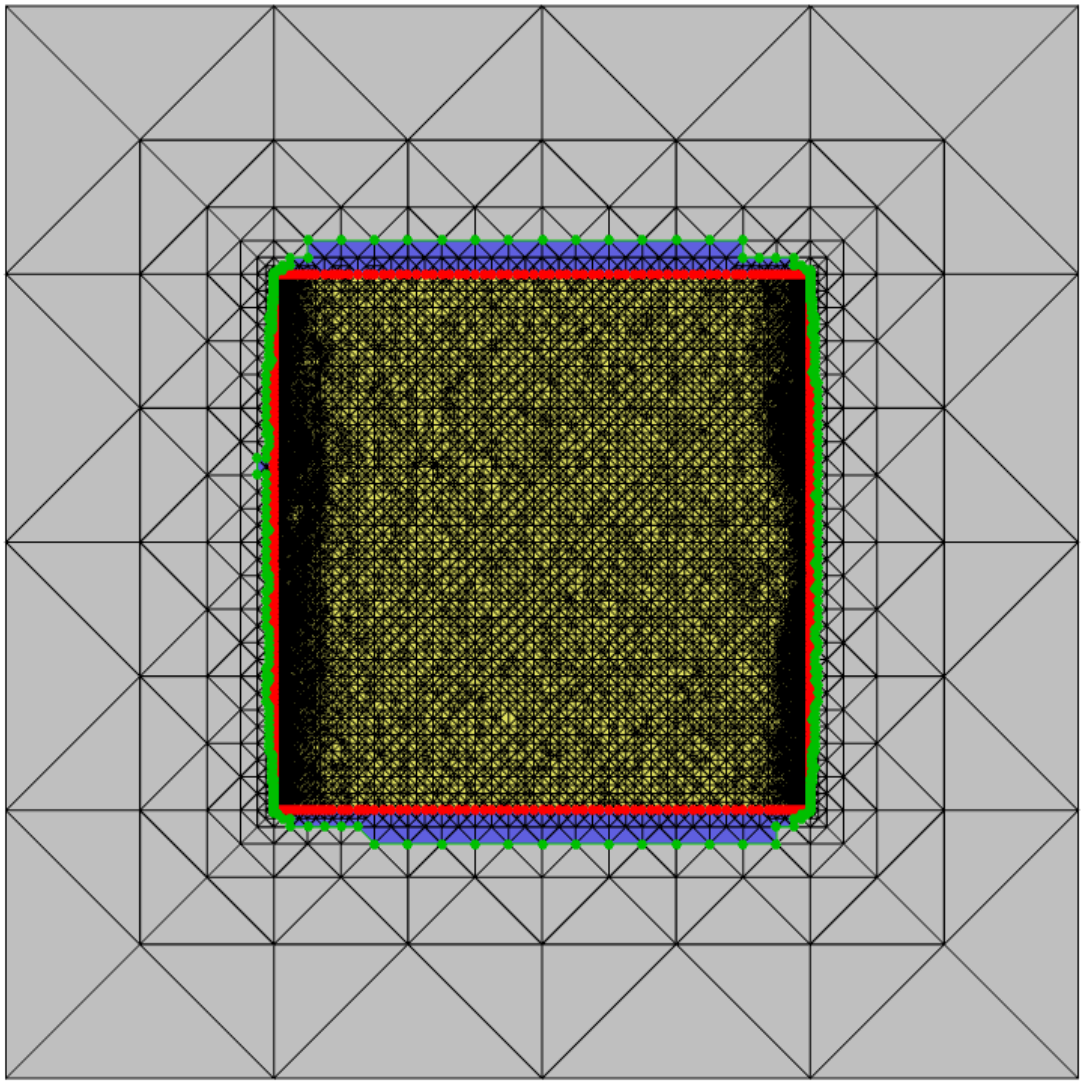}}{{\Huge $\# \TT_\ell = 55373$, $\#\widehat{\TT_\ell} = 62548$}}}{{\Huge $\ell=50$}}
   }
  }
   \caption{Meshes generated by Algorithm~\ref{algo:adap} in Example~\ref{subsec:squareex} for $\theta = 0.4$ and $k = 2$. Only triangles not colored in gray contribute to the computation of the error indicators~\eqref{eq:indicators}. The triangles of the patches $\widetilde{\omega}_\ell^{z,k}$ are depicted in {\color{blue} blue}, whereas the triangles in $\Omega$ are indicated in {\color{yellow} yellow} and the remaining triangles in $\Omegaext$ in {\color{gray} gray}. The inner boundary $\Gamma$ is shown in {\color{red} red} and the outer boundary of the union of the patches in {\color{green} green}.}
    \label{fig:squaremeshesext}
\end{figure}

\begin{figure}[!ht]
 \resizebox{\textwidth}{!}{
   \subfloat{
    \stackunder[10pt]{\includegraphics{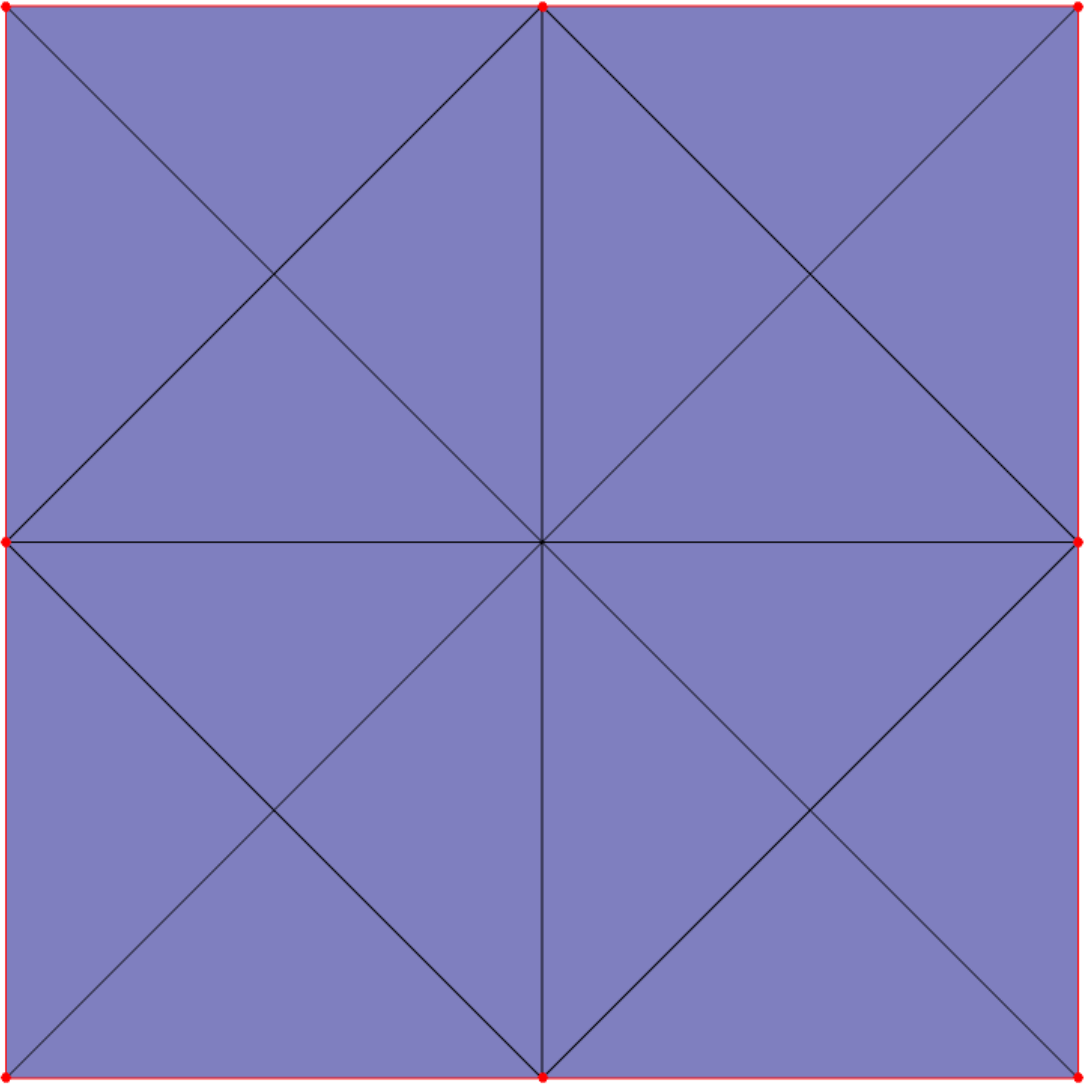}}{{\Huge$\# \TT_\ell = 16$, $\ell = 0$}}
   }
   \subfloat{
    \stackunder[10pt]{\includegraphics{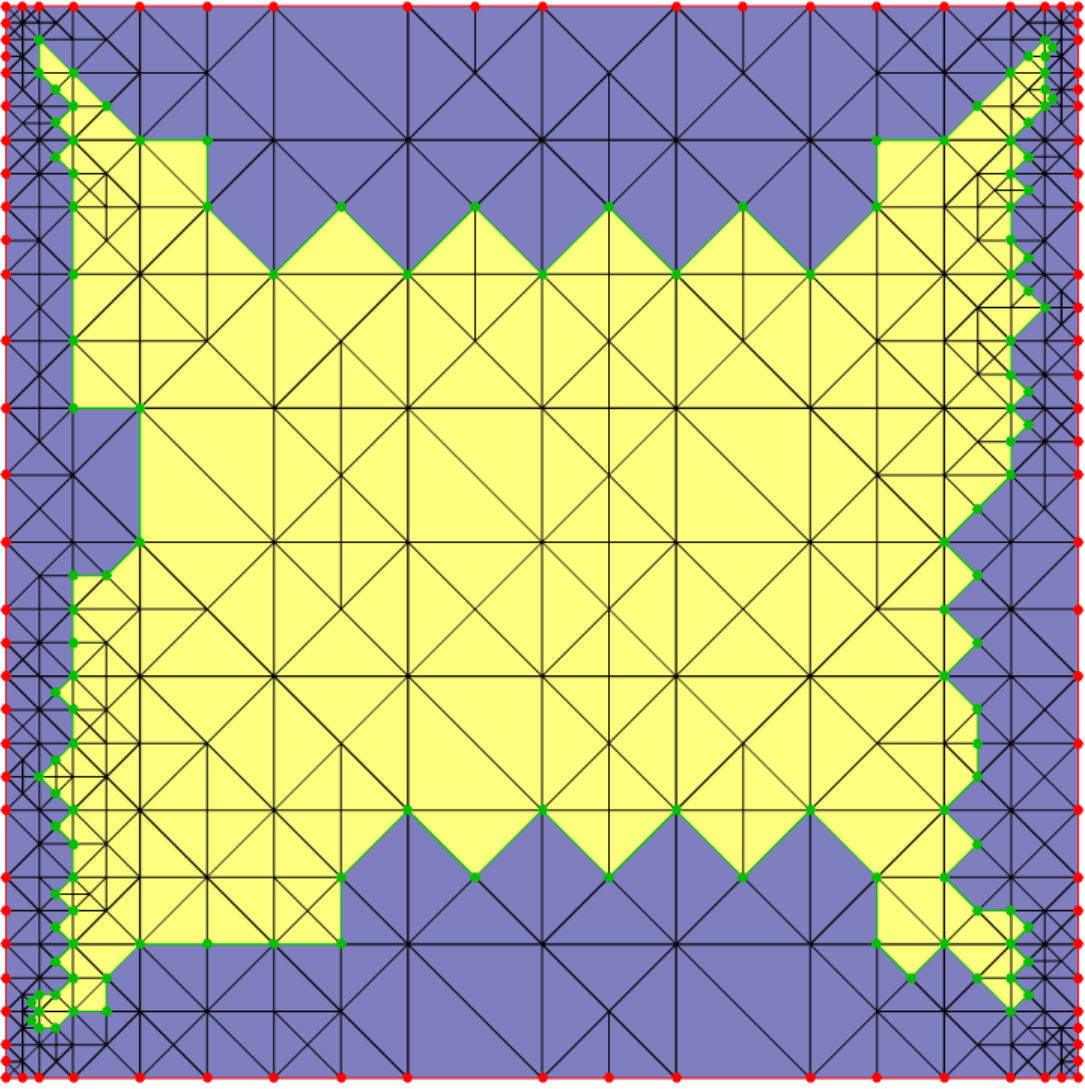}}{{\Huge $\# \TT_\ell = 830$, $\ell = 20$}}
   }
   
   \subfloat{
   \stackunder[10pt]{\includegraphics{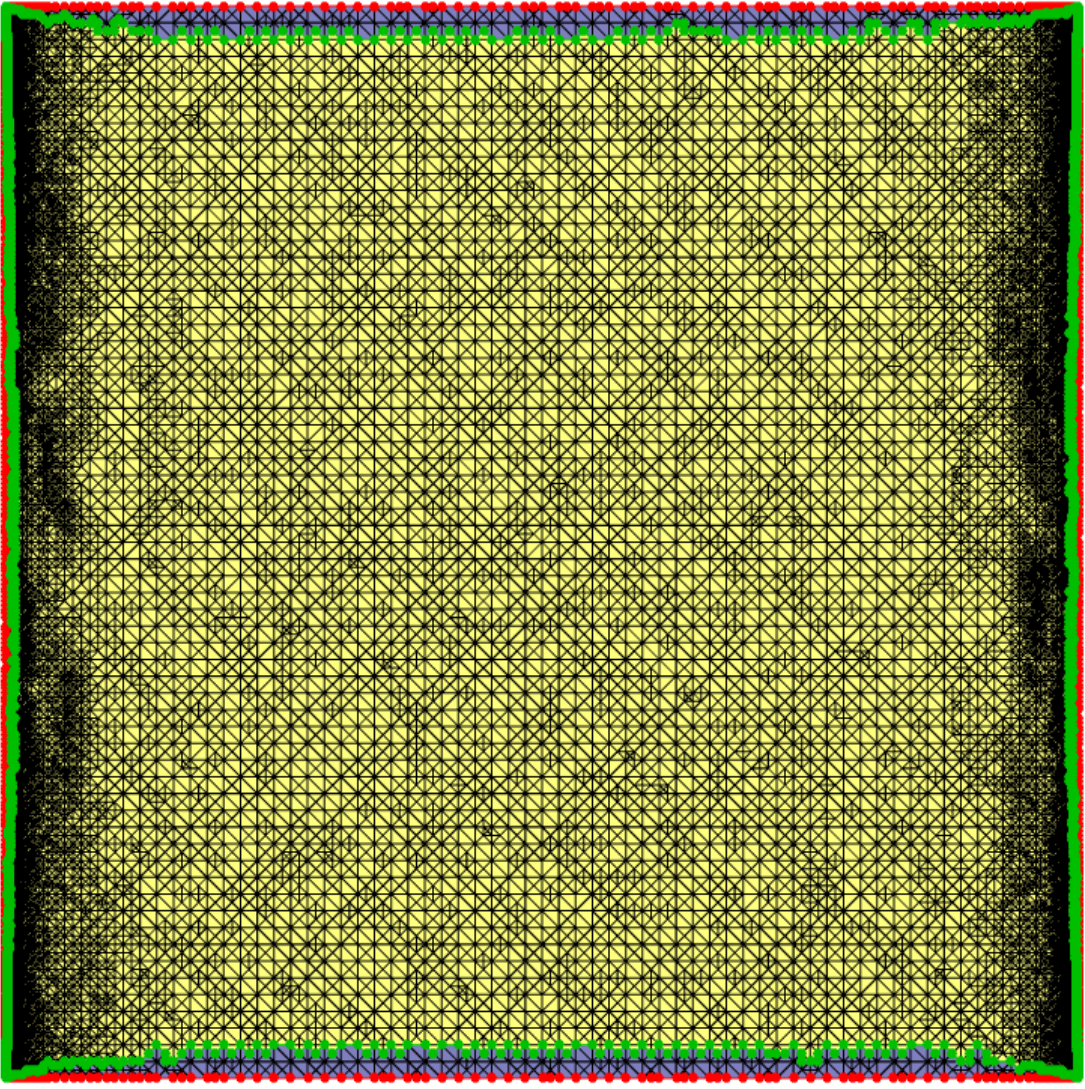}}{{\Huge $\# \TT_\ell = 64594$, $\ell = 50$}}
   }
  }
   \caption{Meshes generated by Algorithm~\ref{algo:adapint} in Example~\ref{subsec:squareex} for $\theta = 0.4$ and $k = 2$. The triangles of the patches $\omega_\ell^{z,k}$ are depicted in {\color{blue} blue}, while the remaining triangles are indicated in {\color{yellow} yellow}. The outer boundary $\Gamma$ is shown in {\color{red} red} and the inner boundary of the union of the patches in {\color{green} green}.}
    \label{fig:squaremeshesint}
\end{figure}
The total upper bound $\widetilde{\eta}_\ell$ from~\eqref{eq:fullest} generated by Algorithm~\ref{algo:adap} with $k = 2$ and different marking parameters $0 < \theta \leq 1$ is shown in Figure~\ref{fig:squareratesext}~(left). We observe that Algorithm~\ref{algo:adap} exhibits optimal convergence rates $\EE(u_\ell,\phi_\ell) = \OO(\widetilde{\eta}_\ell) = \OO((\# \TT_\ell)^{-1/2})$ for any marking parameter $\theta \in \set{0.2,0.4,0.6,0.8}$ with only minimal differences with respect to the particular choice of $\theta$, while uniform refinement leads to suboptimal convergence. Figure~\ref{fig:squareratesext}~(right) depicts the different contributions $\eta_\ell^{\rm{int}}$, $\widetilde{\eta}_\ell^{\rm{ext}}$, $\widetilde{\mathrm{osc}}_\ell^{\rm{D}}$, and $\widetilde{\mathrm{osc}}_\ell^{\rm{N}}$ to the total error estimator $\widetilde{\eta}_\ell$ for $\theta = 0.4$ (note that $\mathrm{osc}_\ell^\Omega = 0$). 
\begin{figure}[!ht]
  \resizebox{\textwidth}{!}{
   \subfloat{
    \includegraphics{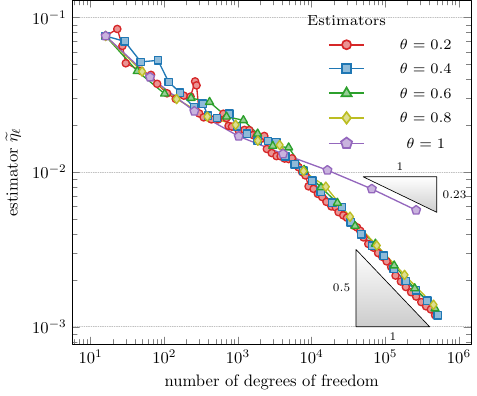}
   }
   \subfloat{
    \includegraphics{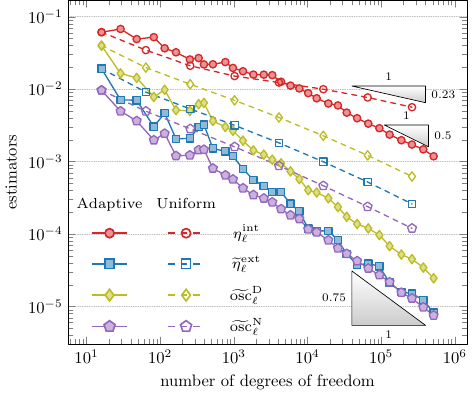}
   }
  }
   \caption{Convergence rates of the full error estimator $\widetilde{\eta}_\ell$ from~\eqref{eq:fullest} in Algorithm~\ref{algo:adap} for $k = 2$ in Example~\ref{subsec:squareex} for different marking parameters $0 < \theta \leq 1$ (left). Comparison of the different contributions $\eta_\ell^{\mathrm{int}}$, $\widetilde{\eta}_\ell^{\mathrm{ext}}$, $\widetilde{\mathrm{osc}}_\ell^{\mathrm{D}}$, and $\widetilde{\mathrm{osc}}_\ell^{\mathrm{N}}$ of $\widetilde{\eta}_\ell$ for $\theta = 0.4$ (right).}
   \label{fig:squareratesext}
 \end{figure}
As usual for FEM-BEM computations, we observe individual optimal rates, i.e., rate $1/2$ for the FEM part and $3/4$ for the BEM part; see~\cite{Melenk2014} for a rigorous mathematical argument for smooth solutions on uniform meshes. Similarly to Figure~\ref{fig:squareratesext}, Figure~\ref{fig:squareratesint} depicts the total upper bound $\eta_\ell$ from~\eqref{eq:fullestalt} generated by Algorithm~\ref{algo:adapint} with $k = 2$ and different marking parameters $0 < \theta \leq 1$ (left), as well as the different contributions $\eta_\ell^{\rm{int}}$, $\eta_\ell^{\rm{ext}}$, $\mathrm{osc}_\ell^{\rm{D}}$, and $\mathrm{osc}_\ell^{\rm{N}}$ to the total error estimator $\eta_\ell$ for $\theta = 0.4$ (right). We again observe optimal convergence rates for any marking parameter $\theta \in \set{0.2,0.4,0.6,0.8}$, while uniform refinement leads to suboptimal convergence.
\begin{figure}[!ht]
  \resizebox{\textwidth}{!}{
   \subfloat{
    \includegraphics{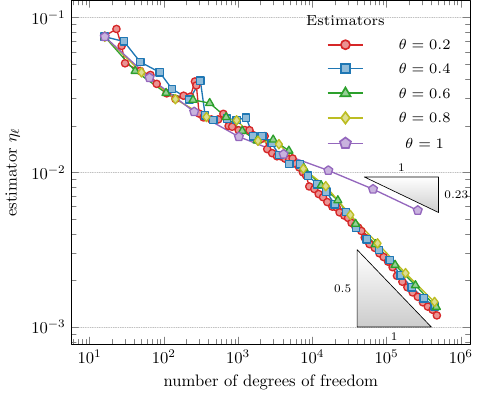}
   }
   \subfloat{
    \includegraphics{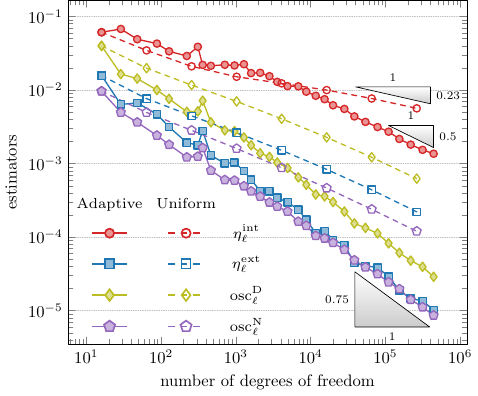}
   }
  }
   \caption{Convergence rates of the full error estimator $\eta_\ell$ from~\eqref{eq:fullestalt} in Algorithm~\ref{algo:adapint} for $k = 2$ in Example~\ref{subsec:squareex} for different marking parameters $0 < \theta \leq 1$ (left). Comparison of the different contributions $\eta_\ell^{\mathrm{int}}$, $\eta_\ell^{\mathrm{ext}}$, $\mathrm{osc}_\ell^{\mathrm{D}}$, and $\mathrm{osc}_\ell^{\mathrm{N}}$ of $\eta_\ell$ for $\theta = 0.4$ (right).}
   \label{fig:squareratesint}
 \end{figure}

Figure~\ref{fig:squareintext} shows a comparison between Algorithm~\ref{algo:adap} and Algorithm~\ref{algo:adapint}. We observe that both approaches lead to similar results, while the computation in the interior domain renders the need for the exterior mesh $\widetilde{\TT}_\ell$ obsolete. This favors the approach of Corollary~\ref{cor:corext} and Algorithm~\ref{algo:adapint} in practice.
\begin{figure}[!ht]
  \resizebox{\textwidth}{!}{
    \includegraphics[scale = 0.1]{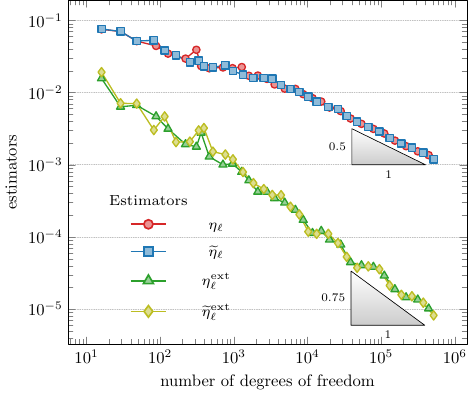}
  }
   \caption{Comparison between the estimators $\eta_\ell, \eta_\ell^{\mathrm{ext}}$ from~\eqref{eq:intindicators}--\eqref{eq:fullestalt} in Algorithm~\ref{algo:adapint} and $\widetilde{\eta}_\ell$, $\widetilde{\eta}_\ell^{\mathrm{ext}}$ from~\eqref{eq:indicators}--\eqref{eq:fullest} in Algorithm~\ref{algo:adap} for Example~\ref{subsec:squareex} with $k = 2$ and $\theta = 0.4$.}
   \label{fig:squareintext}
 \end{figure}

While \eqref{eq:funcupperbound} measures the error between $(u_\ell,\phi_\ell)$ and the exact solution $(u,\phi)$ of the variational formulation of the symmetric Costabel--Han coupling~\cite{Costabel1987,Han1990}, Algorithm~\ref{algo:adap}--\ref{algo:adapint} can also be seen as a way to approximate the functions $(u_{\rm{JN}}, u^{\mathrm{ext}}_{\rm{JN}})$ and $(u_{\rm{BM}},u^{\mathrm{ext}}_{\rm{BM}})$ extracted from the unique solutions $(u_{\rm{JN}},\phi_{\rm{JN}})$ and $(u_{\rm{BM}},\phi_{\rm{BM}})$ of the Johnson--Nédélec coupling (see \cite{Johnson1980}) and the Bielak--MacCamy coupling (see \cite{Bielak1995}), respectively; see also~\cite{Aurada2012a} for a concise presentation of all these coupling procedures. 
There holds
\begin{equation*}
 (u,\phi) 
 = (u_{\rm{JN}},\phi^{\mathrm{ext}}_{\rm{JN}}) 
 = (u_{\rm{BM}},(K'-1/2)\phi_{\rm{BM}})
\end{equation*}
and 
\begin{equation*}
 (u,u^{\mathrm{ext}}) 
 = (u_{\rm{JN}},u^{\mathrm{ext}}_{\rm{JN}}) 
 = (u_{\rm{BM}},u^{\mathrm{ext}}_{\rm{BM}}).
\end{equation*}
We stress that existence and uniqueness of Galerkin solutions to the Johnson--Nédélec coupling and the Bielak--MacCamy coupling require additional mild conditions on the ellipticity constant of $\AAA$; see, e.g.,~\cite{Aurada2012a,Ferrari2023}.

Figure~\ref{fig:squarejnbm} shows the full error estimator $\eta_\ell$ from~\eqref{eq:fullestalt} and its contributions $\eta_\ell^{\mathrm{int}}$, $\eta_\ell^{\mathrm{ext}}$, $\mathrm{osc}_\ell^{\mathrm{D}}$, and $\mathrm{osc}_\ell^{\mathrm{N}}$ generated by Algorithm~\ref{algo:adapint} for adaptive refinement ($\theta = 0.4$) with respect to the Galerkin approximations of the Johnson--Nédélec coupling (left) and the Bielak--MacCamy coupling (right). 
We observe that Algorithm~\ref{algo:adapint} exhibits optimal convergence rates also for these couplings.
\begin{figure}[!ht]
  \resizebox{\textwidth}{!}{
   \subfloat{
    \includegraphics{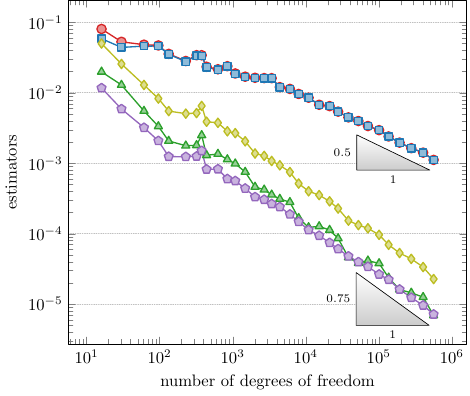}
   }
   \subfloat{
    \includegraphics{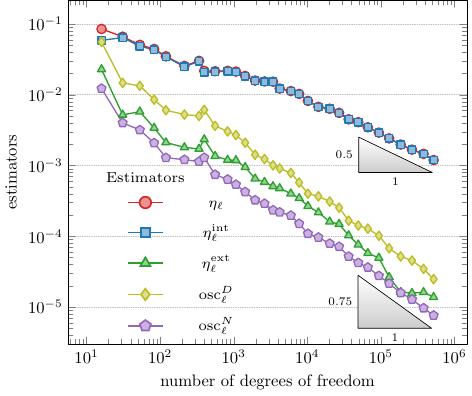}
   }
  }
   \caption{Comparison of the different contributions $\eta_\ell^{\mathrm{int}}$, $\eta_\ell^{\mathrm{ext}}$, $\mathrm{osc}_\ell^{\mathrm{D}}$, and $\mathrm{osc}_\ell^{\mathrm{N}}$ of $\eta_\ell$ from~\eqref{eq:fullestalt} in Algorithm~\ref{algo:adapint} for Example~\ref{subsec:squareex} with $\theta = 0.4$ for the Johnson--Nédélec coupling (left) and the Bielak--MacCamy coupling (right).}
   \label{fig:squarejnbm}
 \end{figure}

 \subsection{Example 2 (L-shaped domain, linear diffusion)} \label{subsec:lshapeex}

We consider the transmission problem~\eqref{eq:nltransmission} with linear diffusion $\AAA (x,y) \coloneqq (x,y)$ on the scaled and rotated L-shaped domain; see Figure~\ref{fig:lmeshesext} and Figure~\ref{fig:lmeshesint}.
 We prescribe the exact solution by 
\begin{equation*}
 \begin{split}
    u(x,y) 
    = r^{2/3} \sin(2\varphi/3)
    \quad \text{and} \quad 
    u^{\mathrm{ext}}(x,y)
    = 0,
 \end{split}
\end{equation*}
where $(r,\varphi)$ are the polar coordinates centered at the reentrant corner of the L-shaped domain. Unlike Example~\ref{subsec:squareex}, the solution $u$ exhibits a singularity at the reentrant corner $(0,0)$, while $u^{\mathrm{ext}}$ is smooth. Figure~\ref{fig:lmeshesext} and Figure~\ref{fig:lmeshesint} show the initial mesh and some adaptively refined meshes obtained by Algorithm~\ref{algo:adap} and Algorithm~\ref{algo:adapint}, respectively, with $k = 2$ and $\theta = 0.4$, which show that both Algorithms resolve the singularity adequately.

\begin{figure}[!ht]
 \resizebox{\textwidth}{!}{
   \subfloat{
    \stackunder{\stackunder[10pt]{\includegraphics{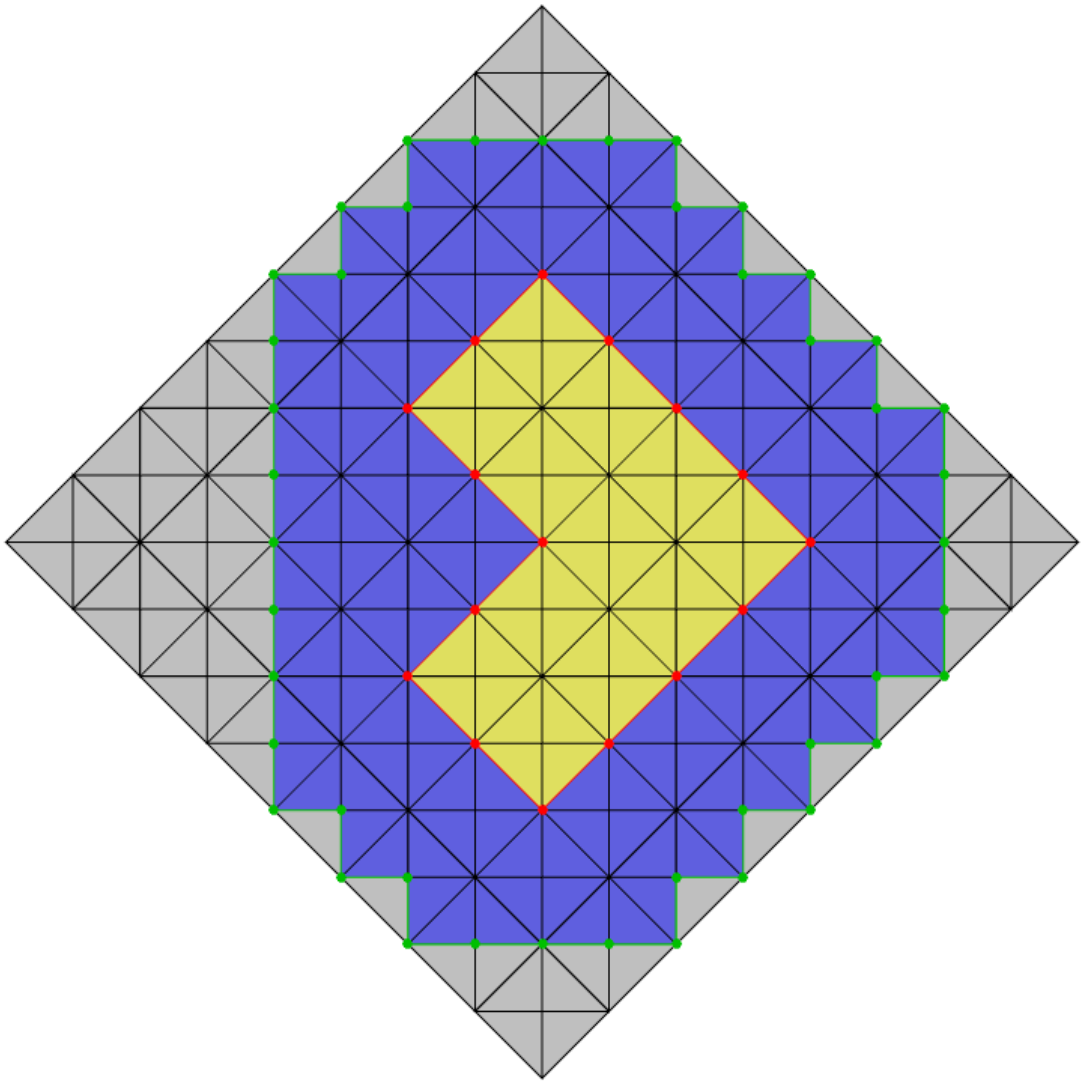}}{{\Huge$\# \TT_\ell = 48$, $\#\widehat{\TT_\ell} = 256$}}}{{\Huge $\ell=0$}}
   }
   \subfloat{
    \stackunder{\stackunder[10pt]{\includegraphics{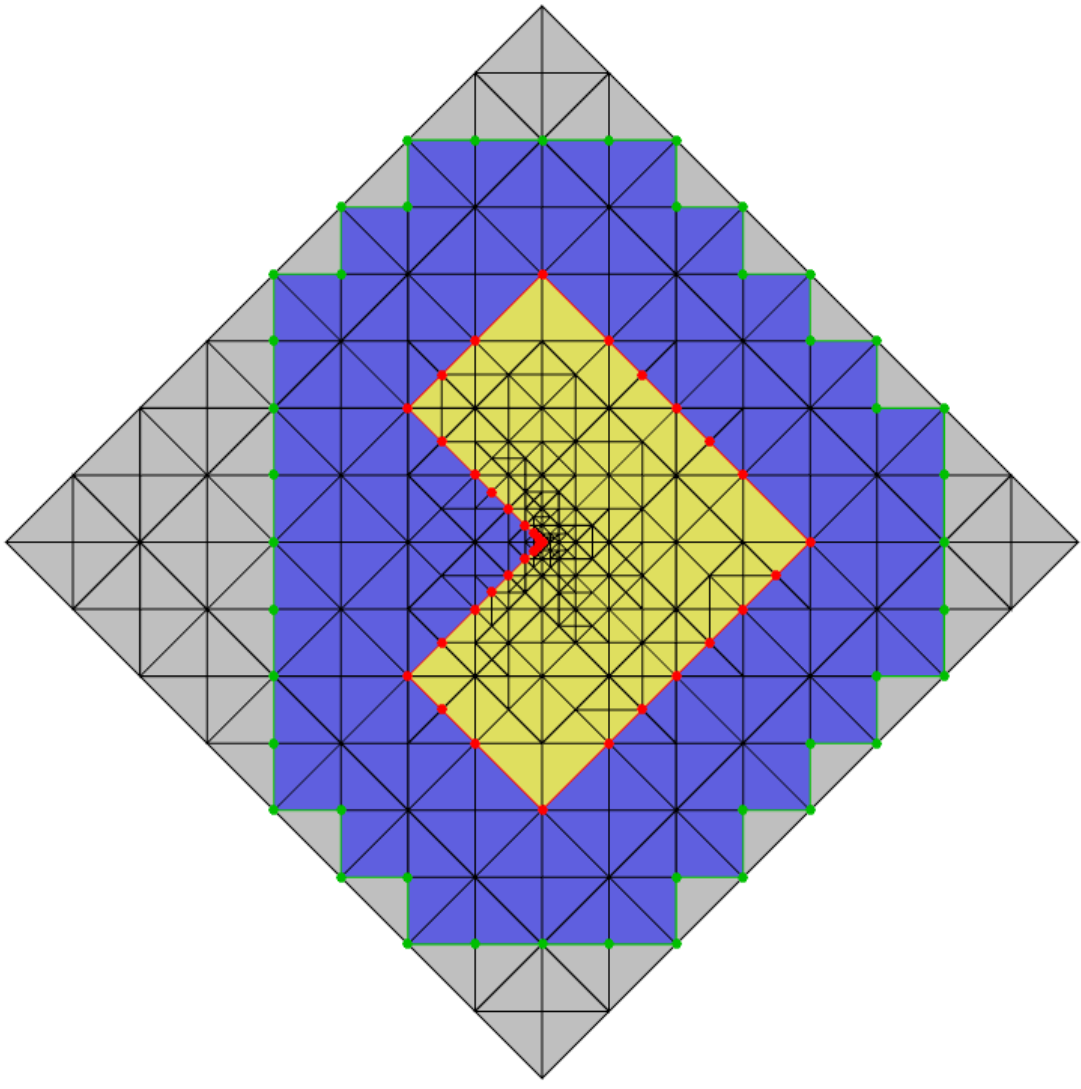}}{{\Huge $\# \TT_\ell = 296$, $\#\widehat{\TT_\ell} = 540$}}}{{\Huge $\ell=10$}}
   }
   
   \subfloat{
    \stackunder{\stackunder[10pt]{\includegraphics{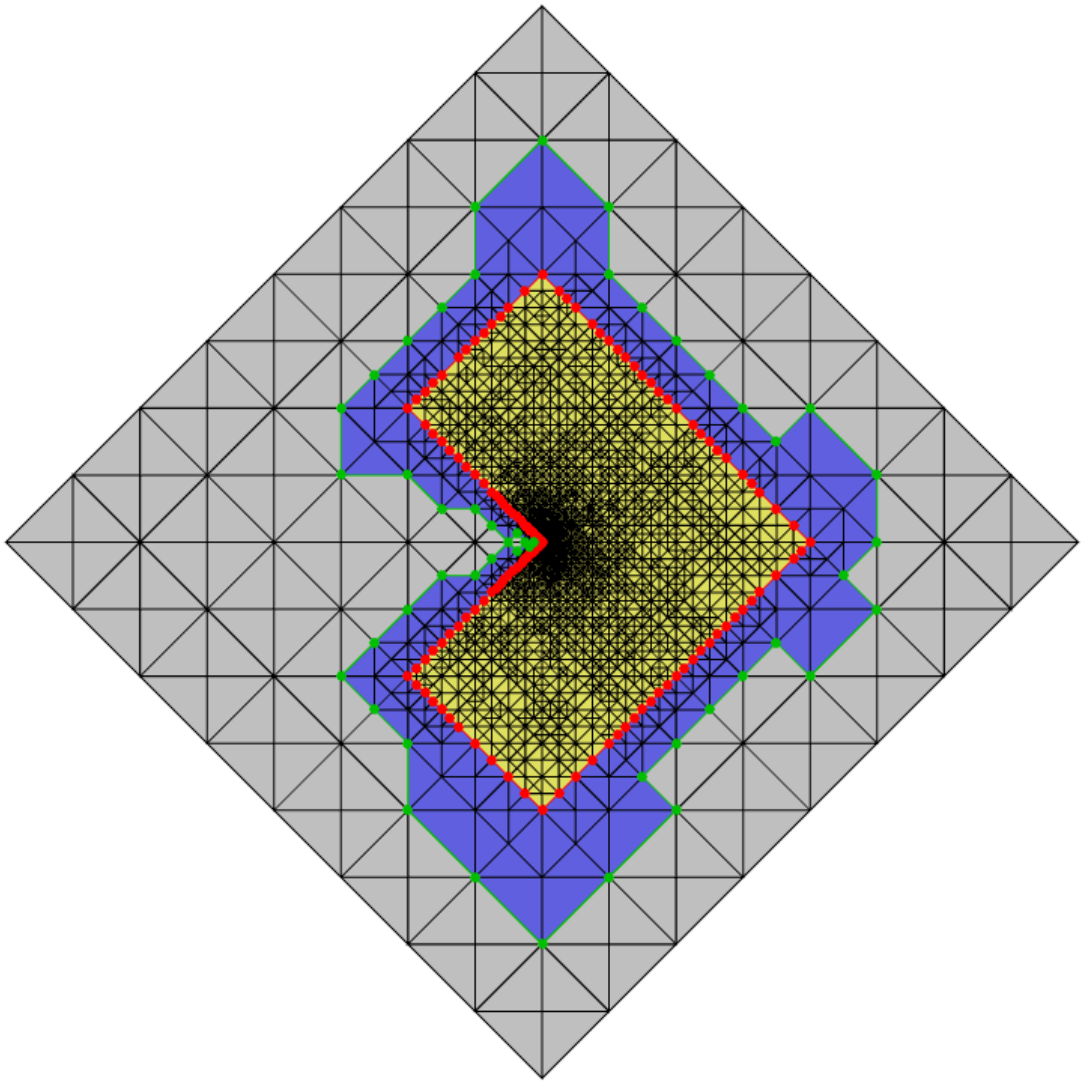}}{{\Huge $\# \TT_\ell = 4415$, $\#\widehat{\TT_\ell} = 5096$}}}{{\Huge $\ell=20$}}
   }
  }
   \caption{Meshes generated by Algorithm~\ref{algo:adap} in Example~\ref{subsec:lshapeex} for $\theta = 0.4$ and $k = 2$. Only triangles not colored in gray contribute to the computation of the error indicators~\eqref{eq:indicators}. The triangles of the patches $\widetilde{\omega}_\ell^{z,k}$ are depicted in {\color{blue} blue}, whereas the triangles in $\Omega$ are indicated in {\color{yellow} yellow} and the remaining triangles in $\Omegaext$ in {\color{gray} gray}. The inner boundary $\Gamma$ is shown in {\color{red} red} and the outer boundary of the union of the patches in {\color{green} green}.}
    \label{fig:lmeshesext}
\end{figure}

\begin{figure}[!ht]
 \resizebox{\textwidth}{!}{
   \subfloat{
    \stackunder[10pt]{\includegraphics{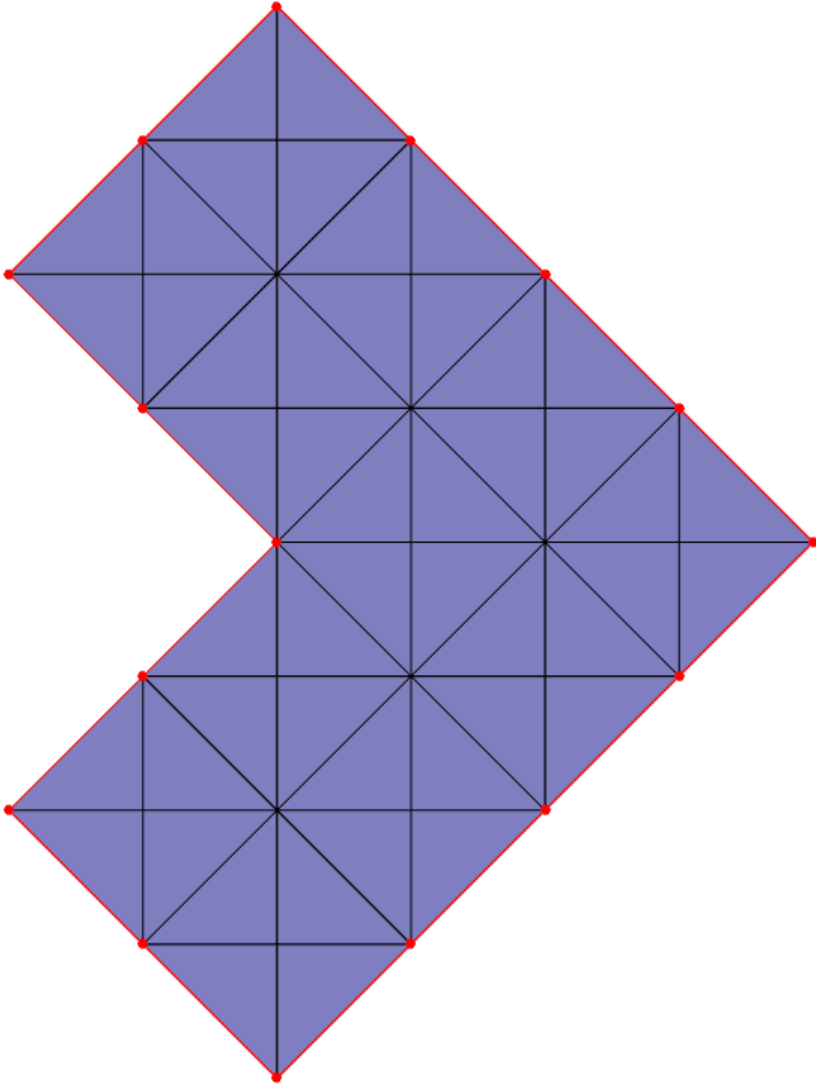}}{{\Huge$\# \TT_\ell = 48$, $\ell = 0$}}
   }
   \subfloat{
    \stackunder[10pt]{\includegraphics{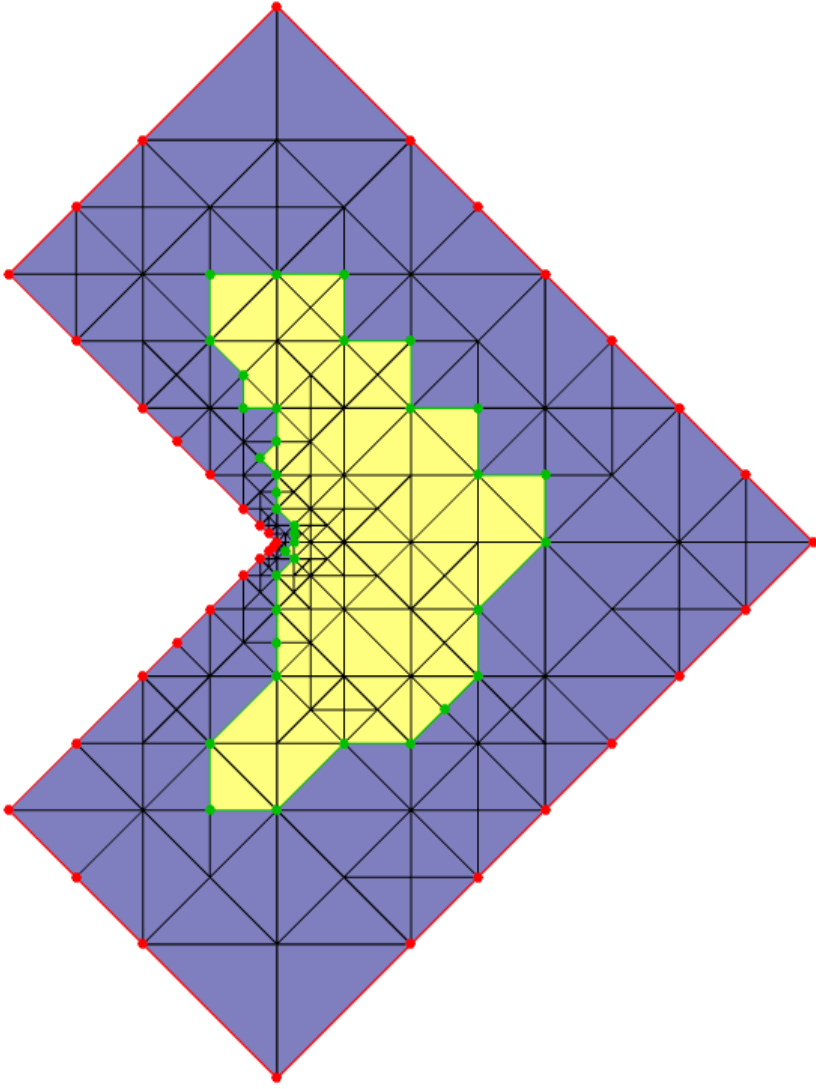}}{{\Huge $\# \TT_\ell = 295$, $\ell = 20$}}
   }
   
   \subfloat{
   \stackunder[10pt]{\includegraphics{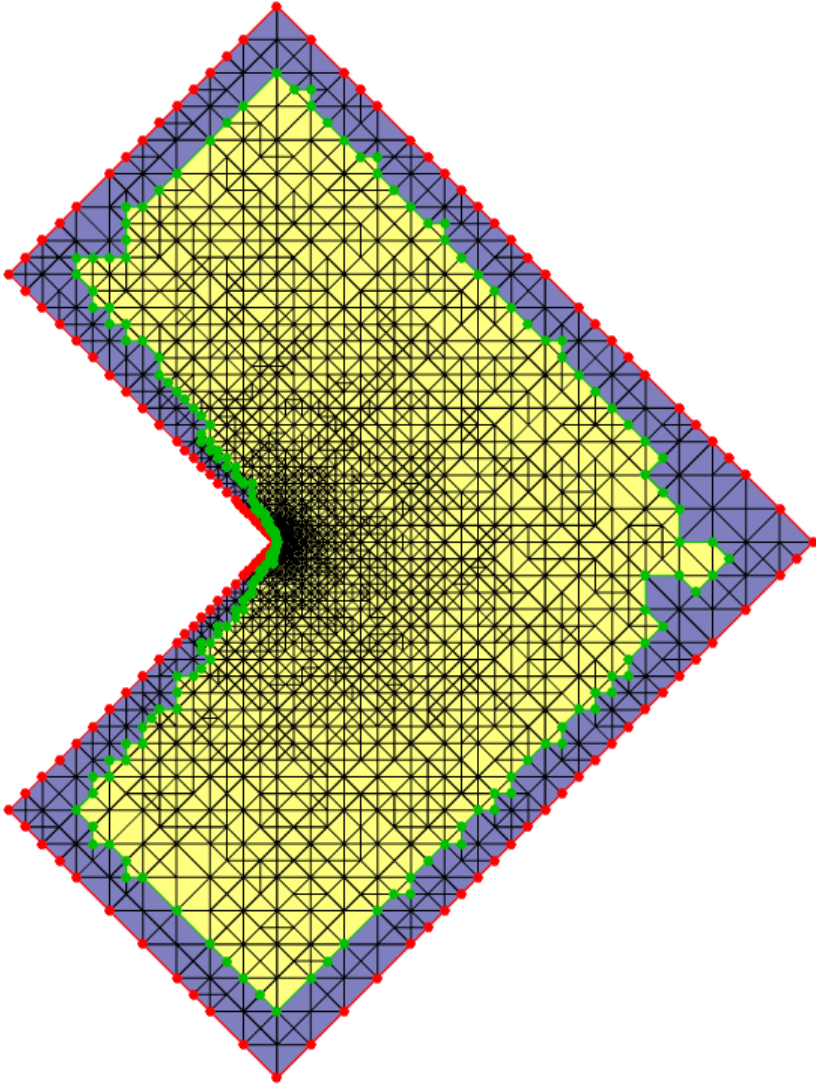}}{{\Huge $\# \TT_\ell = 4600$, $\ell = 20$}}
   }
  }
   \caption{Meshes generated by Algorithm~\ref{algo:adapint} in Example~\ref{subsec:lshapeex} for $\theta = 0.4$ and $k = 2$. The triangles of the patches $\omega_\ell^{z,k}$ are depicted in {\color{blue} blue}, while the remaining triangles are indicated in {\color{yellow} yellow}. The outer boundary $\Gamma$ is shown in {\color{red} red} and the inner boundary of the union of the patches in {\color{green} green}.}
    \label{fig:lmeshesint}
\end{figure}

The total upper bound $\eta_\ell$ from~\eqref{eq:fullestalt} with $k = 2$ and different marking parameters $0 < \theta \leq 1$ is shown in Figure~\ref{fig:lrates}~(left). We observe that Algorithm~\ref{algo:adap} exhibits optimal convergence rates for any marking parameter $\theta \in \set{0.2,0.4,0.6,0.8}$, whereas uniform refinement ($\theta = 1$) leads to suboptimal convergence. 
Figure~\ref{fig:lrates}~(right) depicts the different contributions $\eta_\ell^{\rm{int}}$, $\eta_\ell^{\rm{ext}}$, $\mathrm{osc}_\ell^{\rm{D}}$, and $\mathrm{osc}_\ell^{\rm{N}}$ to the total error estimator $\eta_\ell$ for $\theta = 0.4$.
\begin{figure}[!ht]
  \resizebox{\textwidth}{!}{
   \subfloat{
    \includegraphics{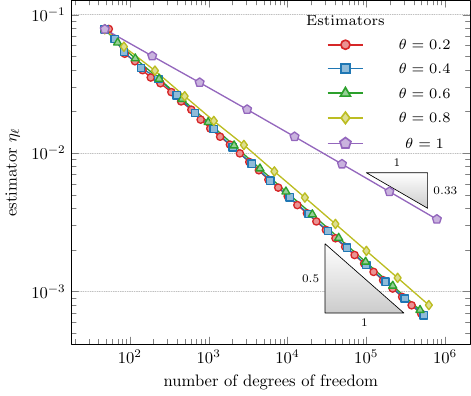}
   }
   \subfloat{
    \includegraphics{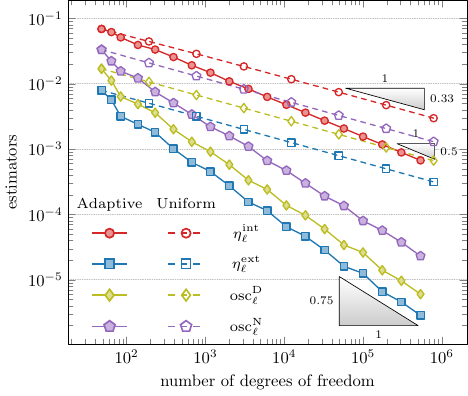}
   }
  }
   \caption{Convergence rates of the full error estimator $\eta_\ell$ from~\eqref{eq:fullestalt} in Algorithm~\ref{algo:adapint} for $k = 2$ in Example~\ref{subsec:lshapeex} for different marking parameters $0 < \theta \leq 1$ (left). Comparison of the different contributions $\eta_\ell^{\mathrm{int}}$, $\eta_\ell^{\mathrm{ext}}$, $\mathrm{osc}_\ell^{\mathrm{D}}$, and $\mathrm{osc}_\ell^{\mathrm{N}}$ of $\eta_\ell$ for $\theta = 0.4$ (right).}
   \label{fig:lrates}
 \end{figure}
Figure~\ref{fig:lintext} illustrates the comparison between $\eta_\ell^{\mathrm{ext}}$ and $\widetilde{\eta}_\ell^{\mathrm{ext}}$, i.e., Algorithm~\ref{algo:adapint} and Algorithm~\ref{algo:adap}. The observations are similar to those in Example~\ref{subsec:squareex}, i.e., both approaches lead to almost identical results, while the computation in the exterior domain is considerably more expensive. 
\begin{figure}[!ht]
  \resizebox{\textwidth}{!}{
    \includegraphics[scale = 0.1]{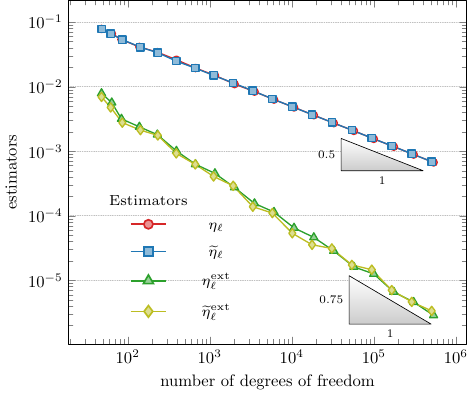}
  }
   \caption{Comparison between the estimators $\eta_\ell, \eta_\ell^{\mathrm{ext}}$ from~\eqref{eq:intindicators}--\eqref{eq:fullestalt} in Algorithm~\ref{algo:adapint} and $\widetilde{\eta}_\ell$, $\widetilde{\eta}_\ell^{\mathrm{ext}}$ from~\eqref{eq:indicators}--\eqref{eq:fullest} in Algorithm~\ref{algo:adap} for Example~\ref{subsec:lshapeex} with $k = 2$ and $\theta = 0.4$.}
   \label{fig:lintext}
 \end{figure}
The full error estimator $\eta_\ell$ and its contributions $\eta_\ell^{\mathrm{int}}$, $\eta_\ell^{\mathrm{ext}}$, $\mathrm{osc}_\ell^{\mathrm{D}}$, and $\mathrm{osc}_\ell^{\mathrm{N}}$ for adaptive ($\theta = 0.4$) refinement with respect to the Galerkin approximations of the Johnson--Nédélec coupling (left) and the Bielak--MacCamy coupling (right) are shown in Figure~\ref{fig:ljnbm}. Analogously to the symmetric Costabel--Han coupling, we observe that adaptive mesh refinement leads to optimal convergence also for these couplings.

\begin{figure}[!ht]
  \resizebox{\textwidth}{!}{
   \subfloat{
    \includegraphics{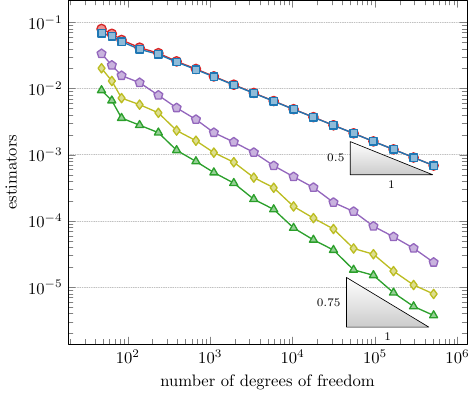}
   }
   \subfloat{
    \includegraphics{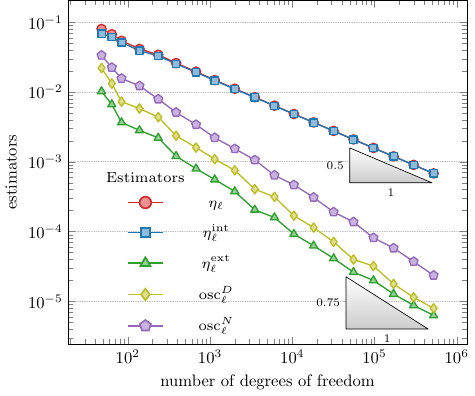}
   }
  }
   \caption{Comparison of the different contributions $\eta_\ell^{\mathrm{int}}$, $\eta_\ell^{\mathrm{ext}}$, $\mathrm{osc}_\ell^{\mathrm{D}}$, and $\mathrm{osc}_\ell^{\mathrm{N}}$ of $\eta_\ell$ in Algorithm~\ref{algo:adapint} for Example~\ref{subsec:lshapeex} with $\theta = 0.4$ for the Johnson--Nédélec coupling (left) and the Bielak--MacCamy coupling (right).}
   \label{fig:ljnbm}
 \end{figure}

\subsection{Example 2 (Z-shaped domain, nonlinear diffusion)} \label{subsec:zshapeex}

We consider the transmission problem~\eqref{eq:nltransmission} with nonlinear diffusion $\AAA (x,y) \coloneqq \mu(\abs{(x,y)}) (x,y)$, where $\mu(t) \coloneqq 2 + 1/(1+t)$. We note that the monotonicity constant of $\AAA$ is $\const{C}{mon} = 2$, while the Lipschitz constant is $\const{C}{Lip} = 4$. We prescribe the exact solution by 
\begin{equation*}
 \begin{split}
 u(x,y) 
 = r^{4/7} \sin(4\varphi/7)
 \quad \text{and} \quad
  u^{\mathrm{ext}}(x,y)
  = \frac{x + y + 1/4}{(x + 1/8)^2 + (y + 1/8)^2},
 \end{split}
\end{equation*} 
where $(r,\varphi)$ are the polar coordinates centered at the reentrant corner of the Z-shaped domain 
\begin{equation*}
 \Omega \coloneqq
 (-1/4,1/4)^2 \setminus \set{(x,y) \in \R^2 \given 0 \leq x \leq 1/4 \text{ and } 1/4 \leq y \leq x};
\end{equation*}
see Figure~\ref{fig:zmeshesext} and Figure~\ref{fig:zmeshesint}.
The solution $u$ exhibits a singularity at the reentrant corner $(0,0)$. Figure~\ref{fig:zmeshesext} and Figure~\ref{fig:zmeshesint} show the initial mesh and some adaptively refined meshes obtained by Algorithm~\ref{algo:adap} and Algorithm~\ref{algo:adapint}, respectively, with $k = 2$ and $\theta = 0.4$, which show that these are able to resolve the singularity adequately.

\begin{figure}[!ht]
 \resizebox{\textwidth}{!}{
   \subfloat{
    \stackunder{\stackunder[10pt]{\includegraphics{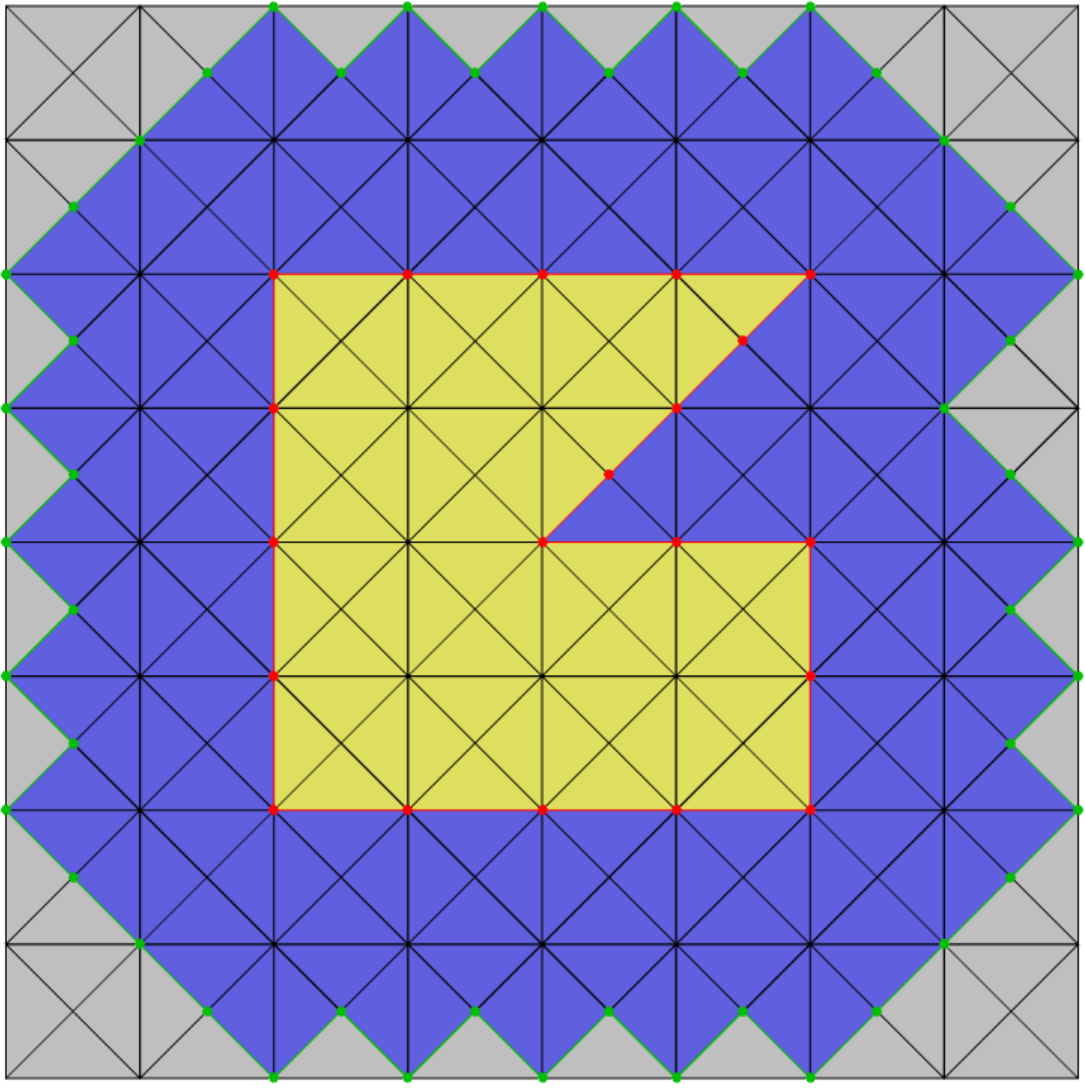}}{{\Huge$\# \TT_\ell = 56$, $\#\widehat{\TT_\ell} = 256$}}}{{\Huge $\ell=0$}}
   }
   \subfloat{
    \stackunder{\stackunder[10pt]{\includegraphics{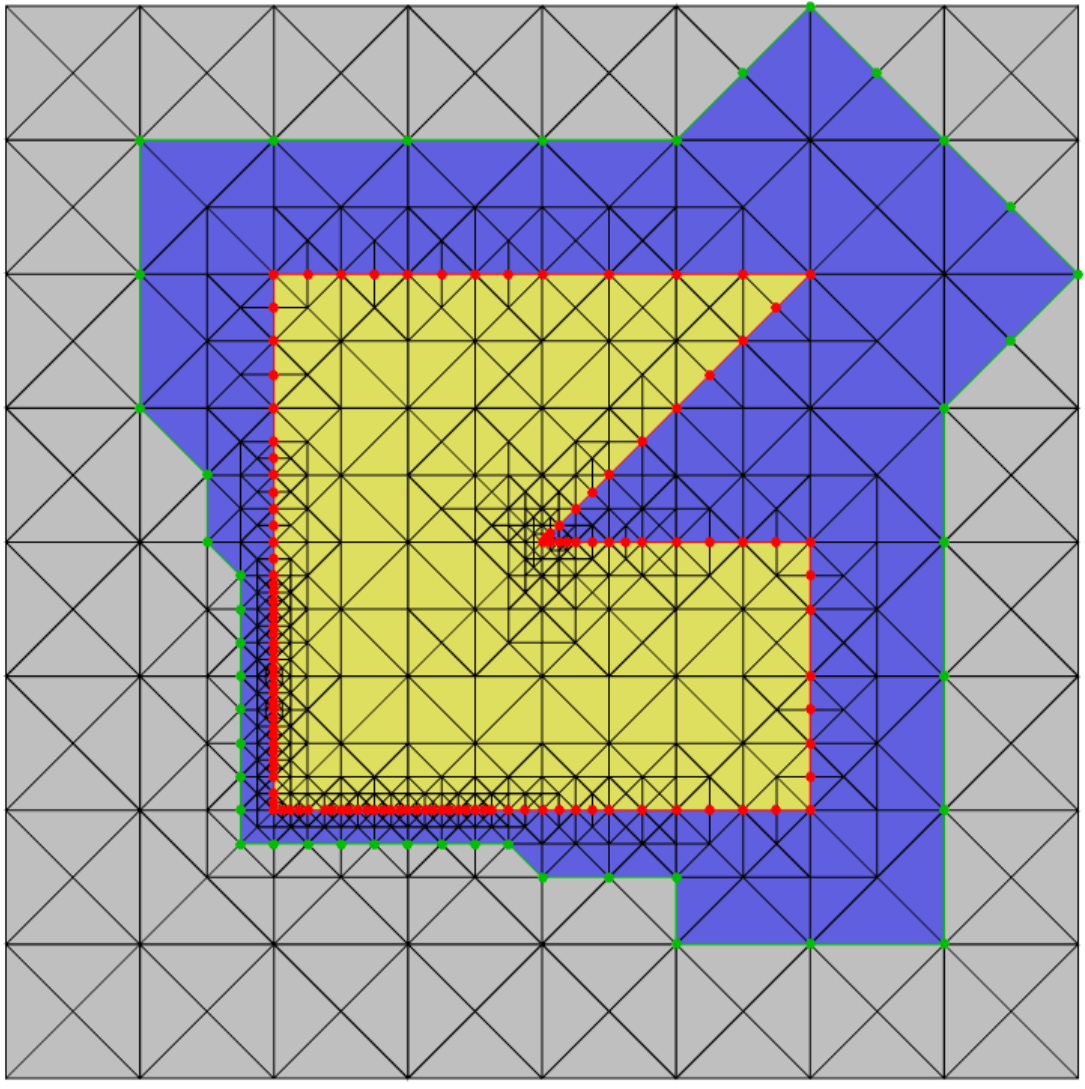}}{{\Huge $\# \TT_\ell = 541$, $\#\widehat{\TT_\ell} = 1132$}}}{{\Huge $\ell=20$}}
   }
   
   \subfloat{
    \stackunder{\stackunder[10pt]{\includegraphics{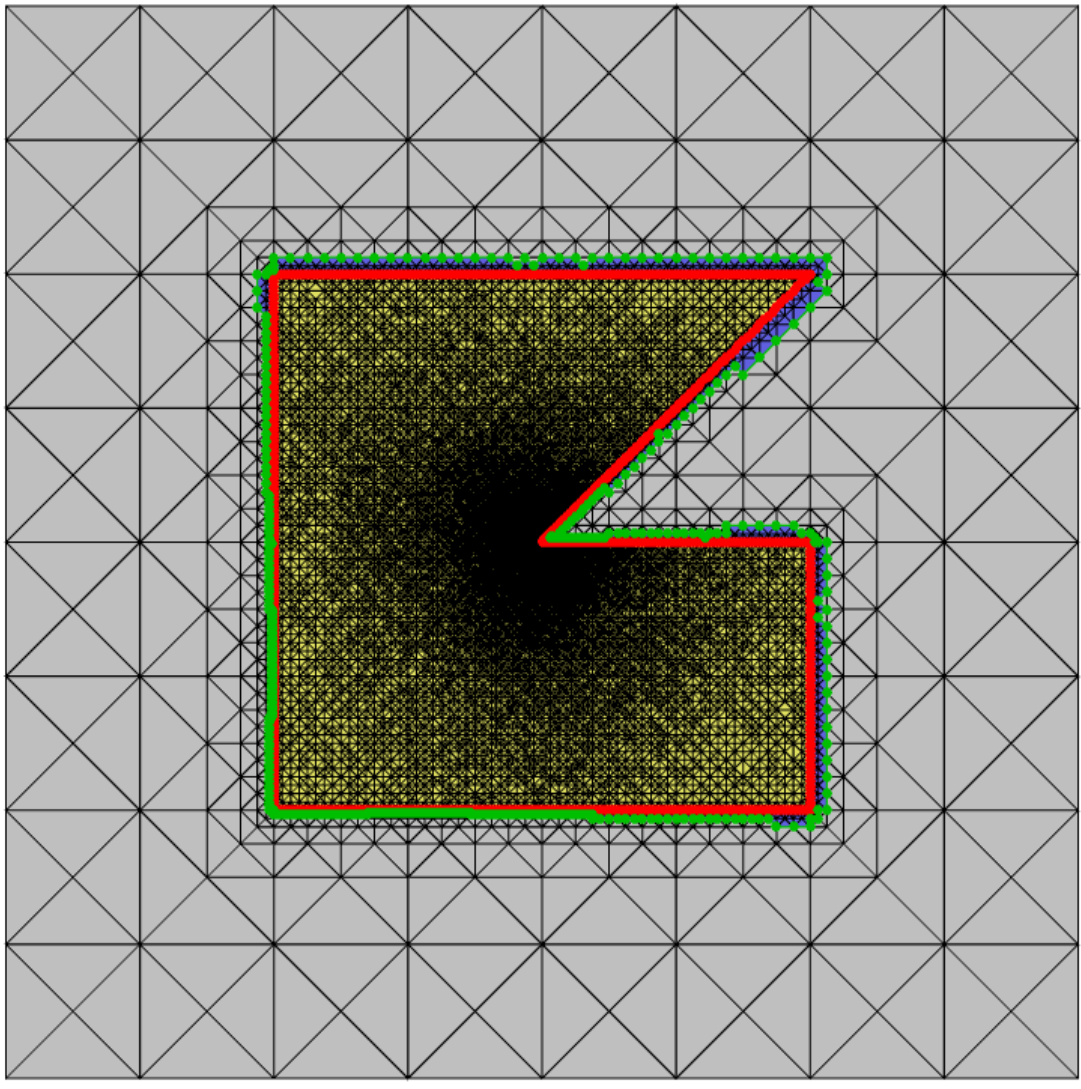}}{{\Huge $\# \TT_\ell = 37848$, $\#\widehat{\TT_\ell} = 43996$}}}{{\Huge $\ell=40$}}
   }
  }
   \caption{Meshes generated by Algorithm~\ref{algo:adap} in Example~\ref{subsec:zshapeex} for $\theta = 0.4$ and $k = 2$. Only triangles not colored in gray contribute to the computation of the error indicators~\eqref{eq:indicators}. The triangles of the patches $\widetilde{\omega}_\ell^{z,k}$ are depicted in {\color{blue} blue}, whereas the triangles in $\Omega$ are indicated in {\color{yellow} yellow} and the remaining triangles in $\Omegaext$ in {\color{gray} gray}. The inner boundary $\Gamma$ is shown in {\color{red} red} and the outer boundary of the union of the patches in {\color{green} green}.}
    \label{fig:zmeshesext}
\end{figure}

\begin{figure}[!ht]
 \resizebox{\textwidth}{!}{
   \subfloat{
    \stackunder[10pt]{\includegraphics{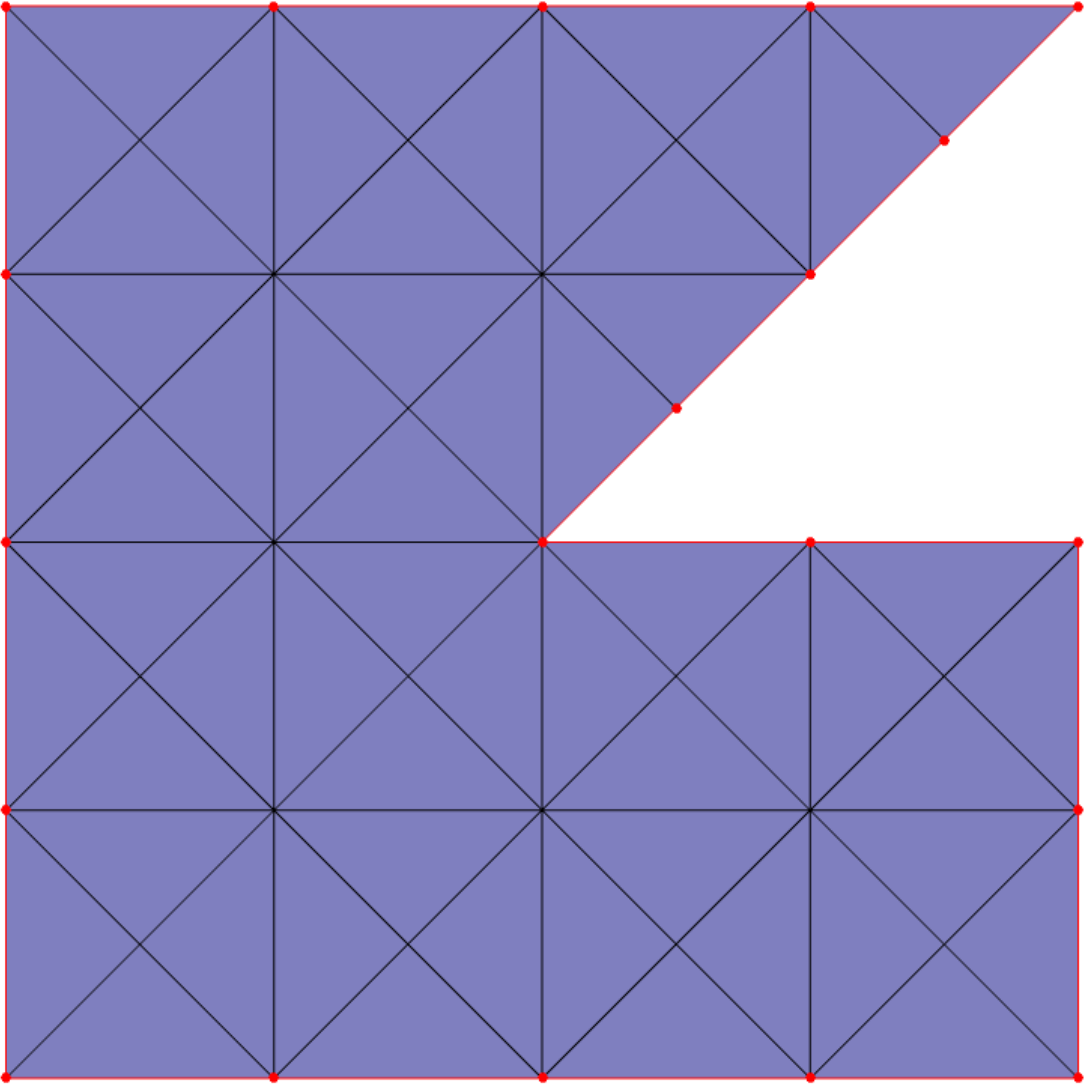}}{{\Huge$\# \TT_\ell = 56$, $\ell = 0$}}
   }
   \subfloat{
    \stackunder[10pt]{\includegraphics{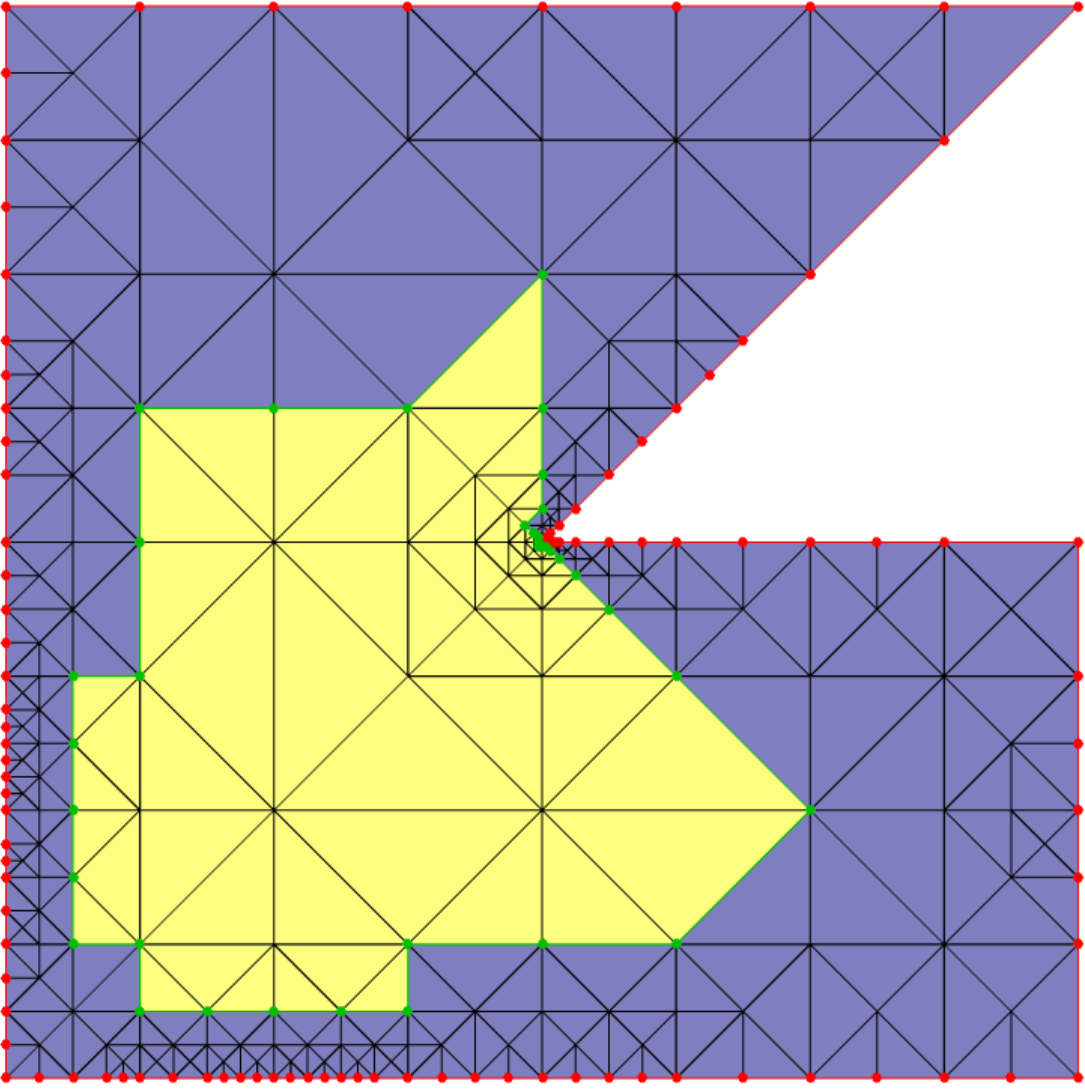}}{{\Huge $\# \TT_\ell = 609$, $\ell = 20$}}
   }
   
   \subfloat{
   \stackunder[10pt]{\includegraphics{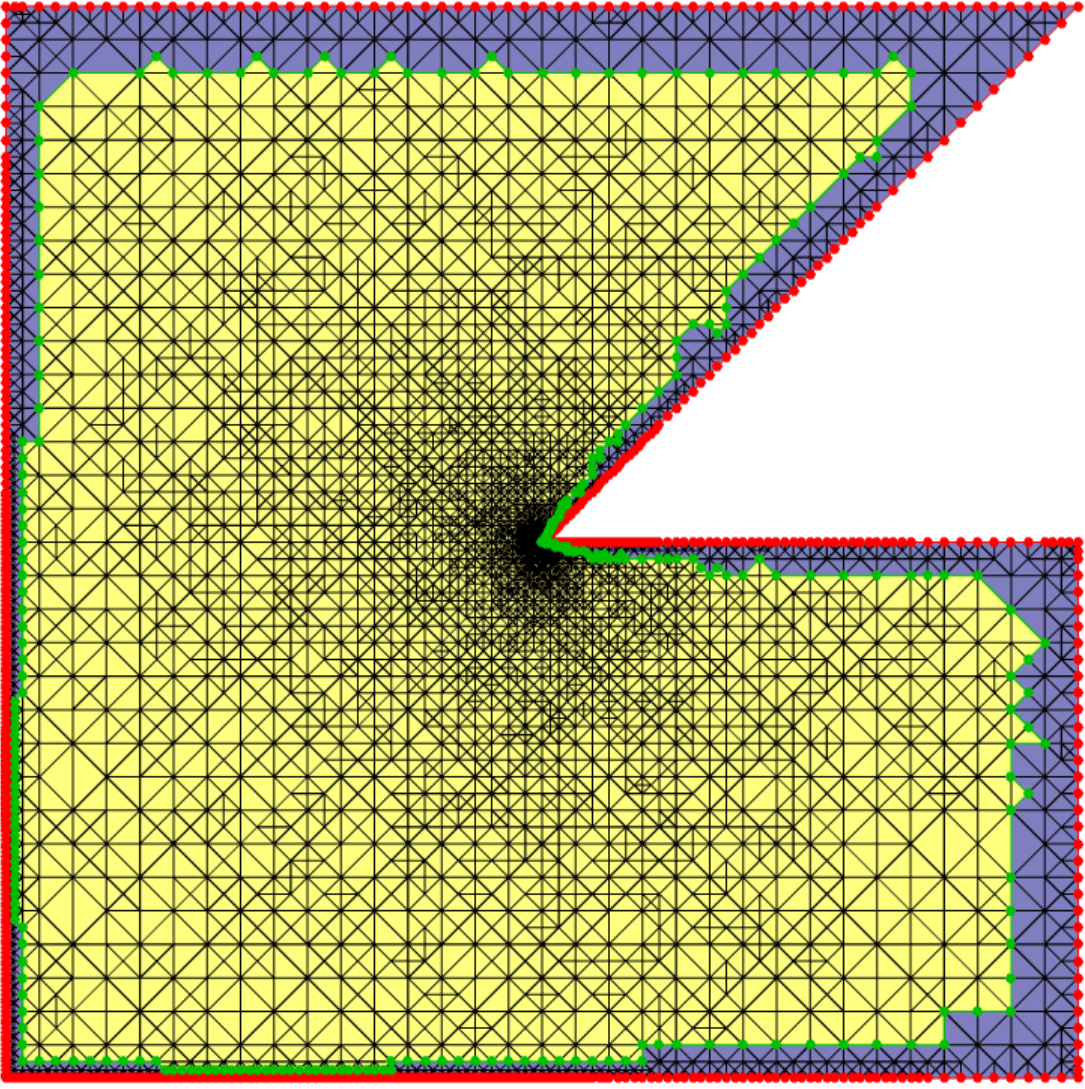}}{{\Huge $\# \TT_\ell = 52424$, $\ell = 40$}}
   }
  }
   \caption{Meshes generated by Algorithm~\ref{algo:adapint} in Example~\ref{subsec:zshapeex} for $\theta = 0.4$ and $k = 2$. The triangles of the patches $\omega_\ell^{z,k}$ are depicted in {\color{blue} blue}, while the remaining triangles are indicated in {\color{yellow} yellow}. The outer boundary $\Gamma$ is shown in {\color{red} red} and the inner boundary of the union of the patches in {\color{green} green}.}
    \label{fig:zmeshesint}
\end{figure}
The total upper bound $\eta_\ell$ from~\eqref{eq:fullestalt} generated by Algorithm~\ref{algo:adapint} with $k = 2$ and different marking parameters $0 < \theta \leq 1$ is shown in Figure~\ref{fig:zrates}~(left). We again observe that Algorithm~\ref{algo:adapint} exhibits optimal convergence rates for any marking parameter $\theta \in \set{0.2,0.4,0.6,0.8}$, whereas uniform refinement ($\theta = 1$) leads to suboptimal convergence.

\begin{figure}[!ht]
  \resizebox{\textwidth}{!}{
   \subfloat{
    \includegraphics{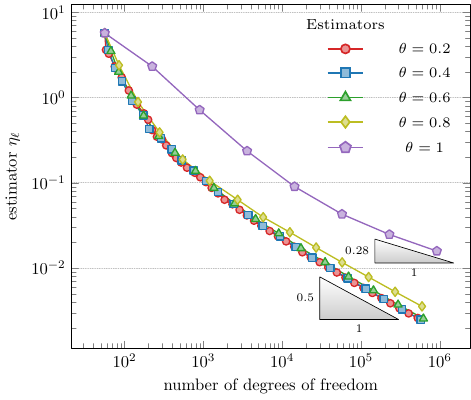}
   }
   \subfloat{
    \includegraphics{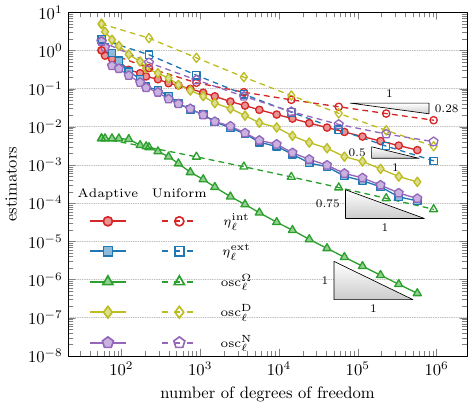}
   }
  }
   \caption{Convergence rates of the full error estimator $\eta_\ell$ from~\eqref{eq:fullestalt} in Algorithm~\ref{algo:adapint} for $k = 2$ in Example~\ref{subsec:zshapeex} for different marking parameters $0 < \theta \leq 1$ (left). Comparison of the different contributions $\eta_\ell^{\mathrm{int}}$, $\eta_\ell^{\mathrm{ext}}$, $\mathrm{osc}_\ell^\Omega$, $\mathrm{osc}_\ell^{\mathrm{D}}$, and $\mathrm{osc}_\ell^{\mathrm{N}}$ of $\eta_\ell$ for $\theta = 0.4$ (right).}
   \label{fig:zrates}
 \end{figure}
Figure~\ref{fig:zrates}~(right) depicts the different contributions $\eta_\ell^{\rm{int}}$, $\eta_\ell^{\rm{ext}}$, $\mathrm{osc}_\ell^{\Omega}$, $\mathrm{osc}_\ell^{\rm{D}}$, and $\mathrm{osc}_\ell^{\rm{N}}$ to the total error estimator $\eta_\ell$ for $\theta = 0.4$.

Figure~\ref{fig:zintext} illustrates the comparison between Algorithm~\ref{algo:adap} and Algorithm~\ref{algo:adapint}. The observations are similar to those in the previous examples, i.e., both approaches lead to almost identical results, while the computation in the exterior domain is more expensive. 

\begin{figure}[!ht]
  \resizebox{\textwidth}{!}{
    \includegraphics[scale = 0.1]{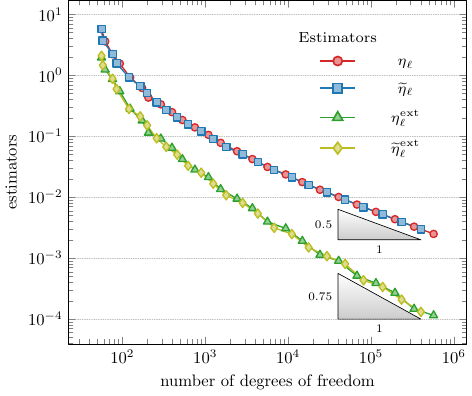}
  }
   \caption{Comparison between the estimators $\eta_\ell, \eta_\ell^{\mathrm{ext}}$ from~\eqref{eq:intindicators}--\eqref{eq:fullestalt} in Algorithm~\ref{algo:adapint} and $\widetilde{\eta}_\ell$, $\widetilde{\eta}_\ell^{\mathrm{ext}}$ from~\eqref{eq:indicators}--\eqref{eq:fullest} in Algorithm~\ref{algo:adap} for Example~\ref{subsec:zshapeex} with $k = 2$ and $\theta = 0.4$.}
   \label{fig:zintext}
 \end{figure}

\begin{figure}[!ht]
  \resizebox{\textwidth}{!}{
   \subfloat{
    \includegraphics{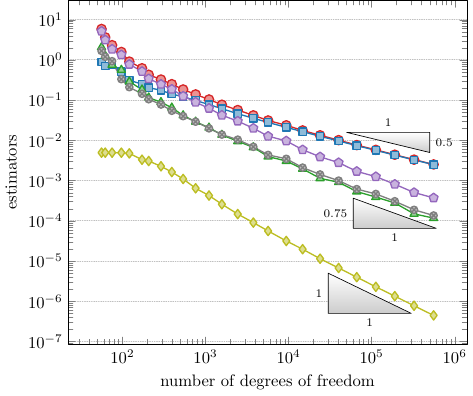}
   }
   \subfloat{
    \includegraphics{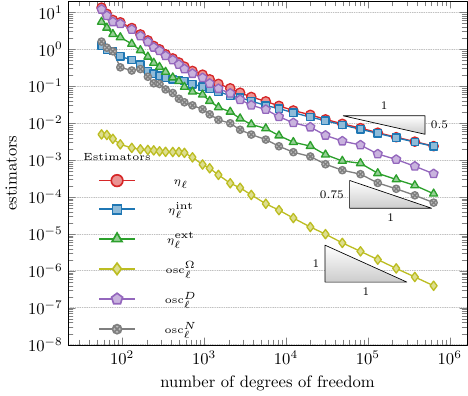}
   }
  }
\caption{Comparison of the different contributions $\eta_\ell^{\mathrm{int}}$, $\eta_\ell^{\mathrm{ext}}$, $\mathrm{osc}_\ell^{\mathrm{D}}$, and $\mathrm{osc}_\ell^{\mathrm{N}}$ of $\eta_\ell$ in Algorithm~\ref{algo:adapint} for Example~\ref{subsec:zshapeex} with $\theta = 0.4$ for the Johnson--Nédélec coupling (left) and the Bielak--MacCamy coupling (right).}
   \label{fig:zjnbm}
 \end{figure}

The total error estimator $\eta_\ell$ and its contributions $\eta_\ell^{\mathrm{int}}$, $\eta_\ell^{\mathrm{ext}}$, $\mathrm{osc}_\ell^\Omega$, $\mathrm{osc}_\ell^{\mathrm{D}}$, and $\mathrm{osc}_\ell^{\mathrm{N}}$ for adaptive ($\theta = 0.4$) and uniform ($\theta = 1$) refinement with respect to the Galerkin approximations of the Johnson--Nédélec coupling (left) and the Bielak--MacCamy coupling (right) are shown in Figure~\ref{fig:zjnbm}. We observe that for adaptive mesh refinement, Algorithm~\ref{algo:adapint} exhibits optimal convergence rates also for these couplings.

\printbibliography
\end{document}